\documentclass[12pt,twoside]{article}
\usepackage[latin1]{inputenc}
\usepackage[english]{babel}
\usepackage{clm-macros}
\usepackage{nishad-macros}
\usepackage{amsmath}
\usepackage{psfrag}
\usepackage{epsfig}
\usepackage{graphicx}
\usepackage{verbatim}
\usepackage{tikz}
\usepackage{datetime}

\usepackage{enumerate}

\usetikzlibrary{decorations}
\usetikzlibrary{decorations.pathmorphing}
\usetikzlibrary{arrows}
\tikzstyle{every node}=[circle, draw, fill=white,inner sep=0pt, minimum width=4pt]
\tikzstyle{nodelabel}=[rounded corners,fill=none,inner sep=5pt,draw=none]
\tikzstyle{matching} = [ultra thick]
\tikzset{snake/.style={decorate, decoration=snake}}

\newdateformat{UKvardate}{
\THEDAY\ \monthname[\THEMONTH], \THEYEAR}

\def\Notname{Notation}
\def\casname{Case}
\def\cnjname{Conjecture}
\def\corname{Corollary}
\def\Defname{Definition}
\def\exename{Exercise}
\def\exaname{Example}
\def\lemname{Lemma}
\def\listassertionname{List of Assertions}
\def\prbname{Problem}
\def\prpname{Proposition}
\def\proofOfname{\proofname{} of}
\def\ssbname{Case}
\def\subname{Case}
\def\thmname{Theorem}

\begin{document}

\title{Generating Near-Bipartite Bricks
\footnote{A shorter version has been accepted for publication in the Journal of Graph Theory.}}
\author{Nishad Kothari\footnote{\noindent Partially supported by NSERC grant (RGPIN-2014-04351, J. Cheriyan).}
\\
nishadkothari@gmail.com
\\% \newline nishadkothari@gmail.com}\\ \\
Dept. of Combinatorics and Optimization, U. of Waterloo}
\UKvardate
\date{20 November, 2016}
  \maketitle 
  \thispagestyle{empty}

\begin{abstract} 

A $3$-connected graph $G$ is a {\it brick} if, for any two vertices $u$~and~$v$, the graph $G-\{u,v\}$ has a perfect matching.
Deleting an edge $e$ from a brick~$G$ results in a graph with zero, one or two vertices of degree two. The {\it bicontraction} of a vertex
of degree two consists of contracting the two edges incident with it; and the {\it retract} of~$G-e$ is the graph~$J$ obtained from it by bicontracting
all its vertices of degree two. An edge $e$ is {\it thin} if $J$ is also a brick. Carvalho, Lucchesi and Murty
[How to build a brick, Discrete Mathematics 306 (2006), 2383-2410] showed that every brick, distinct from $K_4$,
the triangular prism~$\overline{C_6}$
and the Petersen graph, has a thin edge. Their theorem yields a generation procedure for bricks, using which they showed
that every simple planar solid brick is an odd wheel.

\smallskip
A brick~$G$ is {\it \nb} if it has a pair of edges $\alpha$ and $\beta$ such that $G-\{\alpha,\beta\}$ is bipartite and matching covered;
examples are $K_4$~and~$\overline{C_6}$.
The significance of \nb\ graphs arises from the theory of ear decompositions of matching covered graphs.

\smallskip
The object of this paper is to establish a generation procedure which is specific to the class of \nb\ bricks.
In particular, we prove that if $G$ is any \nb\ brick, distinct from $K_4$~and~$\overline{C_6}$, then $G$ has a thin edge~$e$
such that the retract~$J$ of~$G-e$ is also \nb.

\end{abstract}

\newpage
\tableofcontents

\section{Matching Covered Graphs}
\label{sec:mcg}

For general graph theoretic notation and terminology,
we refer the reader to Bondy and Murty \cite{bomu08}.
All graphs considered here are loopless;
however, we allow multiple edges.
An edge of a graph
is {\it admissible} if there is a perfect matching of the
graph that contains it.
A connected graph with two or more vertices
is {\it matching covered} if each of its edges is admissible.
For a comprehensive treatment of matching theory
and its origins, we refer the reader to Lov{\'a}sz and
Plummer \cite{lopl86}, wherein matching covered graphs
are referred to as `1-extendable' graphs.

\smallskip
In this section, we briefly review the relevant
terminology, definitions and results from the theory of matching
covered graphs.

\subsection{Canonical Partition}

Tutte's Theorem states that a graph $G$ has a perfect
matching if and only~if ${\sf odd}(G-S) \leq |S|$
for each subset $S$ of $G$, where ${\sf odd}(G-S)$
denotes the number of odd components of $G-S$.
For a graph $G$ that has a perfect matching,
a nonempty subset $S$ of its vertices is a {\it barrier}
if it satisfies the equality ${\sf odd}(G-S)=|S|$.
The following proposition is easily deduced from Tutte's Theorem,
and yields a characterization of matching covered graphs.

\begin{prop}\label{prop:delete-two-vertices-has-pm}
Let $G$ be a graph that has a perfect matching. Let $u$~and~$v$
be distinct vertices of~$G$. Then the graph $G-\{u,v\}$ has a perfect
matching if and only if there is no barrier of $G$ which contains
both $u$~and~$v$.
\end{prop}

\begin{cor}
\label{cor:characterization-mcg}
Let $G$ be a connected graph with a perfect matching.
Then $G$ is matching covered if and only if every barrier of~$G$
is stable (that is, an independent set).
\end{cor}

The following fundamental theorem is due to Kotzig (see \cite[page 150]{lopl86}).

\begin{thm}
{\sc [The Canonical Partition Theorem]}
\label{thm:canonical-partition}
The maximal barriers of a matching covered graph~$G$ partition its vertex set.
\end{thm}

For a matching covered graph~$G$, the partition of its vertex set defined by
its maximal barriers is called the {\it canonical partition} of~$V(G)$.
For instance, for a bipartite matching covered graph~$H[A,B]$, the canonical
partition of~$V(H)$ consists of precisely two parts, namely, its color classes $A$~and~$B$;
this is implied by the following proposition which may be derived from the well-known
Hall's Theorem. (The neighbourhood of a set of vertices~$S$ is denoted by $N(S)$.)

\begin{prop}
\label{prop:characterizations-of-bipmcg}
Let $H[A,B]$ denote a bipartite graph with four or more vertices, where $|A|=|B|$.
Then the following statements are equivalent:
\begin{enumerate}[(i)]
\item $H$ is matching covered,
\item $|N(S)| \geq |S|+1$ for every nonempty proper subset $S$ of~$A$, and
\item $H-\{a,b\}$ has a perfect matching for each pair of vertices $a \in A$ and $b \in B$.
\end{enumerate}
\end{prop}

A graph $G$, with four or more vertices, is {\it bicritical} if $G-\{u,v\}$
has a perfect matching for every pair of distinct vertices $u$~and~$v$.
A barrier is {\it trivial} if it has a single vertex.
Proposition~\ref{prop:delete-two-vertices-has-pm}
implies the following characterization of bicritical graphs.

\begin{prop}
\label{prop:characterization-bicritical}
Let $G$ be a connected graph with a perfect matching.
Then $G$ is bicritical if and only if every barrier of~$G$
is trivial.
\end{prop}

Thus, for a bicritical graph~$G$, the canonical partition of $V(G)$ consists
of~$|V(G)|$ parts, each of which contains a single vertex.

\subsection{Bricks and Braces}

For a nonempty proper subset~$X$ of the vertices of a graph~$G$, we denote by $\partial(X)$
the cut associated with~$X$, that is, the set of all edges of~$G$ that have one end in~$X$
and the other end in~$\overline{X}:=V(G)-X$.
We refer to $X$~and~$\overline{X}$ as the
{\it shores} of~$\partial(X)$. A cut is {\it trivial} if any of its shores is a singleton. For
a cut $\partial(X)$, we denote the graph obtained by contracting the shore~$\overline{X}$
to a single vertex~$\overline{x}$ by $G/(\overline{X} \rightarrow \overline{x})$.
In case the label of the contraction vertex~$\overline{x}$ is irrelevant, we simply write $G/ \overline{X}$.
The two graphs $G/X$ and $G/ \overline{X}$ are called the
\mbox{{\it $\partial(X)$-contractions}} of~$G$.

\smallskip
Let $G$ be a matching covered graph. A cut $\partial(X)$ is a {\it tight cut}
if \mbox{$|M \cap \partial(X)|=1$} for every perfect matching $M$ of~$G$.
It is easily
verified that if $\partial(X)$ is a nontrivial tight cut of~$G$, then each
$\partial(X)$-contraction is a matching covered graph that has strictly fewer
vertices than~$G$.
If either of the $\partial(X)$-contractions has a nontrivial tight
cut, then that graph can be further decomposed into even smaller matching covered
graphs. We can repeat this procedure until we obtain a list of matching covered
graphs, each of which is free of nontrivial tight cuts. This procedure is known as
a {\it tight cut decomposition} of~$G$.

\smallskip
Let $G$ be a matching covered graph free of nontrivial tight cuts. If $G$
is bipartite then it is a {\it brace}; otherwise it is a {\it brick}.
Thus, a tight
cut decomposition of~$G$ results in a list of bricks and braces.
In general, a matching covered graph may admit several tight cut
decompositions.
However, Lov{\'a}sz \cite{lova87} proved
the following remarkable result, and demonstrated its significance
by using it to compute the dimension of the matching lattice.

\begin{thm}
{\sc [The Unique Decomposition Theorem]}
\label{thm:lovasz-tight-cut-decomposition}
Any two tight cut decompositions of a matching covered graph yield the same list
of bricks and braces (except possibly for multiplicities of edges).
\end{thm}

In particular, any two tight cut decompositions of a matching covered graph~$G$ yield
the same number of bricks; this number is denoted by $b(G)$. We remark that $G$
is bipartite if and only if $b(G)=0$.

\smallskip
Let $G$ be a matching covered graph.
Observe that, if $S$ is a barrier of~$G$, and $K$ is an odd component of~$G-S$,
then $\partial(V(K))$ is a tight cut of~$G$. Such a tight cut is called a {\it barrier cut}.
(For instance, if $v$ is a vertex of degree two then $\{v\} \cup N(v)$ is the shore of a barrier cut.)
In particular, if $G$ is nonbipartite then each nontrivial barrier gives rise to a nontrivial
tight cut.

\smallskip
Now suppose that $\{u,v\}$ is a $2$-vertex-cut of~$G$ such that $G-\{u,v\}$ has an even component,
say~$K$. Then each of the sets $V(K) \cup \{u\}$ and $V(K) \cup \{v\}$ is a shore of a nontrivial
tight cut of~$G$. Such a tight cut is called a $2$-separation cut.
(We remark that a graph may have a tight cut which is neither a barrier cut nor a $2$-separation cut.)

\smallskip
Since a brick is a nonbipartite matching covered graph which is free of nontrivial tight cuts,
it follows from the above observations that every brick is $3$-connected and bicritical.
Edmonds, Lov{\'a}sz and Pulleyblank \cite{elp82} established the converse.

\begin{thm}
\label{thm:elp-bricks}
A graph $G$ is a brick if and only if it is $3$-connected and bicritical.
\end{thm}

In particular, a brick is free of nontrivial barriers and of $2$-vertex-cuts.
Three cubic bricks, namely $K_4$, $\overline{C_6}$ and the Petersen graph,
play a special role in the theory of matching covered graphs.

\subsection{Removable edges}

\smallskip
An edge~$e$ of a matching covered graph~$G$ is {\it removable} if $G-e$ is also matching covered;
otherwise it is {\it non-removable}.
For example, each edge of the Petersen graph is removable.
The following was established by Lov{\'a}sz~\cite{lova87}.

\begin{thm}
{\sc [Removable Edge Theorem]}
\label{thm:lovasz-removable-bricks}
Every brick distinct from $K_4$ and $\overline{C_6}$ has a removable edge.
\end{thm}

We point out that, if $e$ is a removable edge of a brick~$G$, then $G-e$ may not be a brick.
For instance, $G-e$ may have vertices of degree two.

\subsubsection{Near-bricks and \binv\ edges}
\label{sec:near-bricks-binvariant-edges}

Recall that $b(G)$ denotes the number of bricks of a matching covered graph~$G$ (in any tight cut decomposition),
and it is well-defined due to
the Unique Decomposition Theorem~(\ref{thm:lovasz-tight-cut-decomposition}).
A {\it \nbrick} is a matching covered graph with $b(G) = 1$.
Clearly, every brick is a \nbrick. However, the converse is not true.
When proving theorems concerning bricks,
one often needs the flexibility of dealing with the wider class
of {\nbrick}s, whose properties are akin to those of bricks.

\smallskip
A removable edge~$e$ of a matching covered graph~$G$ is
{\it \binv} if \mbox{$b(G-e) = b(G)$}.
In particular, if $G$ is a brick then $e$ is \binv\ if and only if $G-e$ is a \nbrick.
For instance, the graph $St_8$ shown in Figure~\ref{fig:St_8} has a unique
$b$-invariant edge $e$.

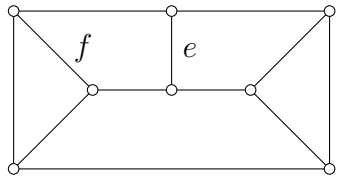
\begin{figure}[!ht]
\centering
%\bigskip
\begin{tikzpicture}[scale=1.4]
%\tikzstyle{every node}=[circle, draw, fill=white,inner sep=0pt, minimum width=5pt]
\draw (0,0) -- (0.75,0.75) -- (0,1.5) -- (0,0);
\draw (3,0) -- (2.25,0.75) -- (3,1.5) -- (3,0);
\draw (0,0)node{} -- (3,0)node{};
\draw (0.75,0.75)node{} -- (2.25,0.75)node{};
\draw (0,1.5)node{} -- (3,1.5)node{};
\draw (1.5,0.75) -- (1.5,1.5);
\draw (1.4,1.125)node[right,nodelabel]{$e$};
\draw (0.32,1.15)node[right,nodelabel]{$f$};
\draw (1.5,0.75)node{};
\draw (1.5,1.5)node{};
%\draw (0.35,0.35)node[right,nodelabel]{$\alpha$};
%\draw (2.9,0.75)node[right,nodelabel]{$\beta$};
%\draw (0.1,0.75)node[left,nodelabel]{$\alpha'$};
%\draw (2.65,0.35)node[left,nodelabel]{$\beta'$};
\end{tikzpicture}
\caption{$St_8$ has a unique $b$-invariant edge $e$}
\label{fig:St_8}
%\bigskip
\end{figure}

\smallskip
It is easily verified that if $G$ is the Petersen graph
and $e$ is any edge, then $b(G-e) = 2$.
Thus each edge of the Petersen graph is removable, but none of them is \binv.
Confirming a conjecture of Lov{\'a}sz, the following result was proved by Carvalho, Lucchesi and Murty \cite{clm02a}.

\begin{thm}
{\sc [\binv\ Edge Theorem]}
\label{thm:clm-binvariant-bricks}
Every brick distinct from $K_4$ and $\overline{C_6}$ and the \mbox{Petersen}
graph has a \binv\ edge.
\end{thm}

\subsubsection{Bicontractions, retracts and bi-splittings}
\label{sec:bicontraction-retract-bisplitting}

Let $G$ be a matching covered graph and let $v$ be a vertex of degree two, with two distinct neighbours
$u$ and $w$. The {\it bicontraction} of~$v$ is the operation of contracting the two edges $vu$ and $vw$
incident with~$v$. Note that $X:=\{u,v,w\}$ is the shore of a tight cut of~$G$, and that the graph resulting
from the bicontraction of~$v$ is the same as the \mbox{$\partial(X)$-contraction $G/X$}, whereas the
other $\partial(X)$-contraction $G/\overline{X}$ is isomorphic to~$C_4$ (possibly with multiple edges).

\smallskip
The {\it retract} of~$G$ is the graph obtained from~$G$ by bicontracting all its degree two vertices.
The above observation implies that the retract of a matching covered
graph is also matching covered.
Carvalho et al. \cite{clm05} showed that the retract of a matching covered graph
is unique up to isomorphism.
It is important to note that even if $G$ is simple, the retract of~$G$ may have multiple edges.

\smallskip
The operation of bi-splitting is the converse of the operation of
bicontraction. Let $H$ be a graph and let $v$ be a vertex of $H$ of
degree at least two. Let $G$ be a graph obtained from $H$ by
replacing the vertex $v$ by two new vertices $v_1$~and~$v_2$,
distributing the edges in $H$ incident with $v$ between $v_1$~and~$v_2$
such that each gets at least one, and then adding a new vertex~$v_0$
and joining it to both $v_1$~and~$v_2$. Then we say that $G$ is
obtained from $H$ by {\it bi-splitting} $v$ into
$v_1$~and~$v_2$. It is easily seen that if $H$ is matching covered,
then $G$ is also matching covered,
and that $H$ can be recovered from $G$ by bicontracting the vertex $v_0$
and denoting the contraction vertex by $v$.

\subsubsection{Thin edges}
\label{sec:thin-edges}

A \binv\ edge~$e$ of a brick~$G$ is {\it thin} if the retract of~$G-e$ is a brick.
As the graph~$G-e$ can have zero, one or two vertices of degree two,
the retract of~$G-e$ is obtained by performing at most two bicontractions, and it has
at least $|V(G)|-4$ vertices.
For example, the retract of $St_8 - e$ (see Figure~\ref{fig:St_8}) is isomorphic to~$K_4$ with
multiple edges; thus, $e$ is a thin edge.
It should be noted that, in general, a \binv\ edge may not be thin.

\smallskip
The original definition of a thin edge, due to Carvalho et al. \cite{clm06}, was in terms
of barriers; `thin' being a reference to the fact that the barriers of~$G-e$ are sparse.
This viewpoint will also be useful to us in latter sections (where further explanation
is provided).
Carvalho, Lucchesi and Murty \cite{clm06} used their
\binv\ Edge Theorem (\ref{thm:clm-binvariant-bricks})
to derive the following stronger result.

\begin{thm}
{\sc [Thin Edge Theorem]}
\label{thm:clm-thin-bricks}
Every brick distinct from $K_4$ and $\overline{C_6}$ and the \mbox{Petersen} graph has a thin edge.
\end{thm}

The following is an immediate consequence of the above theorem.

\begin{thm}
\label{thm:clm-brick-reduction}
{\rm \cite{clm06}}
Given any brick~$G$, there exists a sequence $G_1, G_2, \dots, G_k$ of bricks such that:
\begin{enumerate}[(i)]
\item $G_1$ is either $K_4$ or $\overline{C_6}$ or the Petersen graph,
\item $G_k := G$, and
\item for $2 \leq i \leq k$, there exists a thin edge~$e_i$ of~$G_i$ such that $G_{i-1}$ is the retract of~$G_i - e_i$.
\end{enumerate}
\end{thm}

Carvalho et al. \cite{clm06} also described four elementary `expansion operations'
which may be applied to any brick to
obtain a larger brick with at most four more vertices.
Each of these operations consists of bi-splitting at most two vertices and then adding a suitable edge.
Given a brick~$J$, the application of any of these four operations to~$J$ results in a brick~$G$ such that $G$ has
a thin edge~$e$ with the property that $J$ is the retract of~$G-e$. Thus, any brick may be generated from
one of the three basic bricks ($K_4$ and $\overline{C_6}$ and the Petersen graph) by means of these four
expansion operations.

\subsection{Near-Bipartite Bricks}

A nonbipartite matching covered graph~$G$ is
{\it \nb} if it has a pair $R:=\{\alpha,\beta\}$
of edges such that the graph~$H:=G-R$ is bipartite
and matching covered.
Such a pair $R$ is called a {\it removable doubleton}.

\smallskip
Furthermore, if $G$ happens to be a brick,
we say that $G$ is a {\it near-bipartite brick}.
For instance, $K_4$~and~$\overline{C_6}$ are the
smallest simple \nb\ bricks, and each of
them has three distinct removable doubletons.

\smallskip
Observe that
the edge $\alpha$ joins two vertices in one color class
of~$H$, and that $\beta$ joins two vertices in the other
color class. Consequently, if $M$ is any perfect matching
of~$G$ then $\alpha \in M$ if and only~if $\beta \in M$.
(In particular, neither $\alpha$ nor $\beta$ is a removable edge of~$G$.)
The following is an immediate consequence of \cite[Theorem~5.1]{clm02b}.

\begin{thm}\label{thm:clm-nb-nbrick}
Every \nb\ graph is a \nbrick.
\end{thm}

The significance of \nb\ graphs arises from the
theory of ear decompositions of matching covered graphs;
see \cite{clm99} and \cite{koth16};
in this context, \nb\ graphs constitute the
class of nonbipartite matching covered graphs which are
`closest' to being bipartite. Thus, certain problems which are
rather difficult to solve for general nonbipartite graphs
are easier to solve for the special case of \nb\ graphs;
for instance, although there
has been no significant progress in characterizing
Pfaffian nonbipartite graphs, Fischer and Little \cite{fili01}
were able to characterize Pfaffian \nb\ graphs.

\smallskip
The difficulty in using Theorem~\ref{thm:clm-brick-reduction}
as an induction tool for studying \nb\ bricks,
is that even if $G_k := G$ is a \nb\ brick,
there is no guarantee that all of the intermediate bricks
$G_1, G_2, \dots G_{k-1}$ are also \nb.
For instance, the brick shown in
Figure~\ref{fig:double-biwheel-of-typeI}a
is \nb\ with a (unique)
removable doubleton~\mbox{$R:=\{\alpha,\beta\}$}.
Although the edge~$e$ is thin;
the retract of~$G-e$, as shown in
Figure~\ref{fig:double-biwheel-of-typeI}b,
is not \nb\ since it has three edge-disjoint triangles.
%(Note that each edge indicated by a bold line is inadmissible in~$(G-e)-R$,
%whence $R$ is not a removable doubleton of~$G-e$.)

\begin{figure}[!ht]
\centering
%\bigskip
\begin{tikzpicture}[scale=0.85]
%\tikzstyle{every node}=[circle, draw, fill=white,inner sep=0pt, minimum width=5pt]
\draw (0,0) to [out=300,in=180] (4,-2.5) to [out=0,in=240] (8,0);

\draw (0,0) -- (1,0);
\draw (1,0) -- (2,0);
%\draw[dashed] (1,0) -- (2,0);
\draw (2,0) -- (3,0);
\draw (5,0) -- (6,0);
\draw (6,0) -- (7,0);
%\draw[dashed] (6,0) -- (7,0);
\draw (7,0) -- (8,0);

\draw[ultra thick] (0,0) -- (4,1.5);
\draw (0,0)node[fill=black]{};
\draw (1,0) -- (4,-1.5);
%\draw[dashed] (1,0) -- (4,-1.5);
\draw (1,0)node{};
\draw (2,0) -- (4,1.5);
%\draw[dashed] (2,0) -- (4,1.5);
\draw (2,0)node[fill=black]{};
%\draw (3,0) -- (4,-1.5);
\draw[ultra thick] (3,0) -- (4,-1.5);

%\draw (5,0) -- (4,-1.5);
\draw[ultra thick] (5,0) -- (4,-1.5);
\draw (6,0) -- (4,1.5);
%\draw[dashed] (6,0) -- (4,1.5);
\draw (6,0)node[fill=black]{};
\draw (7,0) -- (4,-1.5);
%\draw[dashed] (7,0) -- (4,-1.5);
\draw (7,0)node{};
\draw (4,-1.5)node[fill=black]{};
%\draw (8,0) -- (4,1.5);
\draw[ultra thick] (8,0) -- (4,1.5);
\draw (8,0)node[fill=black]{};
\draw (4,1.5)node{};

\draw (3,0)node{} -- (5,0)node{};

\draw (4,-0.14)node[above,nodelabel]{$\alpha$};
\draw (4,-2.24)node[nodelabel]{$\beta$};

\draw (1.75,0.88)node[nodelabel]{$e$};
\draw (4,-3.2)node[nodelabel]{(a)};
\end{tikzpicture}
\hspace*{0.3in}
\begin{tikzpicture}[scale=0.85]
%\tikzstyle{every node}=[circle, draw, fill=white,inner sep=0pt, minimum width=5pt]
\draw (2,0) to [out=300,in=170] (4,-1.8) to [out=350,in=210] (8,0);
\draw (8,0) -- (4,-1.5);

\draw (2,0) -- (3,0);
\draw (5,0) -- (6,0);
\draw (6,0) -- (7,0);
\draw (7,0) -- (8,0);

\draw (2,0) -- (4,1.5);
\draw (2,0)node{};
\draw (3,0) -- (4,-1.5);

\draw (5,0) -- (4,-1.5);
\draw (6,0) -- (4,1.5);
\draw (6,0)node{};
\draw (7,0) -- (4,-1.5);
\draw (7,0)node{};
\draw (4,-1.5)node{};
\draw (8,0) -- (4,1.5);
\draw (8,0)node{};
\draw (4,1.5)node{};

\draw (3,0) -- (5,0);
\draw (3,0)node{};
\draw (5,0)node{};
\draw (5,-3.2)node[nodelabel]{(b)};
\end{tikzpicture}
\vspace*{-0.3in}
\caption{(a) A \nb\ brick~$G$ with a thin edge~$e$ ;
(b) The retract of~$G-e$ is not \nb}
\label{fig:double-biwheel-of-typeI}
%\bigskip
\end{figure}
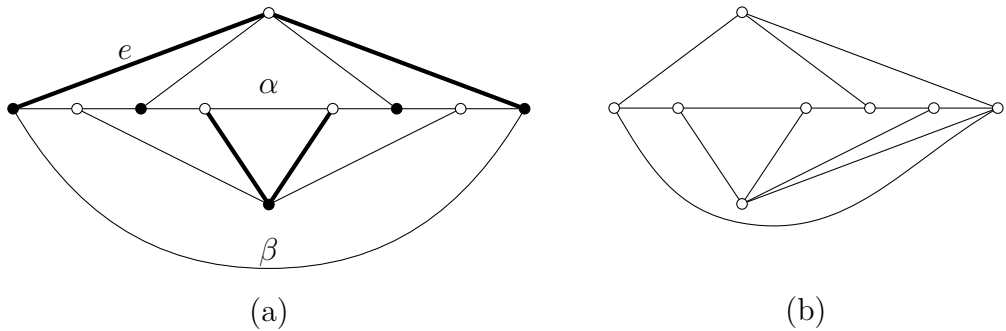

\smallskip
In other words, deleting an arbitrary thin edge
may not preserve the property of being \nb. In this sense,
the Thin Edge Theorem~(\ref{thm:clm-thin-bricks})
is inadequate for obtaining inductive proofs of results
that pertain only to the class of \nb\ bricks.

\smallskip
To fix this problem, we decided to look for a thin edge whose
deletion preserves the property of being \nb.
Our main result is as follows.

\begin{thm}\label{thm:main}
Every \nb\ brick~$G$ distinct from $K_4$ and $\overline{C_6}$
has a thin edge $e$ such that the retract of~$G-e$ is also \nb.
\end{thm}

In fact, we prove a stronger theorem.
In particular, we find it convenient to fix a removable doubleton~$R$ (of the brick under consideration),
and then look for a thin edge whose deletion preserves this removable doubleton.
To make this precise, we will first define a special type of removable edge which
we call `\Rcomp'.

\subsubsection{\Rcomp\ edges}
\label{sec:Rcompatible-edges}

We use the abbreviation {\it \Rgraph} for a \nb\ graph~$G$
with (fixed) removable doubleton~$R$, and we
shall refer to $H:=G-R$ as its {\it underlying bipartite graph}.
In the same spirit,
an {\it \Rbrick} is a brick with a removable doubleton~$R$.

\smallskip
A removable edge~$e$ of an \Rgraph~$G$ is {\it \Rcomp} if it is
removable in~$H$ as well.
Equivalently, an edge~$e$ is \Rcomp\ if $G-e$ and $H-e$ 
are both matching covered.
For instance, the graph~$St_8$ (see Figure~\ref{fig:St_8_nb})
has two removable doubletons \mbox{$R:=\{\alpha,\beta\}$} and
$R' := \{\alpha', \beta'\}$, and its unique removable edge~$e$ is
\Rcomp\ as well as \comp{R'}.

\begin{figure}[!ht]
\centering
%\bigskip
\begin{tikzpicture}[scale=1.4]
%\tikzstyle{every node}=[circle, draw, fill=white,inner sep=0pt, minimum width=5pt]
\draw (0,0) -- (0.75,0.75) -- (0,1.5) -- (0,0);
\draw (3,0) -- (2.25,0.75) -- (3,1.5) -- (3,0);
\draw (0,0)node{} -- (3,0)node[fill=black]{};
\draw (0.75,0.75)node{} -- (2.25,0.75)node{};
\draw (0,1.5)node[fill=black]{} -- (3,1.5)node[fill=black]{};
\draw (1.5,0.75) -- (1.5,1.5);
\draw (1.4,1.125)node[right,nodelabel]{$e$};
\draw (1.5,0.75)node[fill=black]{};
\draw (1.5,1.5)node{};
\draw (0.35,0.35)node[right,nodelabel]{$\alpha$};
\draw (2.9,0.75)node[right,nodelabel]{$\beta$};
\draw (0.1,0.75)node[left,nodelabel]{$\alpha'$};
\draw (2.65,0.35)node[left,nodelabel]{$\beta'$};
\end{tikzpicture}
\caption{$e$ is \Rcomp\ as well as \comp{R'}}
\label{fig:St_8_nb}
%\bigskip
\end{figure}
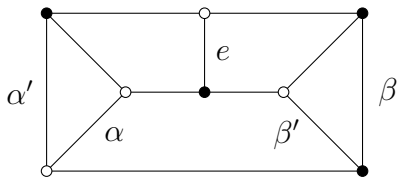

\smallskip
Now, let $G$ denote the \Rbrick\ shown in
Figure~\ref{fig:double-biwheel-of-typeI}a,
where $R:=\{\alpha,\beta\}$.
The thin edge~$e$ is incident with an edge of~$R$
at a cubic vertex; consequently, $H-e$ has a vertex
whose degree is only one, and so it is not matching covered.
In particular, $e$ is not \Rcomp.

\smallskip
The brick shown in Figure~\ref{fig:pseudo-biwheel}
has two distinct removable doubletons $R:=\{\alpha,\beta\}$
and $R':=\{\alpha',\beta'\}$. Its edges $e$~and~$f$
are both \comp{R'}, but neither of them is \Rcomp.

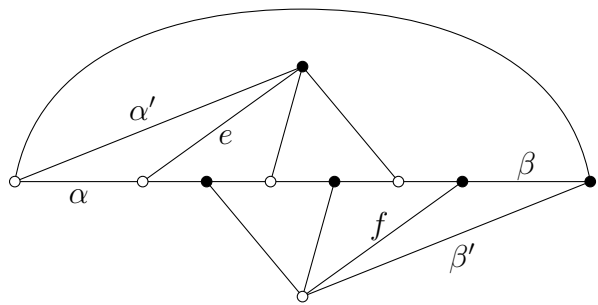
\begin{figure}[!ht]
%\bigskip
\centering
\begin{tikzpicture}[scale=0.85]

%\draw (4,0.7)node[nodelabel]{$a_1$};
%\draw (9,1.3)node[nodelabel]{$b_1$};
%\draw (6.5,3.1)node[nodelabel]{$u_1$};
%\draw (6.5,-1.1)node[nodelabel]{$w_1$};
%\draw (2,1)node[left,nodelabel]{$a_2$};
%\draw (11,1)node[right,nodelabel]{$b_2$};

\draw (5.3,1.7)node[nodelabel]{$e$};
\draw (7.7,0.3)node[nodelabel]{$f$};

\draw (3,0.77)node[nodelabel]{$\alpha$};
\draw (10,1.25)node[,nodelabel]{$\beta$};
\draw (4,2.15)node[nodelabel]{$\alpha'$};
\draw (9,-0.2)node[nodelabel]{$\beta'$};

\draw (2,1) to [out=80,in=180] (6.5,3.7) to [out=0,in=100] (11,1);

\draw (2,1) -- (4,1);
\draw (9,1) -- (11,1);
\draw (6.5,2.8) -- (2,1);
\draw (6.5,-0.8) -- (11,1);

\draw (2,1)node{};
\draw (11,1)node[fill=black]{};

\draw (4,1) -- (9,1);

\draw (4,1)node{} -- (6.5,2.8);
\draw (5,1)node[fill=black]{} -- (6.5,-0.8);
\draw (6,1)node{} -- (6.5,2.8);
\draw (7,1)node[fill=black]{} -- (6.5,-0.8);
\draw (8,1)node{} -- (6.5,2.8)node[fill=black]{};
\draw (9,1)node[fill=black]{} -- (6.5,-0.8)node{};

\end{tikzpicture}
%\vspace*{-0.1in}
\caption{$e$ and $f$ are \comp{R'},
but they are not \Rcomp}
\label{fig:pseudo-biwheel}
%\bigskip
\end{figure}

\smallskip
Observe that, if $e$ is an \Rcomp\ edge of an \Rgraph~$G$,
then $R$ is a removable doubleton of~$G-e$,
whence $G-e$ is also an \Rgraph; in particular,
$G-e$ is \nb.
By Theorem~\ref{thm:clm-nb-nbrick},
$G-e$ is a \nbrick; and this proves the following.

\begin{prop}
\label{prop:Rcompatible-is-binvariant}
Every \Rcomp\ edge is \binv. \qed
\end{prop}

Furthermore, as we will see later, if $e$ is an \Rcomp\ edge of an
\Rbrick~$G$ then the unique brick~$J$ of~$G-e$ is also an \Rbrick;
in particular, $J$ is \nb.
The following is a special case of a
theorem of Carvalho, Lucchesi and Murty \cite{clm99}.

\begin{thm}
{\sc [\Rcomp\ Edge Theorem]}
\label{thm:clm-Rcompatible-nb-bricks}
Every \Rbrick\ distinct from $K_4$~and~$\overline{C_6}$ has an \Rcomp\ edge.
\end{thm}

In \cite{clm99},
they proved a stronger result. In particular,
they showed the existence of an \Rcomp\ edge in {\Rgraph}s
with minimum degree at least three.
(They did not use the term `\Rcomp'.)
Using the notion of $R$-compatibility, we
now define a thin edge whose deletion preserves
the property of being \nb.

\subsubsection{\Rthin\ edges}
\label{sec:Rthin-edges}

A thin edge~$e$ of an \Rbrick~$G$
is {\it \Rthin} if it is \Rcomp.
Equivalently, an edge~$e$ is \Rthin\ if it is \Rcomp\ as well as thin,
and in this case, the retract of~$G-e$ is also an \Rbrick.

\smallskip
As noted earlier, the graph $St_8$, shown in Figure~\ref{fig:St_8_nb},
has two removable doubletons $R$~and~$R'$. Its unique removable edge~$e$
is \Rthin\ as well as \thin{R'}; to see this, note that
the retract~$J$ of~$St_8-e$ is isomorphic to $K_4$ with multiple edges,
and each of $R$ and $R'$ is a removable doubleton of~$J$.

\smallskip
Using the \Rcomp\ Edge Theorem (\ref{thm:clm-Rcompatible-nb-bricks})
of Carvalho et al.,
we prove the following stronger result
(which immediately implies Theorem~\ref{thm:main}).

%\begin{leftbar}
\begin{thm}
{\sc [\Rthin\ Edge Theorem]}
\label{thm:Rthin-nb-bricks}
Every \Rbrick\ distinct from $K_4$~and~$\overline{C_6}$ has an \Rthin\ edge.
\end{thm}
%\end{leftbar}

Our proof of the above theorem uses tools from the work of Carvalho et al. \cite{clm06},
and the overall approach is inspired by their proof of the
Thin Edge Theorem (\ref{thm:clm-thin-bricks}).
The following is an immediate consequence of Theorem~\ref{thm:Rthin-nb-bricks}.

\begin{thm}
\label{thm:nb-brick-reduction}
Given any \Rbrick~$G$,
there exists a sequence $G_1, G_2, \dots, G_k$ of {\Rbrick}s such that:
\begin{enumerate}[(i)]
\item $G_1$ is either $K_4$ or $\overline{C_6}$,
\item $G_k:=G$, and
\item for $2 \leq i \leq k$, there exists an \Rthin\ edge~$e_i$ of~$G_i$
such that $G_{i-1}$ is the retract of~$G_i - e_i$.
\end{enumerate}
\end{thm}

It follows from the above theorem
that every \nb\ brick can be generated
from one of $K_4$~and~$\overline{C_6}$
by means of the expansion operations.
Theorem~\ref{thm:Rthin-nb-bricks} and its proof also appear in the Ph.D. thesis of Kothari \cite{koth16}.

\section{Near-Bipartite Graphs}
\label{sec:nbg}

In this section, we examine properties of \nb\ graphs that are relevant
to our proof of Theorem~\ref{thm:Rthin-nb-bricks}.
Recall that an $R$-graph~$G$ is a \nb\ graph with a fixed
removable doubleton~$R$. We adopt the following notational conventions.

\begin{Not}
\label{Not:Rgraph-removable-doubleton}
For an $R$-graph~$G$, we shall denote by
$H[A,B]$ the underlying bipartite graph~$G-R$.
We let $\alpha$ and $\beta$ denote the constituent edges of~$R$,
and we adopt the convention that
$\alpha:= a_1a_2$ has both ends in~$A$,
whereas $\beta:=b_1b_2$ has both ends in~$B$.
\end{Not}

As we will see, certain pertinent properties of~$G$ are closely related to those of~$H$.
For this reason, we also review well-known facts concerning
bipartite matching covered graphs.

\subsection{The exchange property}

Recall that an edge of a matching covered graph is removable if its deletion results in another matching
covered graph. The removable edges of a bipartite graph satisfy an
`exchange property' and its proof immediately follows from
Proposition~\ref{prop:characterizations-of-bipmcg}.

\begin{prop}
{\sc [Exchange Property of Removable Edges]}
\label{prop:exchange-property-removable-bipmcg}
Let $H$ denote a bipartite matching covered graph, and
let $e$ denote a removable edge of~$H$. If $f$
is a removable edge of~$H-e$, then:
\begin{enumerate}[(i)]
\item $f$ is removable in~$H$, and
\item $e$ is removable in~$H-f$. \qed
\end{enumerate}
\end{prop}

\begin{comment}
\begin{proof}
Observe that the graph~$H-f$ may be obtained from the matching covered graph $H-e-f$
by adding a single ear (that is, edge~$e$).
Thus, by Proposition~\ref{prop:bipartite-ear-decomposition},
$H-f$ is matching covered. This proves {\it (i)}.
Statement {\it (ii)} follows immediately since $H-f-e$ is matching covered.
\end{proof}
\end{comment}

We point out that the conclusion of
Proposition~\ref{prop:exchange-property-removable-bipmcg}
does not hold, in general, for arbitrary removable edges of
nonbipartite graphs.
For instance, as shown in Figure~\ref{fig:St_8},
the edge~$f$ is removable in the matching covered graph $St_8 - e$,
but it is not removable in~$St_8$. However, as we prove next,
the exchange property does hold for \Rcomp\ edges.
Recall that an \Rcomp\ edge
of an $R$-graph~$G$ is one which is removable in $G$
as well as in the underlying bipartite graph~$H:=G-R$;
see Section~\ref{sec:Rcompatible-edges}.

\begin{prop}
{\sc [Exchange Property of \Rcomp\ Edges]}
\label{prop:exchange-property-Rcompatible-nbmcg}
Let $G$ be an \mbox{$R$-graph}, and let $e$ denote an \Rcomp\ edge of~$G$.
If $f$ is an \Rcomp\ edge of~$G-e$, then:
\begin{enumerate}[(i)]
\item $f$ is \Rcomp\ in~$G$, and
\item $e$ is \Rcomp\ in~$G-f$.
\end{enumerate}
\end{prop}
\begin{proof}
Let $H:=G-R$.
Since $f$ is \Rcomp\ in~$G-e$,
each of the graphs~$G-e-f$ and $H-e-f$ is matching covered.
To deduce {\it (i)}, we need to show that each of $G-f$ and $H-f$ is matching covered.
Since $f$ is removable in~$H-e$, it follows
from Proposition~\ref{prop:exchange-property-removable-bipmcg} that $f$
is removable in~$H$ as well. That is, $H-f$ is matching covered.

\smallskip
Next, we note that the edge~$e$ is admissible in~$H-f$.
Thus $e$ is admissible in~$G-f$. As $G-e-f$ is matching covered,
we conclude that $G-f$ is also matching covered.
This proves {\it (i)}. Statement {\it (ii)} follows immediately,
since each of $G-f-e$ and $H-f-e$ is matching covered.
\end{proof}

\subsection{Non-removable edges of bipartite graphs}

Let $H[A,B]$ denote a bipartite graph, on four or more vertices,
that has a perfect matching.
Using the well-known Hall's Theorem,
it can be shown that an edge~$f$ of~$H$ is inadmissible
(that is, $f$ is not in any perfect matching of~$H$)
if and only if there exists a nonempty proper subset
$S$ of~$A$ such that $|N(S)| = |S|$ and
$f$ has one end in~$N(S)$ and its other end is not in~$S$.

\smallskip
Now suppose that $H$ is matching covered,
and let $e$ denote a non-removable edge of~$H$.
Then some edge~$f$ of $H-e$ is inadmissible.
This fact, coupled with the above observation,
may be used to arrive at the following
characterization of non-removable edges in
bipartite matching covered graphs;
see Figure~\ref{fig:non-removable-bipmcg}.

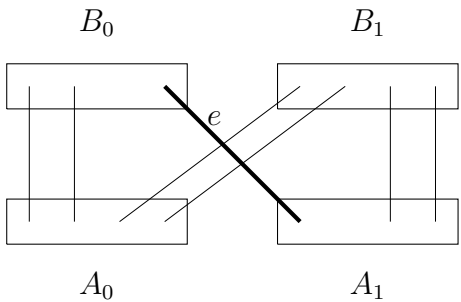
\begin{figure}[!ht]
\centering
%\medskip
\begin{tikzpicture}[scale=0.6]
\draw (0,0) -- (4,0) -- (4,1) -- (0,1) -- (0,0);
\draw (6,0) -- (10,0) -- (10,1) -- (6,1) -- (6,0);
\draw (2,0)node[below,nodelabel]{$A_0$};
\draw (8,0)node[below,nodelabel]{$A_1$};

\draw (0,3) -- (4,3) -- (4,4) -- (0,4) -- (0,3);
\draw (6,3) -- (10,3) -- (10,4) -- (6,4) -- (6,3);
\draw (2,4)node[above,nodelabel]{$B_0$};
\draw (8,4)node[above,nodelabel]{$B_1$};

\draw[ultra thick] (3.5,3.5) -- (6.5,0.5);
\draw (4.6,2.8)node[nodelabel]{$e$};

\draw[ultra thin] (3.5,0.5) -- (7.5,3.5);
\draw[ultra thin] (2.5,0.5) -- (6.5,3.5);

\draw[ultra thin] (1.5,0.5) -- (1.5,3.5);
\draw[ultra thin] (0.5,0.5) -- (0.5,3.5);
\draw[ultra thin] (8.5,0.5) -- (8.5,3.5);
\draw[ultra thin] (9.5,0.5) -- (9.5,3.5);

\end{tikzpicture}
\vspace*{-0.2in}
\caption{Non-removable edge of a bipartite graph}
\label{fig:non-removable-bipmcg}
%\bigskip
\end{figure}

\begin{prop}
{\sc [Characterization of Non-removable Edges]}
\label{prop:characterization-non-removable-bipartite}
Let $H[A,B]$ denote a bipartite matching covered graph
on four or more vertices.
An edge~$e$ of~$H$ is non-removable
if and only if there exist partitions $(A_0,A_1)$ of~$A$
and $(B_0,B_1)$ of~$B$ such that $|A_0| = |B_0|$
and $e$ is the only edge joining
a vertex in~$B_0$ to a vertex in~$A_1$. \qed
\end{prop}

In our work, we will often be interested in finding an \Rcomp\
edge incident at a specified vertex~$v$ of an $R$-brick~$G$.
As a first step, we will upper bound the number of edges of~$\partial(v)$,
which are non-removable in the underlying bipartite graph~\mbox{$H:=G-R$}.
For this purpose, the next lemma
of Lov{\'a}sz and Vempala \cite{love} is especially useful. It is an
extension of
Proposition~\ref{prop:characterization-non-removable-bipartite}.
See Figure~\ref{fig:lovasz-vempala}.

\begin{figure}[!ht]
\centering
%\smallskip
\begin{tikzpicture}[scale=0.65]
\draw[thin] (2,-0.5) -- (4,-5);
\draw[thin] (2,-0.5) -- (7.5,-5);
\draw[thin] (2,-0.5) -- (13,-5);

\draw (0,0) -- (2.5,0) -- (2.5,-1) -- (0,-1) -- (0,0);
\draw (1.25,0) node[nodelabel,above]{$B_0$};

\draw (2,-0.5)node{};
\draw (2.2,-0.5)node[left,nodelabel]{$b$};

\draw (3.5,0) -- (6,0) -- (6,-1) -- (3.5,-1) -- (3.5,0);
\draw (4.75,0) node[nodelabel,above]{$B_1$};

\draw (7,0) -- (9.5,0) -- (9.5,-1) -- (7,-1) -- (7,0);
\draw (8.25,0) node[nodelabel,above]{$B_2$};

\draw[dotted] (10,-0.5) -- (12,-0.5);

\draw (12.5,0) -- (15,0) -- (15,-1) -- (12.5,-1) -- (12.5,0);
\draw (13.75,0) node[nodelabel,above]{$B_r$};

\draw (0,-4.5) -- (2.5,-4.5) -- (2.5,-5.5) -- (0,-5.5) -- (0,-4.5);
\draw (1.25,-5.5) node[nodelabel,below]{$A_0$};

\draw (3.5,-4.5) -- (6,-4.5) -- (6,-5.5) -- (3.5,-5.5) -- (3.5,-4.5);
\draw (4.75,-5.5) node[nodelabel,below]{$A_1$};

\draw (4,-5)node{};
\draw (3.8,-5)node[right,nodelabel]{$a_1$};

\draw (7,-4.5) -- (9.5,-4.5) -- (9.5,-5.5) -- (7,-5.5) -- (7,-4.5);
\draw (8.25,-5.5) node[nodelabel,below]{$A_2$};

\draw (7.5,-5)node{};
\draw (7.3,-5)node[right,nodelabel]{$a_2$};

\draw[dotted] (10,-5) -- (12,-5);

\draw (12.5,-4.5) -- (15,-4.5) -- (15,-5.5) -- (12.5,-5.5) -- (12.5,-4.5);
\draw (13.75,-5.5) node[nodelabel,below]{$A_r$};

\draw (13,-5)node{};
\draw (12.8,-5)node[right,nodelabel]{$a_r$};

\end{tikzpicture}
\vspace*{-0.1in}
\caption{Non-removable edges incident at a vertex}
\label{fig:lovasz-vempala}
%\bigskip
\end{figure}
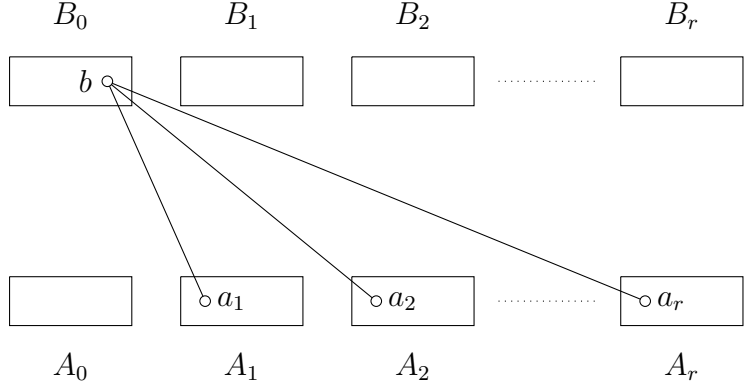

\begin{lem}
{\sc [The Lov{\'a}sz-Vempala Lemma]}
\label{lem:lovasz-vempala}
Let $H[A,B]$ denote a bipartite matching covered graph, and $b \in B$ denote a vertex
of degree~$d \geq 3$. Let $ba_1, ba_2, \dots, ba_d$ be the edges of~$H$
incident with~$b$. Assume that the edges $ba_1, ba_2, \dots, ba_r$ where $0< r \leq d$
are \mbox{non-removable}. Then there exist partitions $(A_0, A_1, \dots, A_r)$ of~$A$
and $(B_0,B_1, \dots, B_r)$ of~$B$, such that $b \in B_0$, and for $i \in \{1,2, \dots, r\}$:
(i) $|A_i| = |B_i|$, (ii) $a_i \in A_i$, and (iii)~$N(A_i) = B_i \cup \{b\}$;
in particular, $ba_i$ is the only edge between $B_0$ and $A_i$. \qed
\end{lem}

Observe that, as per the notation in the above lemma, if $ba_1$ and $ba_2$ are
\mbox{non-removable} edges, then the vertices $a_1$ and $a_2$ have no common neighbour
distinct from~$b$.
That is, there is no $4$-cycle containing edges $ba_1$ and $ba_2$ both.
This proves the following corollary of Lov{\'a}sz and Vempala \cite{love}.

\begin{cor}
\label{cor:quadrilateral-LV}
Let $H$ denote a bipartite matching covered graph, and $b$ denote a vertex
of degree three or more. If $e$ and $f$ are two edges incident at~$b$ which
lie in a $4$-cycle $Q$ then at least one of $e$ and $f$ is removable. \qed
\end{cor}

We conclude with an easy application of the Lov{\'a}sz-Vempala Lemma
in the context of \nb\ bricks.

\begin{cor}
\label{cor:application-of-LV}
Let $G$ be an $R$-brick, and let $H:=G-R$. Then for any
vertex~$b$, at most two edges of $\partial_H(b)$ are
non-removable in~$H$.
\end{cor}
\begin{proof}
We adopt Notation~\ref{Not:Rgraph-removable-doubleton};
assume without loss of generality that $b \in B$.
If $b$ has only two distinct neighbours in~$H$ then the assertion
is easily verified.
Now suppose that $b$ has at least three distinct neighbours in~$H$, and let
$d$ denote the degree of~$b$ in~$H$.

\smallskip
Suppose instead that there are $r \geq 3$ non-removable edges incident
with~$b$; we denote these as $ba_1, ba_2, \dots, ba_r$.
Then, by the Lov{\'a}sz-Vempala Lemma (\ref{lem:lovasz-vempala}),
there exist partitions $(A_0, A_1, \dots, A_r)$ of~$A$ and
$(B_0,B_1, \dots, B_r)$ of~$B$, such that $b \in B_0$, and
for $i \in \{1,2,\dots,r\}$: (i) $|A_i| = |B_i|$, (ii)~$a_i \in A_i$, and
(iii) $N_H(A_i) = B_i \cup \{b\}$. See Figure~\ref{fig:lovasz-vempala}.

\smallskip
Observe that, for $i \in \{1,2,\dots,r\}$,
every vertex of~$A_i$ is isolated in $H-(B_i \cup \{b\})$;
consequently, $B_i \cup \{b\}$ is a nontrivial barrier of~$H$.
Since $G$ is free of nontrivial barriers (by Theorem~\ref{thm:elp-bricks}),
adding the edges of~$R$
must kill each of these barriers.
In particular, $\alpha$ must have an end in each~$A_i$
for $i \in \{1,2,\dots,r\}$.
This is not possible, as $r \geq 3$; thus we have a contradiction.
This completes the proof of Corollary~\ref{cor:application-of-LV}.
\end{proof}

\subsection{Barriers and tight cuts}
\label{sec:barriers-tight-cuts}

We begin with a property of removable edges related to tight cuts
which is easily verified;
it holds for all matching covered graphs.

\begin{prop}
\label{prop:removable-edges-across-tight-cuts}
Let $G$ be a matching covered graph, and $\partial(X)$ a tight cut of~$G$,
and $e$ an edge of~$G[X]$. Then $e$ is removable in~$G/ \overline{X}$
if and only if~$e$ is removable in~$G$. \qed
\end{prop}

Let us revisit the notion of a barrier cut. If $S$ is a barrier of a
matching covered graph~$G$ and
$K$ is an odd component of~$G-S$ then $\partial(V(K))$ is a tight cut
of~$G$, and is referred to as a barrier cut.
In Sections \ref{sec:barriers-tight-cuts-bipmcg}
and \ref{sec:barriers-tight-cuts-nbmcg}, among other
things, we will see that every nontrivial tight cut of a bipartite or
of a \nb\ graph is a barrier cut.

\subsubsection{Bipartite graphs}
\label{sec:barriers-tight-cuts-bipmcg}

Suppose that $X$ is an odd subset of the
vertex set of a bipartite graph~$H[A,B]$.
Then, clearly one of the two sets $A \cap X$ and $B \cap X$
is larger than the other;
the larger of the two sets, denoted~$X_+$,
is called the {\it majority part of~$X$}; and
the smaller set, denoted~$X_-$,
is called the {\it minority part of~$X$}.

\smallskip
The following proposition is easily derived, and it provides a convenient way of visualizing tight cuts in bipartite
matching covered graphs. See Figure~\ref{fig:tight-cut-bipmcg}.

\begin{prop}
{\sc [Tight Cuts in Bipartite Graphs]}
\label{prop:tight-cut-bipmcg}
A cut~$\partial(X)$ of a bipartite matching covered graph~$H$ is tight if and only if the following hold:
\begin{enumerate}[(i)]
\item $|X|$ is odd and $|X_+| = |X_-| + 1$, consequently $|\overline{X}_+| = |\overline{X}_-| + 1$, and
\item there are no edges between~$X_-$ and $\overline{X}_-$. \qed
\end{enumerate}
\end{prop}

\begin{figure}[!ht]
\centering
\begin{tikzpicture}[scale=0.6]
%\draw (-2,-0.5) node[nodelabel]{$U$};
%\draw (-2,-3) node[nodelabel]{$W$};

\draw (0,0) -- (4,0) -- (4,-1) -- (0,-1) -- (0,0);
\draw (2,0) node[above,nodelabel]{$\overline{X}_+$};

\draw (0,-2.5) -- (3.2,-2.5) -- (3.2,-3.5) -- (0,-3.5) -- (0,-2.5);
\draw (1.6,-3.5) node[below,nodelabel]{$\overline{X}_-$};

\draw (5.2,-2.5) -- (9.2,-2.5) -- (9.2,-3.5) -- (5.2,-3.5) -- (5.2,-2.5);
\draw (7.2,-3.5) node[below,nodelabel]{$X_+$};

\draw (6,0) -- (9.2,0) -- (9.2,-1) -- (6,-1) -- (6,0);
\draw (7.6,0) node[above,nodelabel]{$X_-$};

\draw[ultra thin] (1,-0.5) -- (1,-3);
\draw[ultra thin] (1.7,-0.5) -- (1.7,-3);
\draw[ultra thin] (2.4,-0.5) -- (6.1,-3);
\draw[ultra thin] (3.1,-0.5) -- (6.8,-3);
\draw[ultra thin] (7.5,-3) -- (7.5,-0.5);
\draw[ultra thin] (8.2,-3) -- (8.2,-0.5);

\end{tikzpicture}
\vspace*{-0.2in}
\caption{Tight cuts in bipartite matching covered graphs}
\label{fig:tight-cut-bipmcg}
%\bigskip
\end{figure}
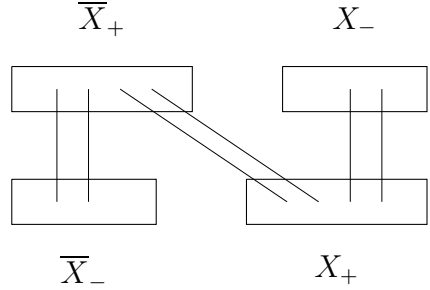

Observe that, in the above proposition,
$X_+$ and $\overline{X}_+$ are both barriers of~$H$.
It follows that every tight cut of a bipartite matching covered graph is a barrier cut.

\smallskip
Recall that, for a bipartite matching covered graph~$H[A,B]$, its maximal
barriers are precisely its color classes $A$ and $B$. Now let $S$ denote a
nontrivial barrier of~$H$ which is not maximal, and adjust notation so that
$S \subset B$. It may be inferred from Proposition~\ref{prop:tight-cut-bipmcg}
that $H-S$ has precisely $|S| - 1$ isolated vertices each of which is a member
of~$A$, and it has precisely one nontrivial odd component~$K$ which gives
rise to a nontrivial barrier cut of~$H$, namely $\partial(V(K))$.

\smallskip
Since braces are bipartite matching covered graphs
which are free of nontrivial tight cuts,
Proposition~\ref{prop:tight-cut-bipmcg} may be used to obtain the following
characterizations of braces.

\begin{prop}
{\sc [Characterizations of Braces]}
\label{prop:characterizations-of-braces}
Let $H[A,B]$ denote a bipartite graph of order six or more, where $|A| = |B|$.
Then the following statements are equivalent:
\begin{enumerate}[(i)]
\item $H$ is a brace,
\item $|N(S)| \geq |S|+2$ for every nonempty subset~$S$ of~$A$ such that $|S| < |A|-1$, and
\item $H-\{a_1,a_2,b_1,b_2\}$ has a perfect matching for any four distinct vertices $a_1,a_2 \in A$ and $b_1,b_2 \in B$. \qed
\end{enumerate}
\end{prop}

\subsubsection{Near-Bipartite graphs}
\label{sec:barriers-tight-cuts-nbmcg}

Let $G$ denote an \Rgraph.
We adopt Notation~\ref{Not:Rgraph-removable-doubleton}.
For an odd subset~$X$ of $V(G)$,
we define its {\it majority part}~$X_+$
and its {\it minority part}~$X_-$
by regarding it as a subset of~$V(H)$.

\smallskip
Observe that, if~$X$ is the shore of a tight cut in~$G$
then it is the shore of a tight cut in~$H$ as well.
This observation, coupled with Proposition~\ref{prop:tight-cut-bipmcg},
may be used to derive
the following characterization of tight cuts in \nb\ graphs.

\begin{prop}
{\sc [Tight Cuts in Near-bipartite Graphs]}
\label{prop:tight-cut-nbmcg}
A cut~$\partial(X)$ of an \mbox{$R$-graph~$G$}
is tight if and only if the following hold:
\begin{enumerate}[(i)]
\item $X$ is odd and $|X_+| = |X_-|+1$,
and consequently, $|\overline{X}_+| = |\overline{X}_-| + 1$,
\item there are no edges between $X_-$ and $\overline{X}_-$;
adjust notation so that $X_- \subset A$,
\item one of $\alpha$ and $\beta$ has both ends in a majority part;
adjust notation so that $\alpha$ has both ends in $\overline{X}_+$, and
\item $\beta$ has at least one end in~$\overline{X}_-$.
\end{enumerate}
Consequently, $X_+$ is a nontrivial barrier of~$G$.
Moreover, the $\partial(X)$-contraction
$G / X$ is \nb\ with removable doubleton~$R$,
whereas the $\partial(X)$-contraction $G/ \overline{X}$ is bipartite.
\end{prop}
\vspace*{-0.2in}
\begin{proof}
A simple counting argument shows that if all of the statements {\it (i)} to {\it (iv)} hold then $\partial(X)$ is indeed a tight cut of~$G$.
See Figure~\ref{fig:tight-cut-nbmcg}.
Now suppose that $\partial(X)$ is a tight cut;
as noted earlier, $\partial(X) - R$ is a tight cut of~$H$.
Thus {\it (i)} and {\it (ii)} follow immediately from
Proposition~\ref{prop:tight-cut-bipmcg}.
Adjust notation so that $X_- \subset A$.

\vspace*{-0.1in}
\begin{figure}[!ht]
\centering
%\smallskip
\begin{tikzpicture}[scale=0.65]
%\draw (-2,-0.5) node[nodelabel]{$U$};
%\draw (-2,-3) node[nodelabel]{$W$};

\draw (0.5,-0.5)node{} -- (1.5,-0.5)node{};
\draw (1,-0.2)node[nodelabel]{$\alpha$};
\draw (0.5,-3)node{} -- (1.5,-3)node{};
\draw (1,-3.35)node[nodelabel]{$\beta$};

\draw (0,0.2) -- (4,0.2) -- (4,-1) -- (0,-1) -- (0,0.2);
\draw (2,0.2) node[above,nodelabel]{$\overline{X}_+$};

\draw (0,-2.5) -- (3.2,-2.5) -- (3.2,-3.7) -- (0,-3.7) -- (0,-2.5);
\draw (1.6,-3.7) node[below,nodelabel]{$\overline{X}_-$};

\draw (5.2,-2.5) -- (9.2,-2.5) -- (9.2,-3.7) -- (5.2,-3.7) -- (5.2,-2.5);
\draw (7.2,-3.7) node[below,nodelabel]{$X_+$};

\draw (6,0.2) -- (9.2,0.2) -- (9.2,-1) -- (6,-1) -- (6,0.2);
\draw (7.6,0.2) node[above,nodelabel]{$X_-$};

%\draw[ultra thin] (1,-0.5) -- (1,-3);
%\draw[ultra thin] (1.7,-0.5) -- (1.7,-3);
\draw[ultra thin] (2.4,-0.5) -- (6.1,-3);
\draw[ultra thin] (3.1,-0.5) -- (6.8,-3);
%\draw[ultra thin] (7.5,-3) -- (7.5,-0.5);
%\draw[ultra thin] (8.2,-3) -- (8.2,-0.5);
\draw (4.6,-5.6)node[nodelabel]{(a)};
\end{tikzpicture}
\hspace*{0.5in}
\begin{tikzpicture}[scale=0.65]
%\draw (-2,-0.5) node[nodelabel]{$U$};
%\draw (-2,-3) node[nodelabel]{$W$};

\draw (0.5,-0.5)node{} -- (1.5,-0.5)node{};
\draw (1,-0.2)node[nodelabel]{$\alpha$};
%\draw (0.5,-3)node{} -- (1.5,-3)node{};
%\draw (1,-3.35)node[nodelabel]{$\beta$};
\draw (2.7,-3)node{} -- (5.7,-3)node{};
\draw (4.2,-3.35)node[nodelabel]{$\beta$};

\draw (0,0.2) -- (4,0.2) -- (4,-1) -- (0,-1) -- (0,0.2);
\draw (2,0.2) node[above,nodelabel]{$\overline{X}_+$};

\draw (0,-2.5) -- (3.2,-2.5) -- (3.2,-3.7) -- (0,-3.7) -- (0,-2.5);
\draw (1.6,-3.7) node[below,nodelabel]{$\overline{X}_-$};

\draw (5.2,-2.5) -- (9.2,-2.5) -- (9.2,-3.7) -- (5.2,-3.7) -- (5.2,-2.5);
\draw (7.2,-3.7) node[below,nodelabel]{$X_+$};

\draw (6,0.2) -- (9.2,0.2) -- (9.2,-1) -- (6,-1) -- (6,0.2);
\draw (7.6,0.2) node[above,nodelabel]{$X_-$};

%\draw[ultra thin] (1,-0.5) -- (1,-3);
%\draw[ultra thin] (1.7,-0.5) -- (1.7,-3);
\draw[ultra thin] (2.4,-0.5) -- (6.1,-3);
\draw[ultra thin] (3.1,-0.5) -- (6.8,-3);
%\draw[ultra thin] (7.5,-3) -- (7.5,-0.5);
%\draw[ultra thin] (8.2,-3) -- (8.2,-0.5);
\draw (4.6,-5.6)node[nodelabel]{(b)};
\end{tikzpicture}
\vspace*{-0.2in}
\caption{Tight cuts in \nb\ graphs}
\label{fig:tight-cut-nbmcg}
%\medskip
\end{figure}
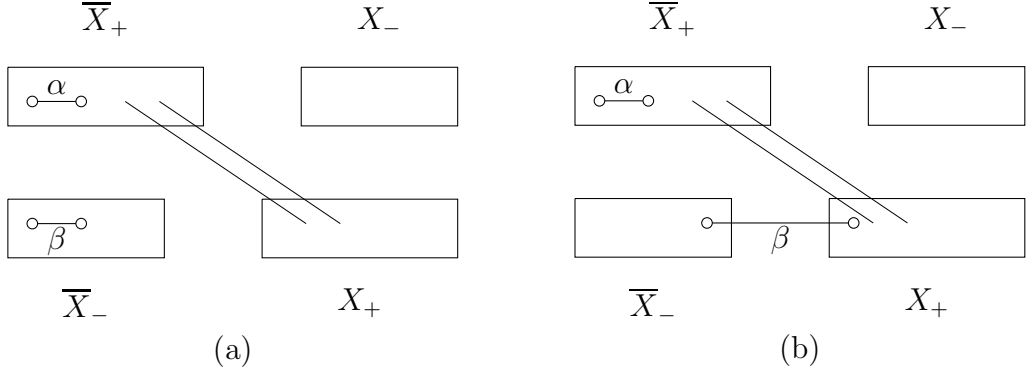

As each perfect matching of~$G$ which
contains~$\alpha$ must also contain~$\beta$, we
infer that at most one of $\alpha$ and $\beta$
lies in~$\partial(X)$. Furthermore,
if $\alpha$ has both ends in~$X_-$,
and likewise,
if $\beta$ has both ends in~$\overline{X}_-$,
then a simple counting argument shows that
any perfect matching~$M$ of~$G$ containing $\alpha$~and~$\beta$
meets~$\partial(X)$ in at least three edges;
this is a contradiction.

\smallskip
The above observations imply that
at least one of~$\alpha$ and $\beta$ has both ends in a majority part;
this proves {\it (iii)}.
As in the statement, adjust notation so that
$\alpha$ has both ends in~$\overline{X}_+$.
Now, if $\beta$ has both ends in~$X_+$
then it is easily seen that $\alpha$ and $\beta$ are both inadmissible.
This proves {\it (iv)}.
Note that, either~$\beta$ has both ends in~$\overline{X}_-$ as shown
in Figure~\ref{fig:tight-cut-nbmcg}a,
or it has one end in~$\overline{X}_-$ and the other end in~$X_+$
as shown in Figure~\ref{fig:tight-cut-nbmcg}b.

\smallskip
Note that $X_+$ is a nontrivial barrier of~$G$, and that
$G/ \overline{X}$ is bipartite.
We let $G_1:=G / X$ denote the other $\partial(X)$-contraction.
Observe that $H_1 := H/X$ is bipartite and matching covered.
Furthermore, in~$G_1$, $\alpha$ has both ends in one color class of~$H_1$,
and likewise, $\beta$ has both ends in the other color class of~$H_1$;
this is true for each of the two cases shown in Figure~\ref{fig:tight-cut-nbmcg}.
Since $H_1=G_1 - R$,
we infer that $G_1$ is \nb\ with removable doubleton~$R$.
This completes the proof of Proposition~\ref{prop:tight-cut-nbmcg}.
\end{proof}

Recall that a \nbrick\ is a matching covered graph whose tight cut decomposition
yields exactly one brick.
The following is an immediate consequence of
Proposition~\ref{prop:tight-cut-nbmcg}.

\begin{cor}
\label{cor:brick-of-nbmcg-is-nb}
An $R$-graph~$G$ is a \nbrick,
and its unique brick is also \nb\ with removable doubleton~$R$.
\qed
\end{cor}

In other words, a \nb\ graph~$G$ is a \nbrick,
and its unique brick, say~$J$, inherits its removable doubletons.
The {\it rank} of~$G$,
denoted ${\sf rank}(G)$,
is the order of the unique brick of~$G$. That is,
${\sf rank}(G) := |V(J)|$.

\smallskip
Proposition~\ref{prop:tight-cut-nbmcg} shows that
every tight cut of a \nb\ graph
is a barrier cut.
Now, let $S$ denote a nontrivial barrier of an $R$-graph~$G$,
and adjust notation so that $S \subset B$. It may be inferred
from Proposition~\ref{prop:tight-cut-nbmcg}
that $G-S$ has precisely $|S|-1$ isolated vertices
each of which is a member of~$A$,
and it has precisely one nontrivial odd component~$K$
which yields a nontrivial tight cut of~$G$, namely~$\partial(V(K))$.
Thus there is a bijective correspondence
between the nontrivial barriers of~$G$ and its nontrivial tight cuts.

%\newpage
\vspace*{-0.1in}
\subsection{The Three Case Lemma}
\label{sec:three-case-lemma}

Recall that a removable edge~$e$ of a brick~$G$ is \binv\ if $G-e$ is a \nbrick.
In this section, we will discuss a lemma of
Carvalho, Lucchesi and Murty \cite{clm02b}
that pertains to the structure of such {\nbrick}s, that is, those which are obtained
from a brick by deleting a single edge. This lemma is used extensively in their
works \cite{clm02a,clm06,clm12}, and it will play a vital role in
the proof of Theorem~\ref{thm:Rthin-nb-bricks}.

\smallskip
We will restrict ourselves to the case in which $G$ is an \Rbrick\ and $e$
is \Rcomp.
(By Proposition~\ref{prop:Rcompatible-is-binvariant}, $e$ is \binv.)
We adopt Notation~\ref{Not:Rgraph-removable-doubleton}.
As the name of the lemma suggests, there will be three cases, depending
on which we say that the `index' of~$e$ is zero, one or two.
In particular, the index of~$e$ (defined later) will be zero if~$G-e$ is a brick.

\smallskip
Now consider the situation in which~$G-e$ is not a brick; that is,
$G-e$ has a nontrivial tight cut. By Proposition~\ref{prop:tight-cut-nbmcg},
$G-e$ has a nontrivial barrier;
let~$S$ be such a barrier which is also maximal, and adjust notation
so that $S \subset B$. We let~$I$ denote the set of isolated vertices
of~$(G-e)-S$; note that $I \subset A$. Since $G$ itself is free
of nontrivial barriers, we infer that one end of~$e$ lies in~$I$ and its other
end lies in~$B-S$. This observation, coupled with the
Canonical Partition Theorem (\ref{thm:canonical-partition})
and the fact that $e$
has only two ends, implies that $G-e$ has at most two maximal
nontrivial barriers; furthermore, if it is has two such barriers then
one is a subset of~$A$ and the other is a subset of~$B$. 

\smallskip
The {\it index}
of~$e$, denoted ${\sf index}(e)$, is the number
of maximal nontrivial barriers in~$G-e$.
It follows from the preceding paragraph that
the index of~$e$ is either zero, one or two; and these
form the three cases.
\begin{comment}
\footnote{
This notion is closely related to the index of a thin edge defined
in Section~\ref{sec:index-thin-edges}. In fact, for an \Rthin\ edge,
these are equivalent; see Proposition~\ref{prop:notions-of-index-coincide}.
}
\end{comment}
\begin{comment}
\footnote{
This notion is closely related to the index of a strictly thin edge defined
in Section~\ref{sec:index-thin-edges}. In fact, for a \sRthin\ edge,
these are equivalent; the index three case does not arise in
\Rcomp\ edges.
}
\end{comment}
This is the gist of the lemma; apart from this,
it provides further information in the index two case which is especially
useful to us.
We now state the Three Case Lemma \cite{clm06},
as it is applicable to an {\Rcomp}\ edge of an \Rbrick;
see
Figures \ref{fig:three-case-lemma-index-one}
and \ref{fig:three-case-lemma-index-two}.
(The reason for the asymmetry in our notation in Case~(2)
is discussed in Section~\ref{sec:index-two}.)

\begin{lem}
{\sc [The Three Case Lemma]}
\label{lem:three-case}
Let $G$ be an \Rbrick, and $e$ an \Rcomp\ edge.
Let $H[A,B]:=G-R$.
Then one of the following three alternatives holds:
\begin{enumerate}
\item[{\rm (0)}] $G-e$ is a brick.
\item[{\rm (1)}] $G-e$ has only one maximal nontrivial barrier,
say~$S$. Adjust notation so that $S \subset B$.
Let $I$ denote the set of isolated vertices of~$(G-e)-S$.
Then $I \subset A$, and $e$ has one end in~$I$ and other end in~$B-S$.
\item[{\rm (2)}] $G-e$ has two maximal nontrivial barriers,
say~$S_1$ and $S^*_2$. Adjust notation so that $S_1 \subset B$
and $S^*_2 \subset A$.
Let $I_1$ denote the set of isolated vertices of~$(G-e)-S_1$,
and $I^*_2$ the set of isolated vertices of~$(G-e)-S^*_2$.
Then the following hold:
\begin{enumerate}[(i)]
\item $I_1 \subset A$ and $I^*_2 \subset B$;
\item $e$ has one end in~$I_1 - S^*_2$ and other end in~$I^*_2 - S_1$;
\item $S_2 := S^*_2 - I_1$ is the unique
maximal nontrivial barrier of \mbox{$(G-e)/X_1$}, where $X_1 := S_1 \cup I_1$;
furthermore, $S_2$ is a barrier of~$G-e$ as well,
and $I_2 := I^*_2 - S_1$ is the set of isolated vertices of $(G-e)-S_2$.~\qed
\end{enumerate}
\end{enumerate}
\end{lem}

Now, let $e$ denote an \Rcomp\ edge of an \Rbrick~$G$.
By the {\it rank} of~$e$,
denoted ${\sf rank}(e)$, we mean
the rank of the \Rgraph~$G-e$. That is, ${\sf rank}(e):={\sf rank}(G-e)$.
Recall that $e$ is \Rthin\ if the retract of~$G-e$ is a brick.
In particular, every \Rcomp\ edge of index zero is \Rthin,
and these are the only edges whose rank equals $n:=|V(G)|$.

\smallskip
In what follows, we will
further discuss the cases in which the index of~$e$ is either one or
two; in each case, we shall relate the rank of~$e$
with the information provided by the Three Case Lemma,
and we examine the conditions under which~$e$ is \Rthin.
These discussions are especially relevant to
Section~\ref{sec:equal-rank-lemma}.

\smallskip
We adopt Notation~\ref{Not:Rgraph-removable-doubleton}.
Let $y$~and~$z$ denote the ends of~$e$ such that $y \in A$ and $z \in B$.
Note that, if $y$ is cubic, then the two neighbours of~$y$ in~$G-e$
constitute a barrier of~$G-e$; a similar statement holds for~$z$.
It follows that if both ends of~$e$ are cubic then the index of~$e$ is two.

\subsubsection{Index one}
\label{sec:index-one}

Suppose that the index of~$e$ is one.
As in case (1) of the Three Case Lemma,
we let $S$ denote the unique maximal nontrivial
barrier of~$G-e$, and $I$ the set of isolated vertices of~$(G-e)-S$.
Note that $|I| = |S|-1$.
We adjust notation so that $S \subset B$ and $I \subset A$;
see Figure~\ref{fig:three-case-lemma-index-one}.
Observe that $y \in I$ and $z \in B-S$.

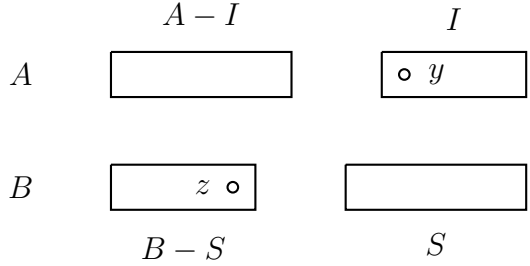
\begin{figure}[!ht]
\centering
\begin{tikzpicture}[scale=0.6, thick]
\draw (-2,-0.5) node[nodelabel]{$A$};
\draw (-2,-3) node[nodelabel]{$B$};

\draw (0,0) -- (4,0) -- (4,-1) -- (0,-1) -- (0,0);
\draw (2,-0.5) node[above,nodelabel]{$A-I$};

\draw (0,-2.5) -- (3.2,-2.5) -- (3.2,-3.5) -- (0,-3.5) -- (0,-2.5);
\draw (1.6,-3) node[below,nodelabel]{$B-S$};

\draw (5.2,-2.5) -- (9.2,-2.5) -- (9.2,-3.5) -- (5.2,-3.5) -- (5.2,-2.5);
\draw (7.2,-3.5) node[below,nodelabel]{$S$};

\draw (6,0) -- (9.2,0) -- (9.2,-1) -- (6,-1) -- (6,0);
\draw (7.6,0) node[above,nodelabel]{$I$};

%\draw[dashed, thin] (2.7,-3) -- (6.5,-0.5);
\draw (2.7,-3) node{}node[left,nodelabel]{$z$};
\draw (6.5,-0.5) node{}node[right,nodelabel]{$y$};
%\draw (5.2,-0.5) node[below,nodelabel]{$e$};
\end{tikzpicture}
\vspace*{-0.2in}
\caption{An \Rcomp\ edge of index one}
\label{fig:three-case-lemma-index-one}
%\bigskip
\end{figure}

In this case, $G-e$ has a unique nontrivial tight cut~$\partial(X)$,
where $X:=S \cup I$.
Consequently, $(G-e)/X$ is the brick of~$G-e$,
and the rank of~$e$ is $|V(G) - X|+1$.
Furthermore, $e$ is \Rthin\ if and only if
$|S|=2$; and in this case, $y$ is cubic,
$N(y)=S \cup \{z\}$, and ${\sf rank}(e)=n-2$.

\subsubsection{Index two}
\label{sec:index-two}

Suppose that the index of~$e$ is two.
As in case (2) of the Three Case Lemma,
we let $S_1$ denote one of the two maximal nontrivial barriers of~$G-e$,
and $I_1$ the set of isolated vertices of~$(G-e)-S_1$, adjusting notation
so that $S_1 \subset B$ and $I_1 \subset A$.
Note that $|I_1| = |S_1| - 1$ and that $y \in I_1$;
see Figure~\ref{fig:three-case-lemma-index-two}.

\smallskip
Now, let $S^*_2$ denote the unique maximal nontrivial barrier of~$G-e$
which is a subset of~$A$, and $I^*_2$ the set of isolated vertices
of~$(G-e)-S^*_2$. 
As in the index one case (see Figure~\ref{fig:three-case-lemma-index-one}),
we would like to break~$V(G)$ into disjoint subsets in order to be able
to compute the rank of~$e$.
However, this is complicated by the possibility that $S^*_2 \cap I_1$ may be
nonempty.
This explains the asymmetry in our notation in case (2).
Fortunately, it turns out that $S_2:=S^*_2 - I_1$ is the only maximal nontrivial
barrier of~$(G-e)/X_1$, where $X_1 := S_1 \cup I_1$.
Furthermore, $S_2$ is a barrier of~$G-e$ as well, and $I_2 := I^*_2-S_1$
is the set of isolated vertices of~$(G-e)-S_2$.
Note that $|I_2| = |S_2|-1$ and that $z \in I_2$;
see Figure~\ref{fig:three-case-lemma-index-two}.
We let $X_2:=S_2 \cup I_2$.

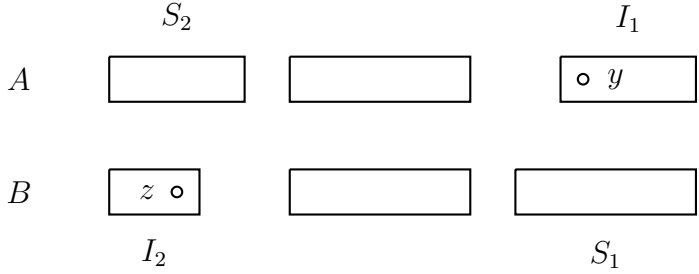
\begin{figure}[!ht]
%\smallskip
\centering
\begin{tikzpicture}[scale=0.6, thick]
\draw (-2,-0.5) node[nodelabel]{$A$};
\draw (-2,-3) node[nodelabel]{$B$};

\draw (0,0) -- (3,0) -- (3,-1) -- (0,-1) -- (0,0);
\draw (1.5,0) node[above,nodelabel]{$S_2$};

\draw (0,-2.5) -- (2,-2.5) -- (2,-3.5) -- (0,-3.5) -- (0,-2.5);
\draw (1,-3.5) node[below,nodelabel]{$I_2$};

\draw (4,0) -- (8,0) -- (8,-1) -- (4,-1) -- (4,0);

\draw (4,-2.5) -- (8,-2.5) -- (8,-3.5) -- (4,-3.5) -- (4,-2.5);

\draw (9,-2.5) -- (13,-2.5) -- (13,-3.5) -- (9,-3.5) -- (9,-2.5);
\draw (11,-3.5) node[below,nodelabel]{$S_1$};

\draw (10,0) -- (13,0) -- (13,-1) -- (10,-1) -- (10,0);
\draw (11.5,0) node[above,nodelabel]{$I_1$};

%\draw[dashed, thin] (1.5,-3) -- (10.5,-0.5);
\draw (1.5,-3) node{}node[left,nodelabel]{$z$};
\draw (10.5,-0.5) node{}node[right,nodelabel]{$y$};
%\draw (6,-1.95) node[above,nodelabel]{$e$};
\end{tikzpicture}
\vspace*{-0.2in}
\caption{An \Rcomp\ edge of index two}
\label{fig:three-case-lemma-index-two}
%\bigskip
\end{figure}

In this case, $\partial(X_1)$ and $\partial(X_2)$ are both tight cuts of~$G-e$;
more importantly, $\partial(X_2)$ is the unique tight cut
of~$(G-e)/X_1$.
Consequently, $((G-e)/X_1)/X_2$ is the brick of~$G-e$, and the rank of~$e$
is~$|V(G) - X_1 - X_2| + 2$.

\smallskip
Furthermore, $e$ is \Rthin\ if and only
if $|S_1|=2=|S_2|$; and in this case, $y$ and $z$ are both cubic,
$N(y)=S_1 \cup \{z\}$ and $N(z) = S_2 \cup \{y\}$, and
${\sf rank}(e)=n-4$; also, by switching the roles of~$S_1$~and~$S^*_2$,
we infer that $|S^*_2|=2$.

\subsubsection{Index and Rank of an \Rthin\ Edge}
\label{sec:index-rank-Rthin}

The following characterization of \Rthin\ edges
is immediate from our discussion in the previous two sections.

\begin{prop}
{\sc [Characterization of \Rthin\ Edges in terms of Barriers]}
\label{prop:equivalent-definition-Rthin}
An \Rcomp\ edge~$e$ of an \Rbrick~$G$ is \Rthin\ if and only if
every barrier of~$G-e$ has at most two vertices. \qed
\end{prop}

In summary, if the index of~$e$ is zero then $e$ is thin and its rank is $n:=|V(G)|$.
If the index of~$e$ is one then ${\sf rank}(e) \leq n-2$, and equality
holds if and only if $e$ is thin.
Likewise, if the index of~$e$ is two then ${\sf rank}(e) \leq n-4$,
and equality holds if and only if $e$ is thin.

\smallskip
The following proposition gives an equivalent definition of index
of an \Rthin\ edge.

\begin{prop}
\label{prop:notions-of-index-coincide}
Let $G$ be an \Rbrick, and $e$ an \Rthin\ edge.
Then the following statements hold:
\begin{enumerate}[(i)]
\item ${\sf index}(e) = 0$ if and only if both ends of $e$ have degree
four or more in~$G$;
\item ${\sf index}(e) = 1$ if and only if exactly one end of~$e$ has degree
three in~$G$; and
\item ${\sf index}(e) = 2$ if and only if both ends of $e$ have degree three in~$G$ and
$e$~does not lie in a triangle.
\end{enumerate}
\end{prop}
\begin{proof}
We note that ${\sf index}(e) = 0$ if and only if $G-e$
is free of nontrivial barriers, that is, $G-e$ is a brick;
and since $e$ is a thin edge, the latter holds if and only if
both ends of $e$ have degree four or more in~$G$.
This proves {\it (i)}.

\smallskip
Let $n:=|V(G)|$.
We note that ${\sf index}(e) = 1$ if and only if ${\sf rank}(e) = n-2$;
and since $e$ is a thin edge,
the latter holds if and only if exactly one end of~$e$ has degree three in~$G$.

\smallskip
Now suppose that ${\sf index}(e) = 2$, whence ${\sf rank}(e) = n-4$,
and consequently, both ends of $e$ have degree three in~$G$.
Conversely, if both ends of $e$ have degree three in~$G$ then $G-e$
has two nontrivial barriers which lie in different color classes of~$(G-e)-R$,
and thus ${\sf index}(e) = 2$; furthermore, since $e$ is \Rcomp,
neither end of $e$ is incident with an edge of~$R$ and thus $e$ does
not lie in a triangle.
\end{proof}

%\newpage
\section{Generating Near-Bipartite Bricks}
\label{sec:generating-nb-bricks}

In this section, our goal is to prove the \Rthin\ Edge Theorem (\ref{thm:Rthin-nb-bricks}).
In fact, we will prove a stronger result, as described below.

\smallskip
Let $G$ be an \Rbrick\ distinct from $K_4$ and $\overline{C_6}$.
Then, by Theorem~\ref{thm:clm-Rcompatible-nb-bricks}
of Carvalho et al., $G$ has an \Rcomp\ edge; let~$e$ be any such edge.
Recall from Section~\ref{sec:three-case-lemma}
that there are two parameters associated with~$e$:
the rank of~$e$ is the order of the unique brick of~$G-e$; and,
the index of~$e$ is the number of maximal nontrivial barriers of~$G-e$, which
by the Three Case Lemma (\ref{lem:three-case})
is either zero, one or two.
Using these parameters, we may state our stronger theorem as follows.

\begin{thm}\label{thm:rank-plus-index}
Let $G$ be an \Rbrick\ which is distinct from $K_4$~and~$\overline{C_6}$,
and let $e$ denote an \Rcomp\ edge of~$G$.
Then one of the following alternatives hold:
\begin{itemize}
\item either $e$ is \Rthin,
\item or there exists another \Rcomp\ edge~$f$ such that:
\begin{enumerate}[(i)]
\item $f$ has an end each of whose neighbours in~$G-e$ lies in a barrier of~$G-e$, and
\item ${\sf rank}(f) + {\sf index}(f) > {\sf rank}(e) + {\sf index}(e)$.
\end{enumerate}
\end{itemize}
\end{thm}

Since the rank and index are bounded quantities,
the above theorem
immediately implies the \Rthin\ Edge Theorem (\ref{thm:Rthin-nb-bricks}).
Our proof uses tools from the work of Carvalho et al. \cite{clm06}, and the
overall approach is inspired by their proof of the
Thin Edge Theorem~(\ref{thm:clm-thin-bricks}).

\smallskip
The following proposition shows that condition~{\it (ii)}
in Theorem~\ref{thm:rank-plus-index} is implied by a weaker
condition involving only the rank function.

\begin{prop}
\label{prop:greater-rank-suffices}
Suppose that $e$ and $f$ denote two \Rcomp\ edges of an \Rbrick~$G$.
If ${\sf rank}(f) > {\sf rank}(e)$ then
\mbox{${\sf rank}(f) + {\sf index}(f) > {\sf rank}(e) + {\sf index}(e)$}.
\end{prop}
\begin{proof}
Note that, since the rank of an edge is even, ${\sf rank}(f) > {\sf rank}(e)+1$.
As the index of an edge is either zero, one or two, we only need to examine
the case in which ${\sf index}(e)=2$ and ${\sf index}(f) =0$. However, in this
case, ${\sf rank}(f) = n$ and ${\sf rank}(e) \leq n-4$ where $n:=|V(G)|$,
and thus the conclusion holds.
\end{proof}

\smallskip
In the statement of Theorem~\ref{thm:rank-plus-index}, if the given
\Rcomp\ edge~$e$ is thin, then the assertion is vacuously true.
Thus, in its proof, we may assume that $e$ is not thin. It then follows from
Proposition~\ref{prop:equivalent-definition-Rthin} that $G-e$ has a barrier
with three or more vertices; let $S$ be such a barrier.
In the next section, we introduce the notion of a candidate edge
(relative to~$e$~and~$S$) which, as we
will see, is an \Rcomp\ edge that satisfies condition~{\it (i)}
in the statement of Theorem~\ref{thm:rank-plus-index},
and has rank at least that of~$e$.

\subsection{The candidate set \candidateset{e}{S}}

Let $G$ be an \Rbrick, and let $e:=yz$ denote an \Rcomp\ edge
which is not thin.
We first set up some notation and conventions which are used
in the rest of this paper.

\begin{Not}
\label{Not:Rcompatible-in-Rbrick}
We shall denote by $H[A,B]$ the underlying bipartite graph $G-R$.
We let $R:=\{\alpha,\beta\}$;
and we adopt the convention that $\alpha:= a_1a_2$ has both ends in~$A$,
whereas $\beta:=b_1b_2$ has both ends in~$B$.
Adjust notation so that $y \in A$ and $z \in B$.
\end{Not}

The reader is advised to review
Section~\ref{sec:barriers-tight-cuts-nbmcg}
before proceeding further.
Let $S$ be a barrier of~$G-e$ such that $|S| \geq 3$,
and $I$ the set of isolated vertices of~$(G-e)-S$.
Adjust notation so that $S \subset B$ and $I \subset A$,
as shown in Figure~\ref{fig:candidate-set}a.
Observe that $X:=S \cup I$ is the shore of a tight cut in~$G-e$,
as well as in~$H-e$.
By Proposition~\ref{prop:tight-cut-nbmcg},
$\alpha$ has both ends in~$A-I$; whereas $\beta$ either
has both ends in~$B-S$, or it has one end in~$B-S$ and another in~$S$.
We denote the bipartite matching covered graph
$$(H-e)/\overline{X} \rightarrow \overline{x}$$ by \bipartitebarrier{e}{S}.
Note that its color classes are the sets $I \cup \{\overline{x}\}$ and $S$;
see Figure~\ref{fig:candidate-set}b.

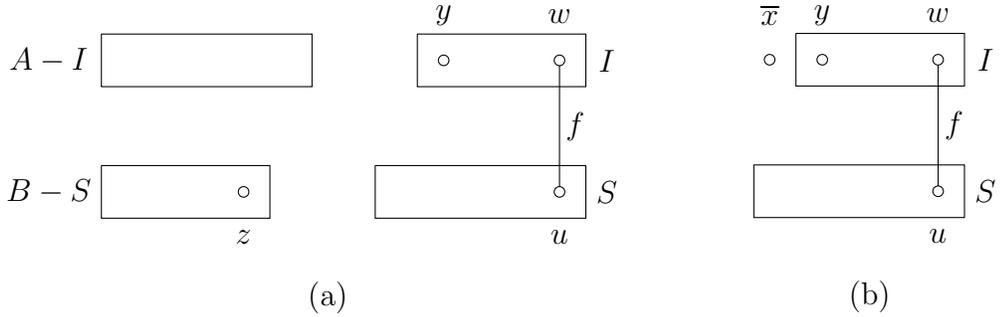
\begin{figure}[!ht]
%\bigskip
\centering
\begin{tikzpicture}[scale=0.7]
%\draw (-2,-0.5) node[nodelabel]{$A$};
%\draw (-2,-3) node[nodelabel]{$B$};
\draw (8.7,-0.5) -- (8.7,-3);
\draw (8.7,0.35) node[nodelabel]{$w$};
\draw (8.7,-3.85) node[nodelabel]{$u$};
\draw (9,-1.75)node[nodelabel]{$f$};
\draw (8.7,-0.5)node{};
\draw (8.7,-3)node{};

\draw (0,0) -- (4,0) -- (4,-1) -- (0,-1) -- (0,0);
\draw (-1,-0.5) node[nodelabel]{$A-I$};

\draw (0,-2.5) -- (3.2,-2.5) -- (3.2,-3.5) -- (0,-3.5) -- (0,-2.5);
\draw (-1,-3) node[nodelabel]{$B-S$};

\draw (5.2,-2.5) -- (9.2,-2.5) -- (9.2,-3.5) -- (5.2,-3.5) -- (5.2,-2.5);
\draw (9.6,-3) node[nodelabel]{$S$};

\draw (6,0) -- (9.2,0) -- (9.2,-1) -- (6,-1) -- (6,0);
\draw (9.6,-0.5) node[nodelabel]{$I$};

%\draw[dashed, thin] (2.7,-3) -- (6.5,-0.5);
\draw (2.7,-3) node{};
\draw (2.7,-3.85) node[nodelabel]{$z$};
\draw (6.5,-0.5) node{};
\draw (6.5,0.35) node[nodelabel]{$y$};
%\draw (5.2,-0.5) node[below,nodelabel]{$e$};

\draw (4.3,-5)node[nodelabel]{(a)};
\end{tikzpicture}
\hspace*{0.4in}
\begin{tikzpicture}[scale=0.7]
\draw (8.7,-0.5) -- (8.7,-3);
\draw (8.7,0.35) node[nodelabel]{$w$};
\draw (8.7,-3.85) node[nodelabel]{$u$};
\draw (9,-1.75)node[nodelabel]{$f$};
\draw (8.7,-0.5)node{};
\draw (8.7,-3)node{};

\draw (5.5,-0.5)node{};
\draw (5.5,0.35)node[nodelabel]{$\overline{x}$};

\draw (5.2,-2.5) -- (9.2,-2.5) -- (9.2,-3.5) -- (5.2,-3.5) -- (5.2,-2.5);
\draw (9.6,-3) node[nodelabel]{$S$};

\draw (6,0) -- (9.2,0) -- (9.2,-1) -- (6,-1) -- (6,0);
\draw (9.6,-0.5) node[nodelabel]{$I$};

\draw (6.5,-0.5) node{};
\draw (6.5,0.35) node[nodelabel]{$y$};
\draw (7.4,-5)node[nodelabel]{(b)};
\end{tikzpicture}
\vspace*{-0.4in}
\caption{(a) $S$ is a barrier of $G-e$ such that $|S| \geq 3$ ; (b) the bipartite graph~\bipartitebarrier{e}{S}}
\label{fig:candidate-set}
%\bigskip
\end{figure}

\begin{Def}
{\sc [The Candidate Set \candidateset{e}{S}]}
\label{Def:candidate-set}
We denote by \candidateset{e}{S} the set of those removable edges
of~\bipartitebarrier{e}{S} which
are not incident with the contraction vertex~$\overline{x}$, and we refer to it
as the candidate set (relative to $e$ and the barrier~$S$ of~$G-e$), and each member
of \candidateset{e}{S} is called a candidate edge.
\end{Def}

We remark that Carvalho et al. \cite{clm06} used a similar notion. 
Since their work concerns general bricks (that is, not just near-bipartite ones),
they consider the graph \mbox{$(G-e)/\overline{X} \rightarrow \overline{x}$}
and its removable edges which are not incident with the contraction vertex.
See Lemma~23 and Theorem~24 in \cite{clm06}.

\smallskip
Now, let $f:=uw$ denote a member of the candidate set \candidateset{e}{S},
as shown in Figure~\ref{fig:candidate-set}b.
The end~$w$ of~$f$ lies in~$I$, and all of
the neighbours of~$w$, in~$G-e$, lie in the barrier~$S$;
consequently, $f$ satisfies condition~{\it (i)}, Theorem~\ref{thm:rank-plus-index}.
It should be noted that $e$ and $f$ are adjacent if and only if
$w$ is the same as~$y$.
We now show that $f$ is an \Rcomp\ edge and it has rank at least that of~$e$.
The argument pertaining to ranks is the same as that in \cite[Lemma~26]{clm06}.

\begin{prop}
{\sc [Properties of Candidate Edges]}
\label{prop:edges-of-candidate-set}
Every member of \candidateset{e}{S}
is an \Rcomp\ edge of~$G-e$, and of~$G$,
and has rank at least that of~$e$.
%furthermore, it satisfies condition {\it (i)} in the
%statement of Theorem~\ref{thm:rank-plus-index},
%and has rank at least that of~$e$.
Conversely, each \Rcomp\ edge of~$G-e$, which is incident with
a vertex of~$I$, is a member of \candidateset{e}{S}.
\end{prop}
\begin{proof}
Let $f$ be any member of \candidateset{e}{S}, as shown in
Figure~\ref{fig:candidate-set}b.
We will use Proposition~\ref{prop:removable-edges-across-tight-cuts}
to show that $f$ is \Rcomp\ in~$G-e$.

\smallskip
Observe that \bipartitebarrier{e}{S} is one of the $C$-contractions of~$H-e$,
where $C:=\partial(X)-e-R$ is a tight cut.
Since $f$ is removable in \bipartitebarrier{e}{S}
and $f \notin C$,
Proposition~\ref{prop:removable-edges-across-tight-cuts} implies that
$f$ is removable in~$H-e$ as well.
A similar argument shows that $f$ is removable in~$G-e$.
Thus, $f$ is \Rcomp\ in $G-e$; the exchange property
(Proposition~\ref{prop:exchange-property-Rcompatible-nbmcg})
implies that $f$ is \Rcomp\ in~$G$ as well.

\smallskip
Note that since both ends of~$f$ are in the bipartite shore~$X$, the brick
of~$G-e-f$ is the same as the brick of~$G-e$. In particular,
${\sf rank}(G-e-f) = {\sf rank}(G-e)$. On the other hand, note that if~$D$
is any tight cut of~$G-f$ then $D-e$ is a tight cut of~$G-e-f$, whence
${\sf rank}(G-f) \geq {\sf rank}(G-e-f)$. Thus ${\sf rank}(f) \geq {\sf rank}(e)$.
This proves the first statement.

\smallskip
Now suppose that $f$ is an \Rcomp\ edge of~$G-e$ which is incident
at some vertex of~$I$.
In particular, $H-e-f$ is matching covered; that is, $f$~is removable
in~$H-e$.
By Proposition~\ref{prop:removable-edges-across-tight-cuts},
$f$ is removable in~\bipartitebarrier{e}{S}.
This completes the proof of Proposition~\ref{prop:edges-of-candidate-set}.
\end{proof}

In summary, we have shown that every candidate edge is \Rcomp;
furthermore, it satisfies condition~{\it (i)}, Theorem~\ref{thm:rank-plus-index};
and it has rank at least that of~$e$.

\smallskip
The following property of candidate sets will be useful in dealing with
those nontrivial
barriers of~$G-e$ which are not maximal.

\begin{cor}
\label{cor:candidate-containment-property}
Let $S^*$ be any barrier of~$G-e$. If $S \subset S^*$
then \candidateset{e}{S} $\subset$ \candidateset{e}{S^*}.
\end{cor}
\begin{proof}
Let $f$ be a member of~\candidateset{e}{S}.
Then $f$ is incident with some vertex of~$I$, say~$w$.
Note that $w$ also lies in~$I^*$ which denotes the set of isolated vertices
of~$(G-e)-S^*$.

\smallskip
As $f$ is a member of \candidateset{e}{S},
Proposition~\ref{prop:edges-of-candidate-set}
implies that $f$ is \Rcomp\ in~$G-e$.
Consequently, since $f$ is incident at $w \in I^*$, the last assertion
of Proposition~\ref{prop:edges-of-candidate-set},
with $S^*$ playing the role of~$S$,
implies that $f$ is a member of \candidateset{e}{S^*}.
Thus \candidateset{e}{S} $\subset$ \candidateset{e}{S^*}.
\end{proof}

Now, we will prove two lemmas; each of which gives an upper bound on the
number of non-removable edges incident at a vertex of
the bipartite graph~\bipartitebarrier{e}{S},
which is distinct from the contraction vertex~$\overline{x}$.
Both of them are easy
applications of the Lov{\'a}sz-Vempala Lemma (\ref{lem:lovasz-vempala});
we will use arguments similar to those in the proof of
Corollary~\ref{cor:application-of-LV}.

\begin{lem}
\label{lem:non-removable-at-S}
Let $u$ denote a vertex of~$S$ which has degree three or more in
\bipartitebarrier{e}{S}.
Then at most two edges of~$\partial(u)-\beta$ are non-removable in
\bipartitebarrier{e}{S}.
Furthermore, if precisely two edges of~$\partial(u) - \beta$ are non-removable
in \bipartitebarrier{e}{S}
and if vertices $u$ and $\overline{x}$ are adjacent then the
edge~$u\overline{x}$ is non-removable in \bipartitebarrier{e}{S}.
\end{lem}
\begin{proof}
Assume that there are $k \geq 1$ non-removable edges incident with the vertex~$u$,
namely, $uw_1, uw_2, \dots, uw_k$. Then, by Lemma~\ref{lem:lovasz-vempala},
there exist partitions $(A_0, A_1, \dots, A_k)$ of $I \cup \{\overline{x}\}$,
and $(B_0, B_1, \dots, B_k)$ of~$S$, such that $u \in B_0$, and
for $j \in \{1, 2, \dots, k\}$: \mbox{(i)~$|A_j|=|B_j|$}, (ii)~$w_j \in A_j$
and (iii) $N(A_j) = B_j \cup \{u\}$. See Figure~\ref{fig:non-removable-at-S}.

\begin{figure}[!ht]
\centering
%\bigskip
\begin{tikzpicture}[scale=0.7]
\draw[thin] (2,-5) -- (4,-0.5);
\draw[thin] (2,-5) -- (7.5,-0.5);
\draw[thin] (2,-5) -- (13,-0.5);

\draw (-2,-0.5)node[nodelabel]{$I \cup \{\overline{x}\}$};
\draw (-2,-5)node[nodelabel]{$S$};

\draw (0,0) -- (2.5,0) -- (2.5,-1) -- (0,-1) -- (0,0);
\draw (1.25,0) node[nodelabel,above]{$A_0$};

\draw (2,-5)node{};
\draw (2.2,-5)node[left,nodelabel]{$u$};

\draw (3.5,0) -- (6,0) -- (6,-1) -- (3.5,-1) -- (3.5,0);
\draw (4.75,0) node[nodelabel,above]{$A_1$};

\draw (7,0) -- (9.5,0) -- (9.5,-1) -- (7,-1) -- (7,0);
\draw (8.25,0) node[nodelabel,above]{$A_2$};

\draw[dotted] (10,-0.5) -- (12,-0.5);

\draw (12.5,0) -- (15,0) -- (15,-1) -- (12.5,-1) -- (12.5,0);
\draw (13.75,0) node[nodelabel,above]{$A_k$};

\draw (0,-4.5) -- (2.5,-4.5) -- (2.5,-5.5) -- (0,-5.5) -- (0,-4.5);
\draw (1.25,-5.5) node[nodelabel,below]{$B_0$};

\draw (3.5,-4.5) -- (6,-4.5) -- (6,-5.5) -- (3.5,-5.5) -- (3.5,-4.5);
\draw (4.75,-5.5) node[nodelabel,below]{$B_1$};

\draw (4,-0.5)node{};
\draw (3.8,-0.5)node[right,nodelabel]{$w_1$};

\draw (7,-4.5) -- (9.5,-4.5) -- (9.5,-5.5) -- (7,-5.5) -- (7,-4.5);
\draw (8.25,-5.5) node[nodelabel,below]{$B_2$};

\draw (7.5,-0.5)node{};
\draw (7.3,-0.5)node[right,nodelabel]{$w_2$};

\draw[dotted] (10,-5) -- (12,-5);

\draw (12.5,-4.5) -- (15,-4.5) -- (15,-5.5) -- (12.5,-5.5) -- (12.5,-4.5);
\draw (13.75,-5.5) node[nodelabel,below]{$B_k$};

\draw (13,-0.5)node{};
\draw (12.8,-0.5)node[right,nodelabel]{$w_k$};

\end{tikzpicture}
\vspace*{-0.2in}
\caption{Illustration for Lemma~\ref{lem:non-removable-at-S}}
\label{fig:non-removable-at-S}
%\bigskip
\end{figure}
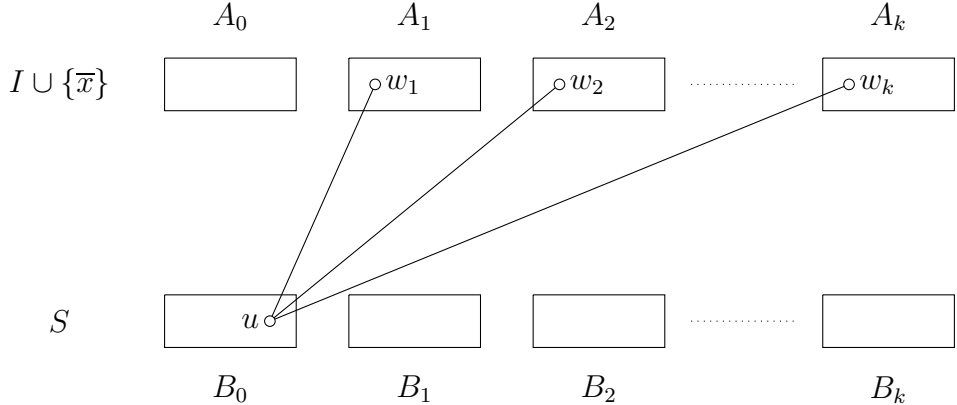

\smallskip
For $1 \leq j \leq k$, note that $B_j \cup \{u\}$ is a barrier of
\bipartitebarrier{e}{S}. Moreover,
if the set~$A_j$ contains neither the contraction vertex~$\overline{x}$ nor
the end~\vertexa~of~$e$, then $B_j \cup \{u\}$ is a barrier of~$G$ itself,
which is not possible as $G$ is a brick. We thus arrive at the conclusion that
$k \leq 2$, which proves the first part of the assertion.

\smallskip
Now consider the case when $k=2$. It follows from the above argument that
one of the vertices~\vertexa~and $\overline{x}$ lies in the set~$A_1$, whereas
the other vertex lies in the set~$A_2$. Adjust notation so that~\vertexa~$\in A_1$
and $\overline{x} \in A_2$.
Observe that if $u$ and $\overline{x}$ are adjacent, then $u\overline{x}$ is
the unique edge between $B_0$ and $A_2$, and it is non-removable
in
\bipartitebarrier{e}{S} by assumption. This completes the proof of
Lemma~\ref{lem:non-removable-at-S}.
\end{proof}

Now we turn to the examination of non-removable edges of
\bipartitebarrier{e}{S}
incident with vertices in~$I$. The proof is similar to that of
Lemma~\ref{lem:non-removable-at-S},
except that the roles of the color classes
$S$ and $I \cup \{\overline{x}\}$ are interchanged.

\begin{lem}
\label{lem:non-removable-at-I}
Let $w$ denote a vertex of~$I$ which has degree three or more in
\bipartitebarrier{e}{S}.
Then at most two edges of~$\partial(w)-e$ are non-removable in
\bipartitebarrier{e}{S}.
Furthermore, if precisely two edges of~$\partial(w)-e$ are non-removable in
\bipartitebarrier{e}{S}
then the following hold:
\begin{enumerate}[(i)]
\item an end of~$\beta$ lies in~$S$; adjust notation so that $b_1 \in S$,
\item in \bipartitebarrier{e}{S},
the vertices $b_1$ and $\overline{x}$ are nonadjacent,
\item if $b_1$ and $w$ are adjacent then the edge~$b_1w$ is
non-removable in \bipartitebarrier{e}{S}, and
\item $w$ is distinct from the end \vertexa~of~$e$.
\end{enumerate}
\end{lem}
\begin{proof}
Suppose that there exist $k \geq 1$ non-removable edges incident at the vertex~$w$,
namely, $wu_1, wu_2, \dots ,wu_k$.
Then, by Lemma~\ref{lem:lovasz-vempala}, there exist partitions
$(A_0, A_1, \dots, A_k)$ of the color class~$I \cup \{\overline{x}\}$,
and $(B_0, B_1, \dots, B_k)$ of the color class~$S$, such that
$w \in A_0$, and for $j \in \{1,2, \dots, k\}$: (i) $|A_j| = |B_j|$,
(ii) $u_j \in B_j$ and (iii) $N(B_j)=A_j \cup \{w\}$.
See Figure~\ref{fig:non-removable-at-I}.

\begin{figure}[!ht]
\centering
%\bigskip
\begin{tikzpicture}[scale=0.7]
\draw[thin] (2,-0.5) -- (4,-5);
\draw[thin] (2,-0.5) -- (7.5,-5);
\draw[thin] (2,-0.5) -- (13,-5);

\draw (-2,-0.5)node[nodelabel]{$I \cup \{\overline{x}\}$};
\draw (-2,-5)node[nodelabel]{$S$};

\draw (0,0) -- (2.5,0) -- (2.5,-1) -- (0,-1) -- (0,0);
\draw (1.25,0) node[nodelabel,above]{$A_0$};

\draw (2,-0.5)node{};
\draw (2.2,-0.5)node[left,nodelabel]{$w$};

\draw (3.5,0) -- (6,0) -- (6,-1) -- (3.5,-1) -- (3.5,0);
\draw (4.75,0) node[nodelabel,above]{$A_1$};

\draw (7,0) -- (9.5,0) -- (9.5,-1) -- (7,-1) -- (7,0);
\draw (8.25,0) node[nodelabel,above]{$A_2$};

\draw[dotted] (10,-0.5) -- (12,-0.5);

\draw (12.5,0) -- (15,0) -- (15,-1) -- (12.5,-1) -- (12.5,0);
\draw (13.75,0) node[nodelabel,above]{$A_k$};

\draw (0,-4.5) -- (2.5,-4.5) -- (2.5,-5.5) -- (0,-5.5) -- (0,-4.5);
\draw (1.25,-5.5) node[nodelabel,below]{$B_0$};

\draw (3.5,-4.5) -- (6,-4.5) -- (6,-5.5) -- (3.5,-5.5) -- (3.5,-4.5);
\draw (4.75,-5.5) node[nodelabel,below]{$B_1$};

\draw (4,-5)node{};
\draw (3.8,-5)node[right,nodelabel]{$u_1$};

\draw (7,-4.5) -- (9.5,-4.5) -- (9.5,-5.5) -- (7,-5.5) -- (7,-4.5);
\draw (8.25,-5.5) node[nodelabel,below]{$B_2$};

\draw (7.5,-5)node{};
\draw (7.3,-5)node[right,nodelabel]{$u_2$};

\draw[dotted] (10,-5) -- (12,-5);

\draw (12.5,-4.5) -- (15,-4.5) -- (15,-5.5) -- (12.5,-5.5) -- (12.5,-4.5);
\draw (13.75,-5.5) node[nodelabel,below]{$B_k$};

\draw (13,-5)node{};
\draw (12.8,-5)node[right,nodelabel]{$u_k$};

\end{tikzpicture}
\vspace*{-0.2in}
\caption{Illustration for Lemma~\ref{lem:non-removable-at-I}}
\label{fig:non-removable-at-I}
%\bigskip
\end{figure}
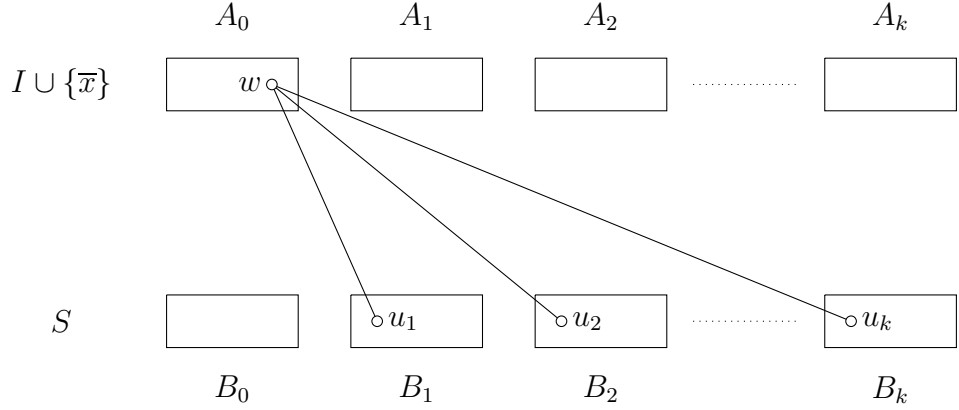

For $1 \leq j \leq k$, note that $A_j \cup \{w\}$ is a barrier of
\bipartitebarrier{e}{S}.
Furthermore, if the contraction vertex~$\overline{x}$ is not in~$A_j$, or if
an end of the edge~$\beta$ is not in~$B_j$, then $A_j \cup \{w\}$
is a barrier of~$G$ itself, which is absurd since $G$ is a brick.
Clearly, this would be the case for some $j \in \{1,2, \dots, k\}$ if $k \geq 3$.
We conclude that $k \leq 2$, thus establishing the first part of the assertion.

\smallskip
Now suppose that $k = 2$. It follows from the preceding paragraph
that an end of $\beta$ lies in~$B_1$ or in~$B_2$. This proves~{\it (i)}.
Adjust notation so that $b_1 \in B_1$. Furthermore, the contraction
vertex~$\overline{x}$ lies in~$A_2$. Consequently,
vertices $b_1$ and $\overline{x}$ are nonadjacent; this verifies {\it (ii)}.
Note that if $b_1$ and $w$ are adjacent, then the edge $b_1w$ is the unique
edge between $A_0$ and $B_1$, and it is non-removable in
\bipartitebarrier{e}{S} by assumption.
This proves {\it (iii)}.
Finally, consider the case in which $w =$~\vertexa,
where \vertexa~is the end of~$e$ in~$I$.
Observe that the neighbourhood of~$A_0 -$~\vertexa~lies in the set~$B_0$ in
the graph~\bipartitebarrier{e}{S} as well as in~$G$, whence $B_0$ is a barrier of~$G$.
We conclude that $|B_0|=1$, and that \vertexa~is the only vertex of~$A_0$.
Furthermore, the neighbourhood of~$A_1$ lies in~$B_1 \cup B_0$,
and thus $B_1 \cup B_0$ is a nontrivial barrier in
\bipartitebarrier{e}{S} as well as in~$G$, which is absurd.
We conclude that $w$ is distinct from the end~\vertexa~of~$e$;
thus {\it (iv)} holds. This completes the proof of
Lemma~\ref{lem:non-removable-at-I}.
\end{proof}

The above lemma implies that each vertex of~$I$,
except possibly the end~\vertexa~of~$e$, is incident with at least one candidate.
Furthermore, if~\vertexa~has degree three or more in \bipartitebarrier{e}{S}
then~\vertexa~is incident with at least two candidates; and likewise,
if any other vertex of~$I$, say~$w$,
has degree four or more then $w$ is incident with at least two candidates.
We thus have the following corollary which is used in the next section.

\begin{cor}
\label{cor:candidate-set-nonempty}
The candidate set
\candidateset{e}{S} has cardinality at least $|S|-2$.
(In particular,
the set~\candidateset{e}{S} is nonempty.)
Furthermore, if \candidateset{e}{S} is a matching
then each vertex of~$I$ is cubic in~$G$ and $|$\candidateset{e}{S}$| = |S|-2$.
\qed
\end{cor}

As we will see later,
by a result of Carvalho et al. (Corollary~\ref{cor:adjacent-candidates-greater-rank}),
if the candidate set \candidateset{e}{S} is not a matching
then it has a member whose rank is strictly greater than that of~$e$.
For this reason, in the proof of Theorem~\ref{thm:rank-plus-index}, we will
mainly have to deal with the case in which the candidate set is a matching.

\subsubsection{When the candidate set is a matching}
\label{sec:candidate-set-matching}

In this section, we suppose that the candidate set~\candidateset{e}{S} is a matching.
We will make several observations, and these will be useful to us in
Section~\ref{sec:proof-of-rank-plus-index-thm} where the proof of
Theorem~\ref{thm:rank-plus-index} is presented.
For all of the figures in the rest of this paper, the solid vertices are those
which are known to be cubic in the brick~$G$;
the hollow vertices may or may not be cubic.

\smallskip
Since \candidateset{e}{S} is a matching,
Corollary~\ref{cor:candidate-set-nonempty} implies
that every vertex of~$I$ is cubic in~$G$,
as shown in Figure~\ref{fig:candidate-set-matching}.
Furthermore, each of these vertices,
except for the end~\vertexa~of~$e$, is incident with exactly one candidate edge;
in particular, $|$\candidateset{e}{S}$| = |I|-1 = |S|-2$.

\begin{Not}
\label{Not:candidates}
We let $w_1, w_2, \dots, w_k$ denote the vertices of~$I-y$, where
$k:= |S|-2$, and for $1\leq j \leq k$, denote the edge of~\candidateset{e}{S}
incident with~$w_j$ by $f_j$ and its end in~$S$ by $u_j$.
\end{Not}

\vspace*{-0.2in}
\begin{figure}[!ht]
%\smallskip
\centering
\begin{tikzpicture}[scale=1]
\draw (-1.5,1.45)node[nodelabel]{$H(e,S):$};
\draw (6.3,0.35)node[nodelabel]{$S$};
\draw (6.3,2.55)node[nodelabel]{$I$};

\draw (0.5,0.35) -- (0.75,0.9);
\draw (0.5,0.35) -- (1.1,0.85);

\draw (1.5,2.55) -- (1.3,2.1);
\draw (1.5,2.55) -- (1.7,2.1);

\draw (2.25,1.45)node[nodelabel]{$f_1$};
\draw (3.75,1.45)node[nodelabel]{$f_2$};
\draw (5.75,1.45)node[nodelabel]{$f_k$};
\draw (4.5,1.45)node[nodelabel]{$\cdots$};

\draw (0,0) -- (6,0) -- (6,0.7) -- (0,0.7) -- (0,0);

\draw (1,2.2) -- (6,2.2) -- (6,2.9) -- (1,2.9) -- (1,2.2);

\draw (5.5,0.35)node{} -- (5.5,2.55)node[fill=black]{};
\draw (3.5,0.35)node{} -- (3.5,2.55)node[fill=black]{};
\draw (2.5,0.35)node{} -- (2.5,2.55)node[fill=black]{};

\draw (0.5,2.55)node{};
\draw (0.5,0.35)node[fill=black]{};
\draw (1.5,0.35)node{};
\draw (1.5,2.55)node[fill=black]{};

\draw (0.5,-0.3)node[nodelabel]{$b_1$};
\draw (1.5,-0.3)node[nodelabel]{$u_0$};
\draw (2.5,-0.3)node[nodelabel]{$u_1$};
\draw (3.5,-0.3)node[nodelabel]{$u_2$};
\draw (4.5,-0.3)node[nodelabel]{$\cdots$};
\draw (5.5,-0.3)node[nodelabel]{$u_k$};

\draw (0.5,3.2)node[nodelabel]{$\overline{x}$};
\draw (1.5,3.2)node[nodelabel]{$y$};
\draw (2.5,3.2)node[nodelabel]{$w_1$};
\draw (3.5,3.2)node[nodelabel]{$w_2$};
\draw (4.5,3.2)node[nodelabel]{$\cdots$};
\draw (5.5,3.2)node[nodelabel]{$w_k$};

\end{tikzpicture}
\vspace*{-0.1in}
\caption{When \candidateset{e}{S} is a matching}
\label{fig:candidate-set-matching}
%\bigskip
\end{figure}
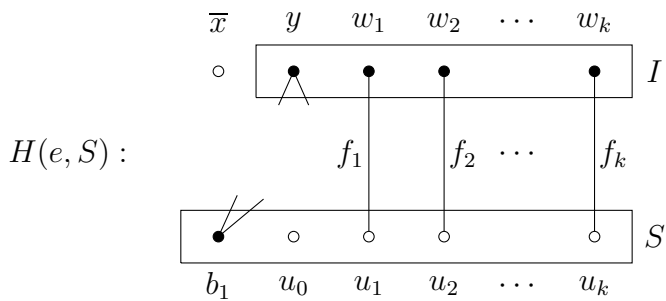

Note that, since \candidateset{e}{S} is a matching, the vertices $u_1, u_2, \dots, u_k$
are distinct, as shown in Figure~\ref{fig:candidate-set-matching}.
Since every vertex of~$I$ is incident with two non-removable edges of
\bipartitebarrier{e}{S},
we deduce the following by assertions {\it (i)}, {\it (ii)} and {\it (iii)} of
Lemma~\ref{lem:non-removable-at-I}, respectively:

\begin{enumerate}[(1)]
\item an end of~$\beta$ lies in~$S$; adjust notation so that $b_1 \in S$,
\item in \bipartitebarrier{e}{S},
vertices $b_1$~and~$\overline{x}$ are nonadjacent;
consequently, in~$G$, all neighbours of~$b_1$, except~$b_2$,
lie in $I$, and
\item $b_1$ is distinct from
each of $u_1, u_2, \dots, u_k$.
\end{enumerate}

Furthermore, since $b_1$ is not incident with any member of \candidateset{e}{S},
Lemma~\ref{lem:non-removable-at-S} implies that it
has precisely two neighbours in~$I$; in particular, $b_1$ is cubic in~$G$.

\begin{Not}
\label{Not:last-vertex-of-barrier}
We let $u_0$ denote the vertex of~$S$ which is distinct
from $b_1, u_1, u_2, \dots, u_k$.
That is, $S = \{b_1, u_0, u_1, u_2, \dots, u_k\}$.
(See Figure~\ref{fig:candidate-set-matching}.)
\end{Not}

As the vertex $u_0$ is not incident with any candidate,
we conclude using Lemma~\ref{lem:non-removable-at-S} that
$u_0$ has at most one neighbour in~$I$.
Observe that if $u_0$ has no neighbours in~$I$ then $(S - u_0) \cup \{z\}$
is a barrier of~$G$ (where~\vertexb~is the end of~$e$ which is not in~$I$), which is
absurd as $G$ is a brick.
Thus, $u_0$ has precisely one neighbour in~$I$.

\smallskip
We note that if~\vertexa~is the unique neighbour of~$u_0$ in the set~$I$,
then $S - u_0$ is a barrier of~$G$, which leads us to the same contradiction as before.
We thus conclude that $u_0$
has precisely one neighbour in the set~$I-y$, and that its remaining neighbours
lie in~$\overline{X}$; see Figure~\ref{fig:candidate-set-matching-more-information}.
In particular, in~\bipartitebarrier{e}{S}, there are are least two edges between
$u_0$~and~$\overline{x}$.

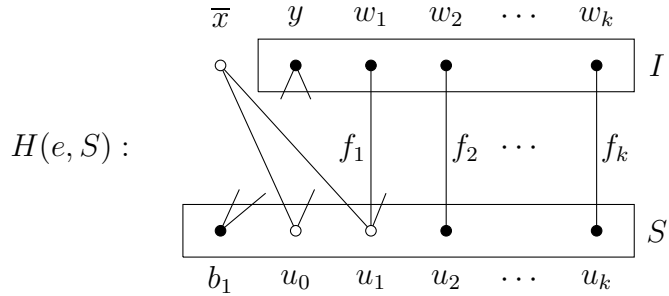
\begin{figure}[!ht]
%\bigskip
\centering
\begin{tikzpicture}[scale=1]
\draw (-1.5,1.45)node[nodelabel]{$H(e,S):$};
\draw (6.3,0.35)node[nodelabel]{$S$};
\draw (6.3,2.55)node[nodelabel]{$I$};

\draw (0.5,0.35) -- (0.75,0.9);
\draw (0.5,0.35) -- (1.1,0.85);

%multiple edges
\draw (1.5,0.35) -- (0.5,2.55);
%\draw (1.5,0.35) to [out=100,in=300] (0.5,2.55);
%\draw (1.5,0.35) to [out=120,in=280] (0.5,2.55);
\draw (1.5,0.35) -- (1.75,0.9);

\draw (2.5,0.35) -- (0.5,2.55);
\draw (2.5,0.35) -- (2.7,0.85);

\draw (1.5,2.55) -- (1.3,2.1);
\draw (1.5,2.55) -- (1.7,2.1);

\draw (2.25,1.45)node[nodelabel]{$f_1$};
\draw (3.75,1.45)node[nodelabel]{$f_2$};
\draw (5.75,1.45)node[nodelabel]{$f_k$};
\draw (4.5,1.45)node[nodelabel]{$\cdots$};

\draw (0,0) -- (6,0) -- (6,0.7) -- (0,0.7) -- (0,0);

\draw (1,2.2) -- (6,2.2) -- (6,2.9) -- (1,2.9) -- (1,2.2);

\draw (5.5,0.35)node[fill=black]{} -- (5.5,2.55)node[fill=black]{};
\draw (3.5,0.35)node[fill=black]{} -- (3.5,2.55)node[fill=black]{};
\draw (2.5,0.35)node{} -- (2.5,2.55)node[fill=black]{};

\draw (0.5,2.55)node{};
\draw (0.5,0.35)node[fill=black]{};
\draw (1.5,0.35)node{};
\draw (1.5,2.55)node[fill=black]{};

\draw (0.5,-0.3)node[nodelabel]{$b_1$};
\draw (1.5,-0.3)node[nodelabel]{$u_0$};
\draw (2.5,-0.3)node[nodelabel]{$u_1$};
\draw (3.5,-0.3)node[nodelabel]{$u_2$};
\draw (4.5,-0.3)node[nodelabel]{$\cdots$};
\draw (5.5,-0.3)node[nodelabel]{$u_k$};

\draw (0.5,3.2)node[nodelabel]{$\overline{x}$};
\draw (1.5,3.2)node[nodelabel]{$y$};
\draw (2.5,3.2)node[nodelabel]{$w_1$};
\draw (3.5,3.2)node[nodelabel]{$w_2$};
\draw (4.5,3.2)node[nodelabel]{$\cdots$};
\draw (5.5,3.2)node[nodelabel]{$w_k$};

\end{tikzpicture}
\vspace*{-0.1in}
\caption{$u_0$ and $u_1$ are the only vertices
adjacent with the contraction vertex~$\overline{x}$}
\label{fig:candidate-set-matching-more-information}
%\bigskip
\end{figure}

Finally, since each vertex $u_j$ in the set~$\{u_1, u_2, \dots, u_k\}$ is incident
with exactly one candidate,
Lemma~\ref{lem:non-removable-at-S} implies that
$u_j$ must satisfy one of the following conditions:

\begin{enumerate}[(i)]
\item either $u_j$ has some neighbour in the set~$\overline{X}$ and it has precisely two
neighbours in the set~$I$,
\item or otherwise, $u_j$ has no neighbours in the set~$\overline{X}$ and it has
precisely three neighbours in the set~$I$.
\end{enumerate}

Observe, by counting degrees of the vertices in~$I$,
that there are precisely $3k+2$ edges with one end in~$I$ and the other
end in~$S$. Of these $3k+2$ edges, precisely two are incident with~$b_1$, and
only one is incident with~$u_0$. Thus there are $3k-1$ edges with one end
in~$I$ and the other end in~$\{u_1, u_2, \dots, u_k\}$. It follows immediately
that exactly one vertex among $u_1, u_2, \dots, u_k$ satisfies condition~(i);
every other vertex satisifes condition (ii).

\begin{Not}
\label{Not:the-special-candidate}
We adjust notation so that $u_1$ is the only vertex
in $\{u_1, u_2, \dots, u_k\}$ which
has neighbours in~$\overline{X}$.
(See Figure~\ref{fig:candidate-set-matching-more-information}.)
\end{Not}

Adopting the notation introduced thus far, the next proposition
summarizes our observations in terms of the brick~$G$.

\begin{prop}
{\sc [When the Candidate Set is a Matching]}
\label{prop:candidate-set-matching}
The following hold:
\begin{enumerate}[(i)]
\item each vertex of~$I$ is cubic,
\item $b_1$ is cubic and its neighbours lie in~$I \cup \{b_2\}$,
\item $u_0$ has precisely one neighbour in~$I - y$,
and all of its remaining neighbours lie in~$\overline{X}$,
\item $u_1$ has precisely two neighbours in~$I$,
and all of its remaining neighbours lie in~$\overline{X}$,
\item if $|S| \geq 4$, then each vertex of~$S - \{b_1,u_0,u_1\}$ has precisely
three neighbours and these neighbours lie in~$I$. \qed
\end{enumerate}
\end{prop}

Observe that, if the barrier~$S$ has precisely three vertices, then the candidate
set \candidateset{e}{S} has only one edge (that is, $f_1=u_1w_1$); in
this case, all of the edges of~$G[X]$ are determined by
Proposition~\ref{prop:candidate-set-matching}, as listed below,
and as shown
in Figure~\ref{fig:candidate-set-matching-barrier-size-three}.
(Note that the underlying simple graph of \bipartitebarrier{e}{S}
is a ladder of order six whose cubic vertices are $u_1$~and~$w_1$.)

\begin{rem}
\label{rem:candidate-set-matching-barrier-size-three}
Suppose that $|S|=3$. Then the following hold:
\begin{enumerate}[(i)]
\item the three neighbours of~$b_1$ are~\vertexa$,w_1$ and $b_2$, 
\item $u_0$ is adjacent with $w_1$, and all of its remaining neighbours lie
in~$\overline{X}$,
\item $u_1$ is adjacent with~\vertexa~and with $w_1$, and all of its remaining
neighbours lie in~$\overline{X}$.
\end{enumerate}
\end{rem}

\vspace*{-0.2in}
\begin{figure}[!ht]
%\smallskip
\centering
\begin{tikzpicture}[scale=1]
\draw (-1.5,1.45)node[nodelabel]{$H(e,S):$};
\draw (3.3,0.35)node[nodelabel]{$S$};
\draw (3.3,2.55)node[nodelabel]{$I$};

%multiple edges
\draw (1.5,0.35) -- (0.5,2.55);

\draw (2.5,0.35) -- (0.5,2.55);
\draw (2.5,0.35) -- (1.5,2.55);

\draw (0.5,0.35) -- (1.5,2.55);
\draw (0.5,0.35) -- (2.5,2.55);

\draw (1.5,0.35) -- (2.5,2.55);

\draw (2.75,1.45)node[nodelabel]{$f_1$};

\draw (0,0) -- (3,0) -- (3,0.7) -- (0,0.7) -- (0,0);

\draw (1,2.2) -- (3,2.2) -- (3,2.9) -- (1,2.9) -- (1,2.2);

\draw (2.5,0.35)node{} -- (2.5,2.55)node[fill=black]{};

\draw (0.5,2.55)node{};
\draw (0.5,0.35)node[fill=black]{};
\draw (1.5,0.35)node{};
\draw (1.5,2.55)node[fill=black]{};

\draw (0.5,-0.3)node[nodelabel]{$b_1$};
\draw (1.5,-0.3)node[nodelabel]{$u_0$};
\draw (2.5,-0.3)node[nodelabel]{$u_1$};

\draw (0.5,3.2)node[nodelabel]{$\overline{x}$};
\draw (1.5,3.2)node[nodelabel]{$y$};
\draw (2.5,3.2)node[nodelabel]{$w_1$};

\end{tikzpicture}
\caption{When \candidateset{e}{S} is a matching, and $S$ has only three vertices}
\label{fig:candidate-set-matching-barrier-size-three}
%\bigskip
\end{figure}
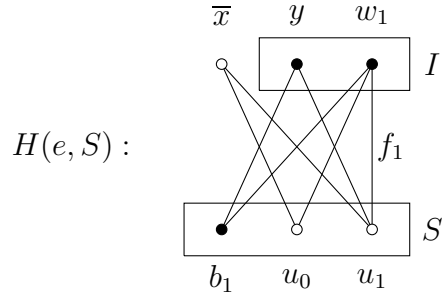

We shall now consider the situation in which $|S| \geq 4$, that is, $k \geq 2$.
Note that, as per our notation, $f_1=u_1w_1$ is the only candidate whose end
in~$S$ (that is, $u_1$) has a neighbour in~$\overline{X}$. In this
sense, $f_1$ is different from the remaining candidates $f_2, f_3, \dots, f_k$.
In the following proposition, we first show that $b_1$
is nonadjacent with the end~$w_1$ of~$f_1$.
Consequently, $b_1$ is adjacent with at least one of
$w_2, w_3, \dots, w_k$; we shall
assume without loss of generality that $b_1$ is adjacent with~$w_2$,
as shown in Figure~\ref{fig:candidate-set-matching-barrier-size-four-or-more}.
In its proof,
we will apply the Lov{\'a}sz-Vempala Lemma (\ref{lem:lovasz-vempala})
to the graph \bipartitebarrier{e}{S}, first at~$w_1$, and then at~$w_2$;
each of these applications is a refinement of the situation in
Lemma~\ref{lem:non-removable-at-I}.

\begin{prop}
\label{prop:candidate-set-matching-barrier-size-four-or-more}
Suppose that $|S| \geq 4$. Then the following hold:
\begin{enumerate}[(i)]
\item $b_1$ and $w_1$ are nonadjacent; adjust notation so that
$b_1w_2$ is an edge of~$G$,
\item \vertexa\ is adjacent with each of $b_1$ and $u_2$, and
\item $u_0$ and $w_2$ are nonadjacent.
\end{enumerate}
\end{prop}
\begin{proof}
First, we apply Lemma~\ref{lem:lovasz-vempala} to the graph
\bipartitebarrier{e}{S} at vertex~$w_1$. Since $f_1=u_1w_1$ is the only
removable edge incident with~$w_1$, there exist
partitions~$(A_0,A_1,A_2)$ of $I \cup \{\overline{x}\}$, and $(B_0,B_1,B_2)$
of~$S$, such that $w_1 \in A_0$, and $|A_j| = |B_j|$ for $j \in \{0,1,2\}$,
vertex~$u_1$ lies in~$B_0$, and the remaining two neighbours of $w_1$
lie in~$B_1$ and in~$B_2$, respectively. Furthermore, $N(B_1) = A_1 \cup \{w_1\}$
and $N(B_2) = A_2 \cup \{w_1\}$.

\smallskip
Suppose that $b_1$ is a neighbour of~$w_1$, and adjust notation so that
$b_1 \in B_1$. The contraction vertex~$\overline{x}$ lies in~$A_2$,
since otherwise $A_2 \cup \{w_1\}$ is a nontrivial barrier in~$G$. We
will deduce that each of the sets $B_0, B_1$ and $B_2$ is a singleton, and thus
the barrier~$S$ has precisely three vertices, contrary to the hypothesis.

\smallskip
First of all, note that the neighbourhood of~$B_1 - b_1$ is contained in $A_1$, and thus
if $|A_1| \geq 2$ then $A_1$ is a nontrivial barrier in~$G$; we conclude that
$|A_1|=1$ and that $B_1 = \{b_1\}$. 
Observe that the contraction vertex~$\overline{x}$ is only adjacent with $u_1$,
which lies in~$B_0$, and with~$u_0$.
Thus the neighbourhood
of~$B_2 - u_0$ is contained in $(A_2 - \overline{x}) \cup \{w_1\}$,
whence the latter is a barrier of~$G$;
we infer that
$A_2 = \{\overline{x}\}$; consequently, the unique vertex of~$B_2$ has
precisely two neighbours, namely $w_1$ and $\overline{x}$.
It follows that $B_2 = \{u_0\}$. Since the vertex~$w_1$ is cubic,
the neighbourhood of $B_0 - u_1$ is contained in~$(A_0 - w_1) \cup A_1$,
whence the latter is a barrier of~$G$;
we infer that $A_0 = \{w_1\}$, thus $B_0 = \{u_1\}$.
It follows that $|S|=3$, contrary to our hypothesis. Thus
$b_1$ and $w_1$ are nonadjacent; this proves {\it (i)}. As in the statement
of the proposition, adjust notation so that $b_1$~and~$w_2$ are adjacent;
see Figure~\ref{fig:candidate-set-matching-barrier-size-four-or-more}.

\smallskip
To deduce {\it (ii)} and {\it (iii)}, we apply Lemma~\ref{lem:lovasz-vempala}
to the graph \bipartitebarrier{e}{S} at vertex~$w_2$.
Similar to the earlier situation, there exist partitions $(A_0, A_1, A_2)$
of~$I \cup \{\overline{x}\}$, and $(B_0,B_1,B_2)$ of~$S$, such that
$w_2 \in A_0$, and $|A_j| = |B_j|$ for $j \in \{1,2,3\}$, vertex $u_2$ lies
in~$B_0$, and the remaining two neighbours of~$w_2$ lie in~$B_1$ and in~$B_2$,
respectively.
Adjust notation so that $b_1$ lies in~$B_1$.
Also, $N(B_1) = A_1 \cup \{w_2\}$ and $N(B_2) = A_2 \cup \{w_2\}$.
As before, we conclude that
$\overline{x}$ lies in $A_2$, and that $|A_1| = |B_1| = 1$.

\smallskip
Observe that the unique vertex of~$A_1$ has all of its neighbours in the
set~$B_0 \cup B_1$. We will show that
$B_0 = \{u_2\}$; this implies that the unique vertex of~$A_1$
has precisely two neighbours, and so it must be the end~\vertexa~of~$e$;
this immediately implies {\it (ii)}.

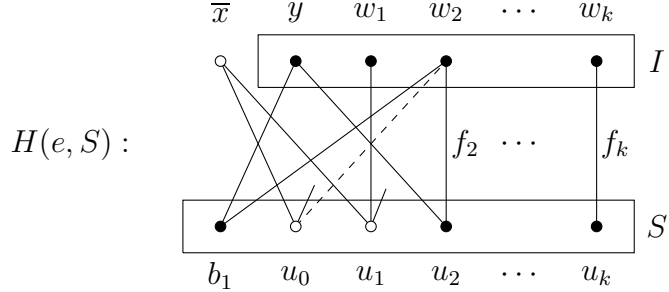
\begin{figure}[!ht]
%\bigskip
\centering
\begin{tikzpicture}[scale=1]
\draw (-1.5,1.45)node[nodelabel]{$H(e,S):$};
\draw (6.3,0.35)node[nodelabel]{$S$};
\draw (6.3,2.55)node[nodelabel]{$I$};

\draw[dashed,thin] (1.5,0.35) -- (3.5,2.55);

\draw (0.5,0.35) -- (1.5,2.55);
\draw (0.5,0.35) -- (3.5,2.55);

%multiple edges
\draw (1.5,0.35) -- (0.5,2.55);
%\draw (1.5,0.35) to [out=100,in=300] (0.5,2.55);
%\draw (1.5,0.35) to [out=120,in=280] (0.5,2.55);
\draw (1.5,0.35) -- (1.75,0.9);

\draw (2.5,0.35) -- (0.5,2.55);
\draw (2.5,0.35) -- (2.7,0.85);

\draw (1.5,2.55) -- (3.5,0.35);

%\draw (2.7,1.45)node[nodelabel]{$f_1$};
\draw (3.75,1.45)node[nodelabel]{$f_2$};
\draw (5.75,1.45)node[nodelabel]{$f_k$};
\draw (4.5,1.45)node[nodelabel]{$\cdots$};

\draw (0,0) -- (6,0) -- (6,0.7) -- (0,0.7) -- (0,0);

\draw (1,2.2) -- (6,2.2) -- (6,2.9) -- (1,2.9) -- (1,2.2);

\draw (5.5,0.35)node[fill=black]{} -- (5.5,2.55)node[fill=black]{};
\draw (3.5,0.35)node[fill=black]{} -- (3.5,2.55)node[fill=black]{};
\draw (2.5,0.35)node{} -- (2.5,2.55)node[fill=black]{};

\draw (0.5,2.55)node{};
\draw (0.5,0.35)node[fill=black]{};
\draw (1.5,0.35)node{};
\draw (1.5,2.55)node[fill=black]{};

\draw (0.5,-0.3)node[nodelabel]{$b_1$};
\draw (1.5,-0.3)node[nodelabel]{$u_0$};
\draw (2.5,-0.3)node[nodelabel]{$u_1$};
\draw (3.5,-0.3)node[nodelabel]{$u_2$};
\draw (4.5,-0.3)node[nodelabel]{$\cdots$};
\draw (5.5,-0.3)node[nodelabel]{$u_k$};

\draw (0.5,3.2)node[nodelabel]{$\overline{x}$};
\draw (1.5,3.2)node[nodelabel]{$y$};
\draw (2.5,3.2)node[nodelabel]{$w_1$};
\draw (3.5,3.2)node[nodelabel]{$w_2$};
\draw (4.5,3.2)node[nodelabel]{$\cdots$};
\draw (5.5,3.2)node[nodelabel]{$w_k$};

\end{tikzpicture}
\vspace*{-0.2in}
\caption{When \candidateset{e}{S} is a matching, and $S$ has four or more vertices;
the vertices $u_0$~and~$w_2$ are nonadjacent}
\label{fig:candidate-set-matching-barrier-size-four-or-more}
%\bigskip
\end{figure}

%\vspace*{-0.2in}
Note that the neighbourhood of~$A_0 - w_2$ is contained in~$B_0$.
Thus, if $|A_0| \geq 2$ then~\vertexa~lies in~$A_0$ (since otherwise
$B_0$ is a barrier of~$G$).
If $|A_0| \geq 3$ then $B_0$ is a barrier
of~$G-e$ with three or more vertices. (Note that the barrier~$B_0$ is contained
in the barrier~$S$.) Since no end of~$\beta$ lies in~$B_0$, 
it follows from our earlier observations that the candidate set \candidateset{e}{B_0}
is not a matching. However, by Corollary~\ref{cor:candidate-containment-property},
\candidateset{e}{B_0} is a subset of \candidateset{e}{S}, and the latter is a matching; this
is absurd. We conclude that $A_0$ has at most two vertices, that is,
either $A_0 = \{w_2\}$ or $A_0 = \{$\vertexa$,w_2\}$.
Now suppose that $A_0 = \{$\vertexa$,w_2\}$. The unique vertex of~$A_1$ is adjacent with
$b_1$, and thus statement {\it (i)} implies that $w_1 \notin A_1$.
Assume without loss of generality
that $A_1 = \{w_3\}$. Since $w_3$ is cubic, we conclude that its neighbourhood
is precisely $B_0 \cup B_1$, and thus $B_0 = \{u_2,u_3\}$.
Observe that $Q:=w_3u_2w_2b_1w_3$ is a $4$-cycle in
\bipartitebarrier{e}{S} containing
the vertex~$w_3$, and thus by Corollary~\ref{cor:quadrilateral-LV}, one of the
edges $w_3u_2$ and $w_3b_1$ is removable in \bipartitebarrier{e}{S}; however,
this contradicts our hypothesis since
the only removable edges are the members of \candidateset{e}{S}.
We thus conclude that $A_0 = \{w_2\}$. As explained earlier,
$A_1 = \{$\vertexa$\}$, and thus~\vertexa~is adjacent with each of $b_1$ and $u_2$;
this proves {\it (ii)}.

\smallskip
Now suppose that $u_0$ and $w_2$ are adjacent. Observe that $u_1 \in B_2$,
and thus all of its neighbours lie in~$A_2$, whence $|A_2| \geq 3$.
The neighbourhood of $B_2 - \{u_0,u_1\}$ is contained in~$A_2 - \overline{x}$,
whence the latter is a nontrivial barrier of~$G$, which is a contradiction.
We conclude that
$u_0$ and $w_2$ are nonadjacent; this proves {\it (iii)}, and completes
the proof of Proposition~\ref{prop:candidate-set-matching-barrier-size-four-or-more}.
\end{proof}

\subsection{The Equal Rank Lemma}
\label{sec:equal-rank-lemma}

Here, we present an important lemma which is used in the proof of
Theorem~\ref{thm:rank-plus-index}.
This lemma considers the situation in which $G$ is an \Rbrick\ and
$e:=$~\vertexa\vertexb~is an \Rcomp\ edge of index two that is not thin,
and $f$ is a candidate relative to a barrier
of~$G-e$ such that $f$ is also of index two and its rank is equal to that of~$e$.
The reader is advised to review the Three Case Lemma (\ref{lem:three-case})
and Section~\ref{sec:index-two}
before proceeding further.

\smallskip
The Equal Rank Lemma (\ref{lem:equal-rank}) relates the barrier structure
of~$G-f$ to that of~$G-e$.
More specifically, the lemma establishes subset/superset relationships between
eight sets of vertices: the barriers
$S_1$ and $S_2$ of~$G-e$ (as in Case 2 of Lemma~\ref{lem:three-case})
and their corresponding sets of isolated vertices $I_1$ and $I_2$, and likewise,
the barriers $S_3$ and $S_4$ of~$G-f$ and their corresponding sets
of isolated vertices $I_3$ and $I_4$. Among other things, the lemma
shows that $S_1 \cup I_1 \cup S_2 \cup I_2 = S_3 \cup I_3 \cup S_4 \cup I_4$.
We now introduce the relevant notation more precisely.

\smallskip
Since $e$ is of index two, by the Three Case Lemma, $G-e$ has precisely two maximal
nontrivial barriers, and since $e$ is not thin, at least one of these barriers, say~$S_1$,
has three or more vertices (see Proposition~\ref{prop:equivalent-definition-Rthin}).
We adopt Notation~\ref{Not:Rcompatible-in-Rbrick} for the brick~$G$ and edge~$e$.
Assume without loss of generality that $S_1 \subset B$, and let
$I_1$ denote the set of isolated vertices of~$(G-e)-S_1$.
We shall denote by $S_2$ the maximal nontrivial barrier of~$(G-e)/X_1$ where
$X_1 := S_1 \cup I_1$, and by $I_2$ the set of isolated vertices of~$(G-e)-S_2$.
Note that the end~\vertexb~of~$e$ lies in~$I_2$ which is a subset of~$B$,
whereas the other end~\vertexa~of~$e$ lies in~$I_1$ which is a subset of~$A$.
See Figure~\ref{fig:equal-rank} (top).

\smallskip
By Corollary~\ref{cor:candidate-set-nonempty}, the candidate set \candidateset{e}{S_1}
is nonempty, and by Proposition~\ref{prop:edges-of-candidate-set},
each of its members is an \Rcomp\ edge whose rank is at
least that of~$e$.
Now, let $f:=uw$ be a member of \candidateset{e}{S_1}
such that $u \in S_1$ and $w \in I_1$, and suppose that
the index of~$f$ is two.
The following result of Carvalho et al. \cite[Lemma 32]{clm06}
plays a crucial role in our proof of the Equal Rank Lemma (\ref{lem:equal-rank}).

\begin{lem}
\label{lem:equal-rank-barrier-containment}
Assume that ${\sf index}(e)={\sf index}(f)=2$.
If ${\sf rank}(e)={\sf rank}(f)$ then
$S_2$ is a subset of a barrier of~$G-f$. \qed
\end{lem}

We shall let $S_3$ denote the maximal nontrivial barrier of~$G-f$ which is contained
in the color class~$B$, and $I_3$ the set of isolated vertices of~$(G-f)-S_3$.
Furthermore, let $S_4$ denote the maximal nontrivial barrier of~$(G-f)/(S_3 \cup I_3)$, and $I_4$ the set of isolated vertices of~$(G-f)-S_4$.
Note that the end~$u$ of~$f$ lies in~$I_4$, and its other end~$w$ lies in~$I_3$.
See Figure~\ref{fig:equal-rank} (bottom).
We are now ready to state the Equal Rank Lemma using the notation introduced so far.

%\vspace*{-0.4in}
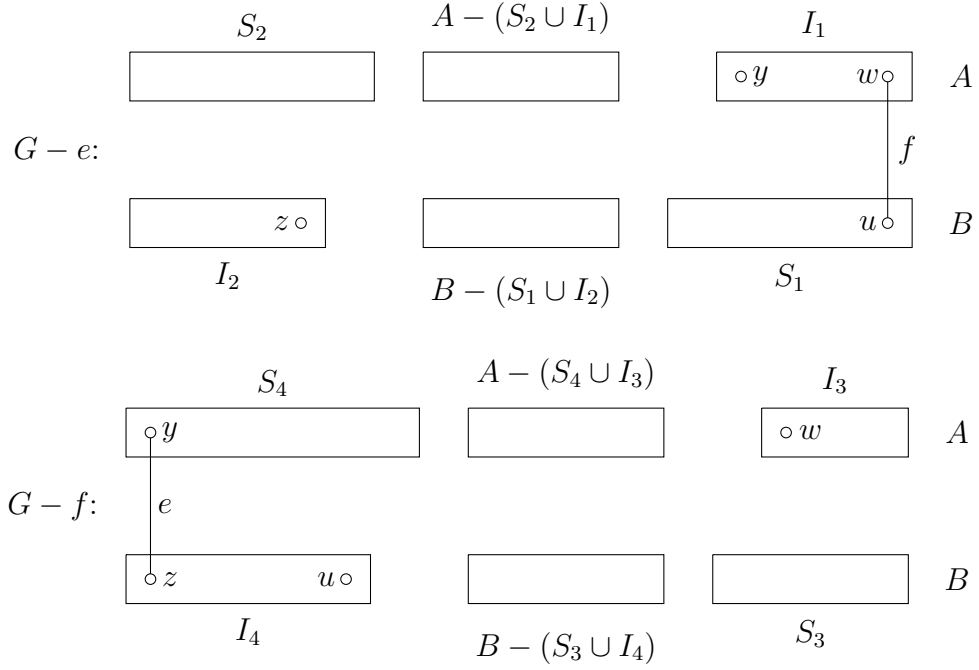
\begin{figure}[!ht]
\centering
\begin{tikzpicture}[scale=0.65]
\draw (-2.5,-2) node[nodelabel]{$G-e$:};

\draw (16,-0.5) node[nodelabel]{$A$};
\draw (16,-3.5) node[nodelabel]{$B$};

\draw (-1,0) -- (-1,-1) -- (4,-1) -- (4,0) -- (-1,0);
\draw (1.5,-0.3) node[above,nodelabel]{$S_2$};

\draw (-1,-3) -- (-1,-4) -- (3,-4) -- (3,-3) -- (-1,-3);
\draw (1,-3.8) node[below,nodelabel]{$I_2$};

\draw (5,0) -- (5,-1) -- (9,-1) -- (9,0) -- (5,0);
\draw (7,-1.5) node[above,nodelabel]{$A-(S_2 \cup I_1)$};

\draw (5,-3) -- (5,-4) -- (9,-4) -- (9,-3) -- (5,-3);
\draw (7,-2.7) node[below,nodelabel]{$B-(S_1 \cup I_2)$};

\draw (11,0) -- (11,-1) -- (15,-1) -- (15,0) -- (11,0);
\draw (13,-0.2) node[above,nodelabel]{$I_1$};

\draw (10,-3) -- (10,-4) -- (15,-4) -- (15,-3) -- (10,-3);
\draw (12.5,-3.8) node[below,nodelabel]{$S_1$};

\draw (2.5,-3.5) node{};
\draw (2.1,-3.5) node[nodelabel]{\vertexb};

\draw (11.5,-0.5) node{};
\draw (11.9,-0.5) node[nodelabel]{\vertexa};

\draw (14.5,-0.5) -- (14.5,-3.5);
\draw (14.5,-0.5) node{};
\draw (14.1,-0.5) node[nodelabel]{$w$};
\draw (14.5,-3.5) node{};
\draw (14.1,-3.5) node[nodelabel]{$u$};
\draw (14.9,-2) node[nodelabel]{$f$};
\end{tikzpicture}

\vspace*{-0.7in}
\begin{tikzpicture}[scale=0.65]
\draw (-2.5,-2) node[nodelabel]{$G-f$:};

\draw (16,-0.5) node[nodelabel]{$A$};
\draw (16,-3.5) node[nodelabel]{$B$};

\draw (-1,0) -- (-1,-1) -- (5,-1) -- (5,0) -- (-1,0);
\draw (2,-0.3) node[above,nodelabel]{$S_4$};

\draw (-1,-3) -- (-1,-4) -- (4,-4) -- (4,-3) -- (-1,-3);
\draw (1.5,-3.8) node[below,nodelabel]{$I_4$};

\draw (6,0) -- (6,-1) -- (10,-1) -- (10,0) -- (6,0);
\draw (8,-1.5) node[above,nodelabel]{$A-(S_4 \cup I_3)$};

\draw (6,-3) -- (6,-4) -- (10,-4) -- (10,-3) -- (6,-3);
\draw (8,-2.7) node[below,nodelabel]{$B-(S_3 \cup I_4)$};

\draw (12,0) -- (12,-1) -- (15,-1) -- (15,0) -- (12,0);
\draw (13.5,-0.2) node[above,nodelabel]{$I_3$};

\draw (11,-3) -- (11,-4) -- (15,-4) -- (15,-3) -- (11,-3);
\draw (13,-3.8) node[below,nodelabel]{$S_3$};

\draw (3.5,-3.5) node{};
\draw (3.1,-3.5) node[nodelabel]{$u$};

\draw (12.5,-0.5) node{};
\draw (13,-0.5) node[nodelabel]{$w$};

\draw (-0.5,-0.5) -- (-0.5,-3.5);
\draw (-0.5,-0.5) node{};
\draw (-0.1,-0.5) node[nodelabel]{\vertexa};
\draw (-0.5,-3.5) node{};
\draw (-0.1,-3.5) node[nodelabel]{\vertexb};
\draw (-0.2,-2) node[nodelabel]{$e$};
\end{tikzpicture}
\vspace*{-0.4in}
\caption{The Equal Rank Lemma}
\label{fig:equal-rank}
%\medskip
\end{figure}

%\vspace*{-0.1in}
\begin{lem}
{\sc [The Equal Rank Lemma]}
\label{lem:equal-rank}
Assume that \mbox{${\sf index}(e)={\sf index}(f)=2$}.
If ${\sf rank}(e) = {\sf rank}(f)$
then the following statements hold:
\begin{enumerate}[{\it (i)}]
\item $e$~and~$f$ are nonadjacent,
\item $S_3 \subseteq S_1-u$ and $I_3 \subseteq I_1-y$,
\item $S_2 \subset S_4$ and $I_2 \subset I_4$,
\item $S_1 \cup I_2 = S_3 \cup I_4$ and $S_2 \cup I_1 = S_4 \cup I_3$,
\item $N(u) \subseteq S_2 \cup I_1$, and
\item $e$ is a member of the candidate set \candidateset{f}{S_4}.
\end{enumerate}
\end{lem}
\begin{proof}
We examine the graph~$G-e-f$ in order to prove {\it (i)}~and~{\it (ii)}.
Clearly, $S_3$ is a barrier of~$G-e-f$.
Observe that, since $f$ has an end in~$S_1$, every barrier of~$G-e-f$
which contains $S_1$ is a barrier of~$G-e$ as well. Since $S_1$
is a maximal barrier of~$G-e$, we infer that $S_1$ is a maximal barrier
of~$G-e-f$ as well.
By the Canonical Partition Theorem~(\ref{thm:canonical-partition}),
to prove that $S_3$ is a subset of~$S_1$, it
suffices to show that $S_1 \cap S_3$ is nonempty.
To see this, note that $w \in I_1 \cap I_3$, and thus any neighbour of~$w$ in $G-e-f$
lies in $S_1 \cap S_3$. Furthermore, since $u \notin S_3$, we conclude that
$S_3 \subseteq S_1 - u$; this proves part of {\it (ii)}.
In particular, $z \notin S_3$. Consequently, $y \notin I_3$, and thus $y$~and~$w$
are distinct. This proves {\it (i)}.

\smallskip
Now we prove the remaining part of~{\it (ii)}.
Let $v \in I_3$, that is, $v$ is isolated in~$(G-f)-S_3$. Consequently, $v$ is isolated
in~$(G-f)-S_1$. Since $f$ has an end in~$S_1$, we infer that $v$ is isolated
in~$(G-e)-S_1$, that is, $v \in I_1$. Thus $I_3 \subseteq I_1 - y$. This proves~{\it (ii)}.

\smallskip
We will now prove {\it (iii)} and {\it (iv)}. We begin by showing that $S_2$ is a subset
of~$S_4$. By Lemma~\ref{lem:equal-rank-barrier-containment},
$S_2$ is a subset of the unique maximal nontrivial barrier of~$G-f$
which is contained in the color class~$A$, say~$S^*_4$.
By the Three Case Lemma~(\ref{lem:three-case}), $S^*_4 = S_4 \cup I'$ for some
(possibly empty) subset~$I'$ of~$I_3$. That is, $S_2$ is a subset of~$S_4 \cup I'$.
Note that $S_2$~and~$I_1$ are disjoint; by {\it (ii)}, $S_2 \cap I' = \emptyset$.
Thus, $S_2 \subseteq S_4$.

\smallskip
Since the ranks of $e$ and $f$ are equal, it follows that $|A-(S_2 \cup I_1)| = |A-(S_4 \cup I_3)|$
and likewise, $|B-(S_1 \cup I_2)| = |B-(S_3 \cup I_4)|$. In order to prove {\it (iv)}, it suffices
to prove the following claim.

\begin{claim}
\label{claim:core-subset}
$A-(S_2 \cup I_1) \subseteq A-(S_4 \cup I_3)$
and $B-(S_1 \cup I_2) \subseteq B-(S_3 \cup I_4)$.
\end{claim}
\begin{proof}
Let $v_1 \in A-(S_2 \cup I_1)$. By {\it (ii)}, $v_1 \notin I_3$.
To prove that $v_1$ lies in $A-(S_4 \cup I_3)$, it suffices
to show that $v_1 \notin S_4$.

\smallskip
Now, let $v_2$ be any vertex in~$S_2$. We have already
shown that $S_2 \subseteq S_4$, and thus $v_2 \in S_4$.
Note that, if $v_1$ also belongs to the barrier~$S_4$,
then $(G-f)-\{v_1,v_2\}$ would not have a perfect matching.
In the following paragraph, we will show that $(G-e-f) - \{v_1,v_2\}$
has a perfect matching, say~$M$; consequently, $v_1 \notin S_4$.

\smallskip
Let $H_1$ be the graph~$(G-e-f)/\overline{X_1} \rightarrow \overline{x_1}$,
and let $H_2$ be the graph $(G-e-f)/\overline{X_2} \rightarrow \overline{x_2}$
where $X_2 := S_2 \cup I_2$.
Note that $H_1$~and~$H_2$ are bipartite matching covered graphs.
Let $J:=((G-e-f)/X_1 \rightarrow x_1)/X_2 \rightarrow x_2$.
Note that $J$ is the brick of~$G-e-f$.
Let $M_J$ be a perfect matching of~$J-\{x_2,v_1\}$.
Let $g$ denote the edge of $M_J$ incident with the contraction vertex~$x_1$.
Let $M_1$ be a perfect matching of~$H_1$ which contains~$g$.
Let $M_2$ be a perfect matching of~$H_2-\{v_2,\overline{x_2}\}$.
Observe that $M:=M_1 + M_J +M_2$ is the desired matching.

\smallskip
Now, let $v \in B-(S_1 \cup I_2)$.
By {\it (ii)}, $v \notin S_3$.
To prove that $v$ lies in $B-(S_3 \cup I_4)$, it suffices
to show that $v \notin I_4$.
To see this, note that since $J$ is a brick,
by Theorem~\ref{thm:elp-bricks}, $J-\{x_1,x_2\}$
is connected; thus, $v$ is not isolated in $(G-f)-S_4$, that is,
$v \notin I_4$.
\end{proof}

It follows from {\it (ii)} and {\it (iv)} that the end~$y$ of~$e$
lies in~$S_4$, and thus $S_2$ is a proper subset of~$S_4$.
Also, we infer from {\it (ii)} and {\it (iv)} that $I_2$ is a subset
of~$I_4$. Furthermore, the end~$u$ of~$f$ lies in~$I_4$,
whence $I_2$ is a proper subset of~$I_4$. This proves~{\it (iii)}.

\smallskip
It remains to prove {\it (v)} and {\it (vi)}.
As noted above, $u \in I_4$.
Thus, all neighbors of~$u$ in~$G$ lie in~$S_4 \cup \{w\} \subseteq S_4 \cup I_3$.
It follows from {\it (iv)} that $N(u) \subseteq S_2 \cup I_1$.
This proves~{\it (v)}.

\smallskip
Finally, we prove~{\it (vi)}.
Recall that \bipartitebarrier{f}{S_4} denotes the bipartite matching covered graph
$(H-f)/ \overline{X_4} \rightarrow \overline{x_4}$ where $X_4:=S_4 \cup I_4$,
and that \candidateset{f}{S_4} is the set of those removable edges
of \bipartitebarrier{f}{S_4} which are not incident with the contraction
vertex~$\overline{x_4}$. Since $f$ is \Rcomp\ in~$G-e$
(by Proposition~\ref{prop:edges-of-candidate-set}), the exchange property
(Proposition~\ref{prop:exchange-property-Rcompatible-nbmcg})
implies that $e$ is \Rcomp\ in $G-f$. Now,
since the end~\vertexb~of~$e$ lies in~$I_4$,
the last assertion of Proposition~\ref{prop:edges-of-candidate-set}
implies that $e$ is a member of \candidateset{f}{S_4}.
This proves~{\it (vi)}, and finishes the proof of the
Equal Rank Lemma.
\end{proof}

\subsection{Proof of Theorem~\ref{thm:rank-plus-index}}
\label{sec:proof-of-rank-plus-index-thm}

Before we proceed to prove Theorem~\ref{thm:rank-plus-index},
we state two results of Carvalho et al. \cite{clm06} which are useful to us.
Suppose that $G$ is an \Rbrick\ and $e$ is an \Rcomp\ edge which is not
thin. We let $S_1$ denote a maximal nontrivial barrier of~$G-e$ such that
$|S_1| \geq 3$, and let $f$ denote a member of the candidate set \candidateset{e}{S_1}.

\smallskip
Note that, since $e$ is not thin, its rank is at most $n-4$ where $n:=|V(G)|$.
If the index of $f$ is zero then its rank is~$n$, and in particular,
it is greater than that of~$e$.
The following result of Carvalho et al. \cite[Lemma 31]{clm06} shows that this
conclusion holds even if the index of~$f$ is one.

\begin{lem}
\label{lem:index-one-candidate-greater-rank}
Suppose that $f$ is a member of the candidate set \candidateset{e}{S_1}.
If the index of~$f$ is one then ${\sf rank}(f) > {\sf rank}(e)$. \qed
\end{lem}

The following corollary of
Lemmas~\ref{lem:equal-rank-barrier-containment} and
\ref{lem:index-one-candidate-greater-rank}
was used implicitly by Carvalho et al. \cite{clm06} in their proof
of the Thin Edge Theorem (\ref{thm:clm-thin-bricks}).
We provide its proof for the sake of completeness.

\begin{cor}
\label{cor:adjacent-candidates-greater-rank}
Assume that the index of $e$ is two.
If the candidate set \candidateset{e}{S_1} contains two adjacent edges,
say $f$ and $g$,
then at least one of them has rank strictly greater than ${\sf rank}(e)$.
\end{cor}
\begin{proof}
We know by Proposition~\ref{prop:edges-of-candidate-set} that each of $f$ and $g$
has rank at least ${\sf rank}(e)$.
If either of them has rank strictly greater than that of $e$ then there is nothing to prove.
Now, suppose that ${\sf rank}(f)={\sf rank}(g) = {\sf rank}(e)$.
It follows from Lemma~\ref{lem:index-one-candidate-greater-rank}
that both $f$~and~$g$ are of index two.
We intend to arrive at a contradiction using
Lemma~\ref{lem:equal-rank-barrier-containment}.
We let $I_1$ denote the set of isolated vertices of~$(G-e)-S_1$,
and $S_2$ denote the unique maximal nontrivial barrier of~$(G-e) / (S_1 \cup I_1)$.
By Lemma~\ref{lem:equal-rank-barrier-containment}, $S_2$ is a subset of a barrier
of~$G-f$, and likewise, $S_2$ is a subset of a barrier of~$G-g$.

\smallskip
Consider two distinct vertices of~$S_2$, say $v_1$ and $v_2$.
Let $M$ be a perfect matching of the graph $G-\{v_1,v_2\}$.
(Such a perfect matching exists as $G$ is a brick.)
As noted above, $S_2$ is a subset
of a barrier of~$G-f$. In particular, $v_1$ and $v_2$ lie in a barrier of~$G-f$,
whence $(G-f) - \{v_1,v_2\}$ has no perfect matching.
Thus $f$ lies in~$M$.
Likewise, $g$ also lies in~$M$.
This is absurd since $f$ and $g$ are adjacent.
We conclude that one of $f$ and $g$ has rank strictly greater than ${\sf rank}(e)$.
This completes the proof of Corollary~\ref{cor:adjacent-candidates-greater-rank}.
\end{proof}

We now proceed to prove Theorem~\ref{thm:rank-plus-index}.

\medskip
\noindent
\underline{Proof of Theorem~\ref{thm:rank-plus-index}}:
As in the statement of the theorem,
let $e$ denote an \Rcomp\ edge of an \Rbrick~$G$.
If the edge~$e$ is thin, then there is nothing to prove.
Now consider the case in which $e$ is not thin.
By the Three Case Lemma (\ref{lem:three-case}),
$G-e$ has either one or two maximal nontrivial barriers,
and by Proposition~\ref{prop:equivalent-definition-Rthin},
at least one such barrier has three or more vertices.
Our goal is to establish the existence of another
\Rcomp\ edge~$f$ which satisfies conditions {\it (i)} and {\it (ii)}
in the statement of Theorem~\ref{thm:rank-plus-index}.

\smallskip
Recall that each candidate edge (relative to $e$ and a barrier of~$G-e$
with three or more vertices) is an \Rcomp\ edge of~$G$ which
satisfies condition {\it (i)} of Theorem~\ref{thm:rank-plus-index}
and has rank at least ${\sf rank}(e)$.
(See Definition~\ref{Def:candidate-set} and
Proposition~\ref{prop:edges-of-candidate-set}.)
Furthermore, if a candidate has rank strictly greater than ${\sf rank}(e)$,
then by Proposition~\ref{prop:greater-rank-suffices}, it also satisfies
condition~{\it (ii)} of Theorem~\ref{thm:rank-plus-index},
and in this case we are done.
Keeping these observations in view,
we now use Lemma~\ref{lem:index-one-candidate-greater-rank} to
get rid of the case in which index of~$e$ is one.

\begin{Claim}
\label{claim:assume-index-two}
We may assume that the index of~$e$ is two.
\end{Claim}
\begin{proof}
Suppose not. Then the index of~$e$ is one, and we let $S$ denote the
unique maximal nontrivial barrier of~$G-e$. As discussed earlier, $|S| \geq 3$.
Let $f$ denote a member of the candidate set \candidateset{e}{S},
which is nonempty by Corollary~\ref{cor:candidate-set-nonempty}.
If the index of $f$ is zero then its rank is clearly greater than ${\sf rank}(e)$, and
by Lemma~\ref{lem:index-one-candidate-greater-rank}, this conclusion holds
even if the index of $f$ is one. Now consider the case in which $f$ is of index two.
Since \mbox{${\sf rank}(f) \geq {\sf rank}(e)$}, we conclude that
$f$ satisfies condition~{\it (ii)}, Theorem~\ref{thm:rank-plus-index}.
Thus, irrespective of its index, the edge~$f$ satisfies
both conditions~{\it (i)} and {\it (ii)},
and we are done.
\end{proof}

We shall now invoke Corollary~\ref{cor:adjacent-candidates-greater-rank} to
dispose of the case in which the candidate set (relative to some barrier of~$G-e$)
is not a matching.

\begin{Claim}
\label{claim:assume-candidate-set-matching}
We may assume that if $S$ is a nontrivial barrier
(not necessarily maximal)
of~$G-e$ with three or more
vertices then the corresponding candidate set \candidateset{e}{S} is a matching.
\end{Claim}
\begin{proof}
Suppose that the candidate set \candidateset{e}{S} is not a matching,
and thus it contains two adjacent edges, say $f$ and $g$.
We let~$S^*$ denote the maximal nontrivial barrier of~$G-e$ such that
$S \subseteq S^*$. By Corollary~\ref{cor:candidate-containment-property},
edges $f$ and $g$ are members of \candidateset{e}{S^*} as well.
Since $e$ is of index two (by Claim~\ref{claim:assume-index-two}),
Corollary~\ref{cor:adjacent-candidates-greater-rank} implies
that at least one of $f$ and $g$, say~$f$, has rank strictly greater than
that of~$e$. Thus $f$ satisfies both conditions~{\it (i)} and {\it (ii)},
Theorem~\ref{thm:rank-plus-index}, and we are done.
\end{proof}

Now, since $e$ is of index two (by Claim~\ref{claim:assume-index-two}), the
graph~$G-e$ has precisely two maximal nontrivial barriers. Among these two,
we shall denote by $S_1$ the barrier which is bigger (breaking ties arbitrarily if
they are of equal size), and by $I_1$ the set of isolated vertices of~$(G-e)-S_1$.
Thus $|S_1| \geq 3$.
Let~\vertexa~and~\vertexb~denote the ends of~$e$.
We adopt Notation~\ref{Not:Rcompatible-in-Rbrick}.
Assume without loss of generality that $S_1$ is a subset of~$B$,
and thus by the Three Case Lemma (\ref{lem:three-case}),
the end~\vertexa~of~$e$ lies in~$I_1$.

\smallskip
As the candidate set \candidateset{e}{S_1} is a matching
(by Claim~\ref{claim:assume-candidate-set-matching}), we invoke the observations
made in Section~\ref{sec:candidate-set-matching}, with $S_1$ playing the
role of~$S$, and $I_1$ playing the role of~$I$, and likewise,
$X_1:=S_1 \cup I_1$ playing the role of~$X$.
In particular, we adopt
Notations~\ref{Not:candidates}, \ref{Not:last-vertex-of-barrier} and
\ref{Not:the-special-candidate} and
we apply Proposition~\ref{prop:candidate-set-matching}.
See Figure~\ref{fig:rank-plus-index-proof}.

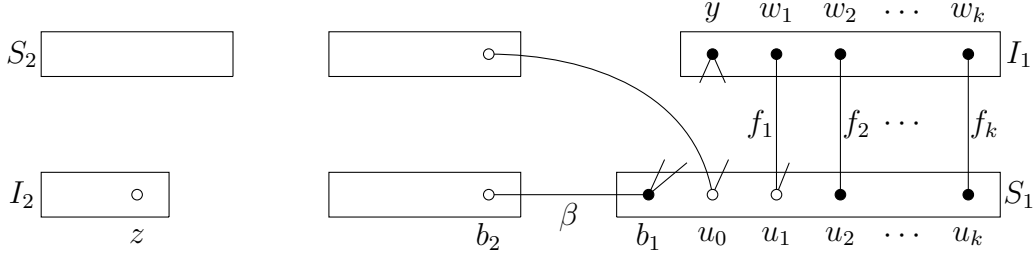
\begin{figure}[!ht]
%\bigskip
\centering
\begin{tikzpicture}[scale=0.85]
\draw (6.3,0.35)node[nodelabel]{$S_1$};
\draw (6.3,2.55)node[nodelabel]{$I_1$};

\draw (-9.3,0.35)node[nodelabel]{$I_2$};
\draw (-9.3,2.55)node[nodelabel]{$S_2$};

%\draw (-3,-1.1)node[nodelabel]{$B-(S_1 \cup I_2)$};
%\draw (-3,4)node[nodelabel]{$A-(S_2 \cup I_1)$};

\draw (1.5,0.35) to [out=100,in=0] (-2,2.55);
\draw (-2,2.55)node{};

\draw (0.5,0.35) -- (0.75,0.9);
\draw (0.5,0.35) -- (1.1,0.85);

%\draw[ultra thick] (1.5,0.35) -- (0.5,2.55);
%\draw (1.5,0.35) to [out=100,in=300] (0.5,2.55);
%\draw (1.5,0.35) to [out=120,in=280] (0.5,2.55);
\draw (1.5,0.35) -- (1.75,0.9);

%\draw (2.5,0.35) -- (0.5,2.55);
\draw (2.5,0.35) -- (2.7,0.85);

\draw (1.5,2.55) -- (1.3,2.1);
\draw (1.5,2.55) -- (1.7,2.1);

\draw (2.25,1.45)node[nodelabel]{$f_1$};
\draw (3.75,1.45)node[nodelabel]{$f_2$};
\draw (5.75,1.45)node[nodelabel]{$f_k$};
\draw (4.5,1.45)node[nodelabel]{$\cdots$};

\draw (0,0) -- (6,0) -- (6,0.7) -- (0,0.7) -- (0,0);

\draw (-1.5,0) -- (-4.5,0) -- (-4.5,0.7) -- (-1.5,0.7) -- (-1.5,0);

\draw (-7,0) -- (-9,0) -- (-9,0.7) -- (-7,0.7) -- (-7,0);

\draw (1,2.2) -- (6,2.2) -- (6,2.9) -- (1,2.9) -- (1,2.2);

\draw (-1.5,2.2) -- (-4.5,2.2) -- (-4.5,2.9) -- (-1.5,2.9) -- (-1.5,2.2);

\draw (-6,2.2) -- (-9,2.2) -- (-9,2.9) -- (-6,2.9) -- (-6,2.2);

\draw (5.5,0.35)node[fill=black]{} -- (5.5,2.55)node[fill=black]{};
\draw (3.5,0.35)node[fill=black]{} -- (3.5,2.55)node[fill=black]{};
\draw (2.5,0.35)node{} -- (2.5,2.55)node[fill=black]{};

%\draw (0.5,2.55)node{};
\draw (0.5,0.35)node[fill=black]{};
\draw (1.5,0.35)node{};
\draw (1.5,2.55)node[fill=black]{};

\draw (0.5,0.35) -- (-2,0.35);
\draw (-0.75,0)node[nodelabel]{$\beta$};
\draw (-2,0.35)node{};
\draw (-2,-0.3)node[nodelabel]{$b_2$};

\draw (0.5,-0.3)node[nodelabel]{$b_1$};
\draw (1.5,-0.3)node[nodelabel]{$u_0$};
\draw (2.5,-0.3)node[nodelabel]{$u_1$};
\draw (3.5,-0.3)node[nodelabel]{$u_2$};
\draw (4.5,-0.3)node[nodelabel]{$\cdots$};
\draw (5.5,-0.3)node[nodelabel]{$u_k$};

%\draw (0.5,3.2)node[nodelabel]{$\overline{x}$};
\draw (1.5,3.2)node[nodelabel]{$y$};
\draw (2.5,3.2)node[nodelabel]{$w_1$};
\draw (3.5,3.2)node[nodelabel]{$w_2$};
\draw (4.5,3.2)node[nodelabel]{$\cdots$};
\draw (5.5,3.2)node[nodelabel]{$w_k$};

\draw (-7.5,0.35)node{};
\draw (-7.5,-0.3)node[nodelabel]{$z$};

\end{tikzpicture}
\vspace*{-0.3in}
\caption{Index of~$e$ is two, and $S_1$ is the largest barrier of~$G-e$}
\label{fig:rank-plus-index-proof}
%\bigskip
\end{figure}

\smallskip
We let $S_2$ denote the unique maximal nontrivial barrier of~$(G-e)/X_1$,
and $I_2$ the set of isolated vertices of~$(G-e)-S_2$.
By the Three Case Lemma (\ref{lem:three-case}),
the end~\vertexb~of~$e$ lies
in~$I_2$, as shown in Figure~\ref{fig:rank-plus-index-proof}.
Note that $|S_2| \leq |S_1|$ by the choice of~$S_1$.

\smallskip
Note that, as per statements {\it (iv)} and {\it (v)} of
Proposition~\ref{prop:candidate-set-matching},
the edge $f_1=u_1w_1$ is the only member of
the candidate set \candidateset{e}{S_1} whose end in the barrier~$S_1$
(that is, vertex~$u_1$) has some neighbour which lies in~$\overline{X_1}$.
Also, if $|S_1|=3$ then $f_1$ is the unique member of \candidateset{e}{S_1}.
For these reasons, it will play a special role.

\begin{Claim}
\label{claim:assume-special-candidate-equal-rank}
We may assume that ${\sf rank}(f_1) = {\sf rank}(e)$.
Consequently, the following hold:
\begin{enumerate}[(i)]
\item the index of~$f_1$ is two,
\item all neighbours of~$u_1$ lie in~$S_2 \cup I_1$, and
\item the vertex~$u_0$ has at least one neighbour in the set~$A-(S_2 \cup I_1)$.
\end{enumerate}
\end{Claim}
\begin{proof}
By Proposition~\ref{prop:edges-of-candidate-set},
$f_1$ is an \Rcomp\ edge which has rank at least that of~$e$,
and it satisfies condition~{\it (i)}, Theorem~\ref{thm:rank-plus-index}.
If ${\sf rank}(f_1) > {\sf rank}(e)$, then by Proposition~\ref{prop:greater-rank-suffices},
$f_1$ satisfies condition~{\it (ii)} as well, and we are done. We may thus assume
that ${\sf rank}(f_1) = {\sf rank}(e)$. It follows from
Lemma~\ref{lem:index-one-candidate-greater-rank} that the index of~$f_1$ is two;
that is, {\it (i)}~holds.
Since $e$ and $f_1=u_1w_1$ are of equal rank and of index two each,
the Equal Rank Lemma (\ref{lem:equal-rank}){\it (v)}
implies that each neighbour of $u_1$ lies in the set~$S_2 \cup I_1$, and this
proves {\it (ii)}. We shall now use this fact to deduce {\it (iii)}.

\smallskip
Since $H$ is bipartite and matching covered,
Proposition~\ref{prop:characterizations-of-bipmcg}{\it (ii)}
implies that the neighbourhood of the set~$A-(S_2 \cup I_1)$, in the graph~$H$,
has cardinality at least $|A-(S_2 \cup I_1)| + 1$, and since
$|A-(S_2 \cup I_1)| = |B - (S_1 \cup I_2)|$, we conclude that
the set~$A-(S_2 \cup I_1)$ has at least one neighbour
which is not in~$B-(S_1 \cup I_2)$;
it follows from Proposition~\ref{prop:candidate-set-matching}
and statement {\it (ii)}
proved above
that the only such neighbour is the vertex~$u_0$ of barrier~$S_1$.
In other words, the vertex $u_0$ has at least one neighbour in the
set~$A-(S_2 \cup I_1)$ as shown in Figure~\ref{fig:rank-plus-index-proof};
this proves {\it (iii)}, and completes the proof of
Claim~\ref{claim:assume-special-candidate-equal-rank}.
\end{proof}

We shall now consider two cases depending on the cardinality of~$S_1$.

\bigskip
\noindent
\underline{Case 1}: $|S_1| \geq 4$.

\medskip
\noindent
We invoke Proposition~\ref{prop:candidate-set-matching-barrier-size-four-or-more},
with $S_1$ playing the role of~$S$, and we adjust notation accordingly.
See Figure~\ref{fig:rank-plus-index-proof-barrier-size-four-or-more}.
Observe that $Q:=u_2w_2b_1$\vertexa$u_2$ is a $4$-cycle of~$G$ which
contains the edge~$f_2=u_2w_2$. Since $f_2$ is a candidate,
it is an \Rcomp\ edge whose rank is at least that of~$e$, and it
satisfies condition~{\it (i)}, Theorem~\ref{thm:rank-plus-index}.
We will use the $4$-cycle~$Q$ and the Equal Rank Lemma to conclude that
$f_2$ has rank strictly greater than that of~$e$, and thus it satisfies
condition~{\it (ii)} as well.

\begin{figure}[!ht]
%\bigskip
\centering
\begin{tikzpicture}[scale=0.85]
\draw (6.3,0.35)node[nodelabel]{$S_1$};
\draw (6.3,2.55)node[nodelabel]{$I_1$};

\draw (-9.3,0.35)node[nodelabel]{$I_2$};
\draw (-9.3,2.55)node[nodelabel]{$S_2$};

%\draw (-3,-1.1)node[nodelabel]{$B-(S_1 \cup I_2)$};
%\draw (-3,4)node[nodelabel]{$A-(S_2 \cup I_1)$};

\draw (1.5,0.35) to [out=100,in=0] (-2,2.55);
\draw (-2,2.55)node{};

\draw (0.5,0.35) -- (1.5,2.55);
\draw (0.5,0.35) -- (3.5,2.55);
\draw (1.5,2.55) -- (3.5,0.35);

%\draw[ultra thick] (1.5,0.35) -- (0.5,2.55);
%\draw (1.5,0.35) to [out=100,in=300] (0.5,2.55);
%\draw (1.5,0.35) to [out=120,in=280] (0.5,2.55);
\draw (1.5,0.35) -- (1.75,0.9);

%\draw (2.5,0.35) -- (0.5,2.55);
\draw (2.5,0.35) -- (2.7,0.85);

%\draw (1.5,2.55) -- (1.3,2.1);
%\draw (1.5,2.55) -- (1.7,2.1);

%\draw (2.25,1.45)node[nodelabel]{$f_1$};
\draw (3.75,1.45)node[nodelabel]{$f_2$};
\draw (5.75,1.45)node[nodelabel]{$f_k$};
\draw (4.5,1.45)node[nodelabel]{$\cdots$};

\draw (0,0) -- (6,0) -- (6,0.7) -- (0,0.7) -- (0,0);

\draw (-1.5,0) -- (-4.5,0) -- (-4.5,0.7) -- (-1.5,0.7) -- (-1.5,0);

\draw (-7,0) -- (-9,0) -- (-9,0.7) -- (-7,0.7) -- (-7,0);

\draw (1,2.2) -- (6,2.2) -- (6,2.9) -- (1,2.9) -- (1,2.2);

\draw (-1.5,2.2) -- (-4.5,2.2) -- (-4.5,2.9) -- (-1.5,2.9) -- (-1.5,2.2);

\draw (-6,2.2) -- (-9,2.2) -- (-9,2.9) -- (-6,2.9) -- (-6,2.2);

\draw (5.5,0.35)node[fill=black]{} -- (5.5,2.55)node[fill=black]{};
\draw (3.5,0.35)node[fill=black]{} -- (3.5,2.55)node[fill=black]{};
\draw (2.5,0.35)node{} -- (2.5,2.55)node[fill=black]{};

%\draw (0.5,2.55)node{};
\draw (0.5,0.35)node[fill=black]{};
\draw (1.5,0.35)node{};
\draw (1.5,2.55)node[fill=black]{};

\draw (0.5,0.35) -- (-2,0.35);
\draw (-0.75,0)node[nodelabel]{$\beta$};
\draw (-2,0.35)node{};
\draw (-2,-0.3)node[nodelabel]{$b_2$};

\draw (0.5,-0.3)node[nodelabel]{$b_1$};
\draw (1.5,-0.3)node[nodelabel]{$u_0$};
\draw (2.5,-0.3)node[nodelabel]{$u_1$};
\draw (3.5,-0.3)node[nodelabel]{$u_2$};
\draw (4.5,-0.3)node[nodelabel]{$\cdots$};
\draw (5.5,-0.3)node[nodelabel]{$u_k$};

%\draw (0.5,3.2)node[nodelabel]{$\overline{x}$};
\draw (1.5,3.2)node[nodelabel]{$y$};
\draw (2.5,3.2)node[nodelabel]{$w_1$};
\draw (3.5,3.2)node[nodelabel]{$w_2$};
\draw (4.5,3.2)node[nodelabel]{$\cdots$};
\draw (5.5,3.2)node[nodelabel]{$w_k$};

\draw (-7.5,0.35)node{};
\draw (-7.5,-0.3)node[nodelabel]{$z$};

\end{tikzpicture}
\vspace*{-0.3in}
\caption{When $|S_1| \geq 4$}
\label{fig:rank-plus-index-proof-barrier-size-four-or-more}
%\bigskip
\end{figure}
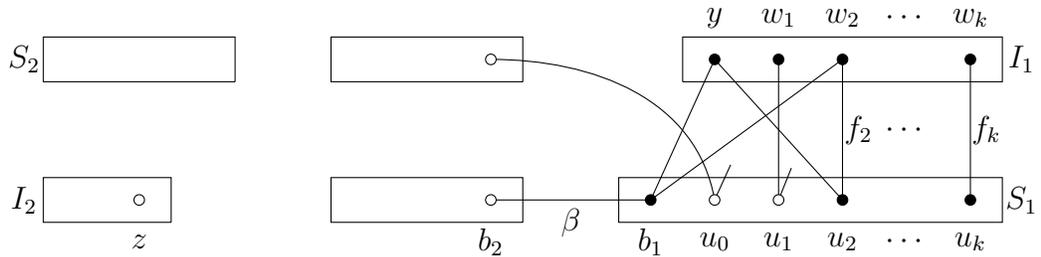

Now, let $v$ denote the neighbour
of $w_2$ which is distinct from~$u_2$ and $b_1$.
Clearly, $v \in S_1$;
by Proposition~\ref{prop:candidate-set-matching-barrier-size-four-or-more}{\it (iii)},
$v$ is distinct from~$u_0$.

\smallskip
Since each end of~$f_2$ is cubic, it is an \Rcomp\ edge of index two.
We first set up some notation concerning the barrier structure of~$G-f_2$.
We denote by $S_3$ the maximal nontrivial barrier of~$G-f_2$ which is a subset
of~$B$, and by $I_3$ the set of isolated vertices of~$(G-f_2)-S_3$.
We let $S_4$ denote the unique maximal nontrivial barrier of~$(G-f_2)/(S_3 \cup I_3)$,
and $I_4$ the set of isolated vertices of~$(G-f_2)-S_4$.
By the Three Case Lemma (\ref{lem:three-case}), the end~$u_2$ of~$f_2$ lies
in~$I_4$, and its end~$w_2$ lies in~$I_3$.
Also, since $w_2 \in I_3$, $v \in S_3$.

\smallskip
Now, suppose for the sake of contradiction that ${\sf rank}(f_2) = {\sf rank}(e)$.
Then we may apply the Equal Rank Lemma (\ref{lem:equal-rank}) to
conclude that $S_1 \cup I_2 = S_3 \cup I_4$ and that $S_2 \cup I_1 = S_4 \cup I_3$. Furthermore, by Claim~\ref{claim:assume-special-candidate-equal-rank}{\it (iii)},
the vertex~$u_0$ has a neighbour in $A-(S_4 \cup I_3)$,
and thus $u_0 \notin I_4$.
We infer that $u_0 \in S_3$.
We have thus shown that $v$ and $u_0$ are distinct vertices of the
barrier~$S_3$ of~$G-f_2$.
Consequently, $(G-f_2)-\{v,u_0\}$ has no perfect matching;
we will now use the $4$-cycle~$Q=u_2w_2b_1$\vertexa$u_2$
to contradict this assertion.

\smallskip
Since $G$ is a brick, $G-\{v,u_0\}$ has a perfect matching, say~$M$.
If~$f_2$ is not in~$M$ then we have the desired contradiction.
Now suppose that $f_2 \in M$. Since $v$ and $u_0$ both lie in the color
class~$B$ of~$H$, we conclude that $\alpha \in M$ and that $\beta \notin M$.
See Figure~\ref{fig:rank-plus-index-proof-barrier-size-four-or-more}.
Note that each of $v$ and $u_0$ is distinct from~$b_1$,
and that the neighbours of~$b_1$ are precisely $b_2, w_2$ and~\vertexa.
Since $\beta=b_1b_2$ is not in $M$, and since $f_2=u_2w_2$ lies in~$M$,
it must be the case that \vertexa$b_1$ lies in~$M$. Now observe that
the symmetric difference of $M$ and $Q$ is a perfect matching
of~$(G-f_2)-\{v,u_0\}$, and thus we have the desired contradiction.

\smallskip
We conclude that ${\sf rank}(f_2) > {\sf rank}(e)$, and thus
$f_2$ is the desired \Rcomp\ edge which satisfies both
conditions~{\it (i)} and {\it (ii)}, Theorem~\ref{thm:rank-plus-index}.

\bigskip
\noindent
\underline{Case 2}: $|S_1| = 3$.

\medskip
\noindent
We note that since $S_1$ has precisely three vertices,
by Remark~\ref{rem:candidate-set-matching-barrier-size-three},
all of the edges of~$G[X_1]$ are determined (where $X_1=S_1 \cup I_1$).
See Figure~\ref{fig:rank-plus-index-proof-barrier-size-three}.
Furthermore, $f_1$ is the only member
of the candidate set \candidateset{e}{S_1},
and by Claim~\ref{claim:assume-special-candidate-equal-rank},
its index is two and its rank is equal to ${\sf rank}(e)$.
We will examine the barrier structure of~$G-f_1$
using the Equal Rank Lemma~(\ref{lem:equal-rank}),
and argue that some edge adjacent with the given
edge~$e=$~\vertexa\vertexb~(that is,
either incident at~\vertexa, or incident at~\vertexb)
is \Rcomp\ and that its rank is strictly greater than ${\sf rank}(e)$.
Observe that, since ${\sf index}(e)=2$,
each edge adjacent with~$e$ satisfies
condition {\it (i)}, Theorem~\ref{thm:rank-plus-index}.

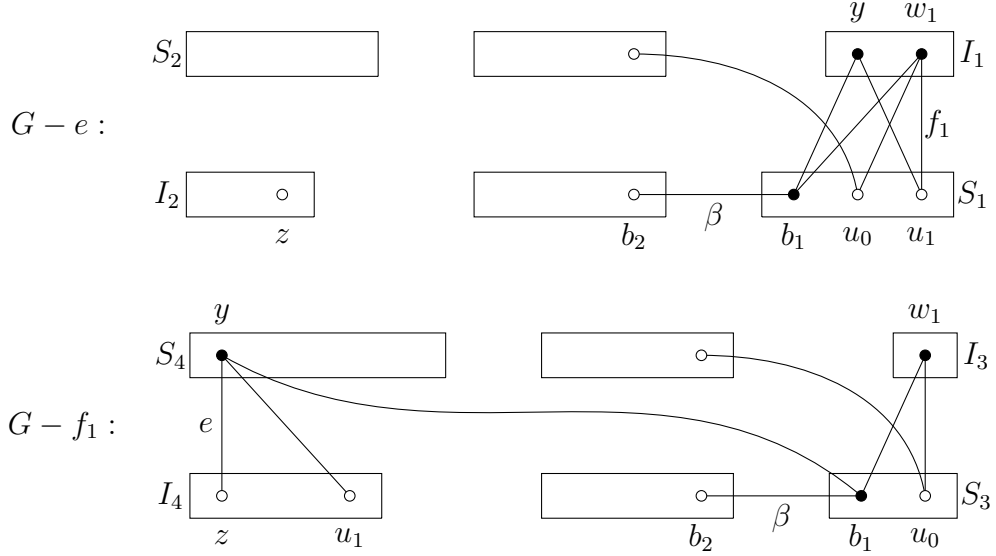
\begin{figure}[!ht]
%\bigskip
\centering
\begin{tikzpicture}[scale=0.85]
\draw (-11,1.45)node[nodelabel]{$G-e:$};

\draw (3.3,0.35)node[nodelabel]{$S_1$};
\draw (3.3,2.55)node[nodelabel]{$I_1$};

\draw (-9.3,0.35)node[nodelabel]{$I_2$};
\draw (-9.3,2.55)node[nodelabel]{$S_2$};

\draw (1.5,0.35) to [out=100,in=0] (-2,2.55);
\draw (-2,2.55)node{};

\draw (0.5,0.35) -- (-2,0.35);
\draw (-0.75,0)node[nodelabel]{$\beta$};
\draw (-2,0.35)node{};
\draw (-2,-0.3)node[nodelabel]{$b_2$};

%\draw (2.5,0.35) -- (0.5,2.55);
\draw (2.5,0.35) -- (1.5,2.55);

\draw (0.5,0.35) -- (1.5,2.55);
\draw (0.5,0.35) -- (2.5,2.55);

\draw (1.5,0.35) -- (2.5,2.55);

\draw (2.75,1.45)node[nodelabel]{$f_1$};

\draw (0,0) -- (3,0) -- (3,0.7) -- (0,0.7) -- (0,0);

\draw (1,2.2) -- (3,2.2) -- (3,2.9) -- (1,2.9) -- (1,2.2);

\draw (-1.5,0) -- (-4.5,0) -- (-4.5,0.7) -- (-1.5,0.7) -- (-1.5,0);

\draw (-7,0) -- (-9,0) -- (-9,0.7) -- (-7,0.7) -- (-7,0);

\draw (-1.5,2.2) -- (-4.5,2.2) -- (-4.5,2.9) -- (-1.5,2.9) -- (-1.5,2.2);

\draw (-6,2.2) -- (-9,2.2) -- (-9,2.9) -- (-6,2.9) -- (-6,2.2);

\draw (2.5,0.35)node{} -- (2.5,2.55)node[fill=black]{};

%\draw (0.5,2.55)node{};
\draw (0.5,0.35)node[fill=black]{};
\draw (1.5,0.35)node{};
\draw (1.5,2.55)node[fill=black]{};

\draw (0.5,-0.3)node[nodelabel]{$b_1$};
\draw (1.5,-0.3)node[nodelabel]{$u_0$};
\draw (2.5,-0.3)node[nodelabel]{$u_1$};

\draw (1.5,3.2)node[nodelabel]{$y$};
\draw (2.5,3.2)node[nodelabel]{$w_1$};

\draw (-7.5,0.35)node{};
\draw (-7.5,-0.3)node[nodelabel]{$z$};

\end{tikzpicture}

%\vspace*{0.2in}
\begin{tikzpicture}[scale=0.85]
\draw (-12,1.45)node[nodelabel]{$G-f_1:$};

\draw (2.3,0.35)node[nodelabel]{$S_3$};
\draw (2.3,2.55)node[nodelabel]{$I_3$};

\draw (-10.3,0.35)node[nodelabel]{$I_4$};
\draw (-10.3,2.55)node[nodelabel]{$S_4$};

\draw (1.5,0.35) to [out=100,in=0] (-2,2.55);
\draw (-2,2.55)node{};

\draw (0.5,0.35) -- (-2,0.35);
\draw (-0.75,0)node[nodelabel]{$\beta$};
\draw (-2,0.35)node{};
\draw (-2,-0.3)node[nodelabel]{$b_2$};

%\draw (2.5,0.35) -- (0.5,2.55);
%\draw (2.5,0.35) -- (1.5,2.55);

%\draw (0.5,0.35) -- (1.5,2.55);
%\draw (0.5,0.35) -- (2.5,2.55);

%\draw (1.5,0.35) -- (2.5,2.55);

%\draw (2.75,1.45)node[nodelabel]{$f_1$};

\draw (1.5,2.55) -- (0.5,0.35);
\draw (1.5,2.55) -- (1.5,0.35);

\draw (0,0) -- (2,0) -- (2,0.7) -- (0,0.7) -- (0,0);

\draw (1,2.2) -- (2,2.2) -- (2,2.9) -- (1,2.9) -- (1,2.2);

\draw (-1.5,0) -- (-4.5,0) -- (-4.5,0.7) -- (-1.5,0.7) -- (-1.5,0);

\draw (-7,0) -- (-10,0) -- (-10,0.7) -- (-7,0.7) -- (-7,0);

\draw (-1.5,2.2) -- (-4.5,2.2) -- (-4.5,2.9) -- (-1.5,2.9) -- (-1.5,2.2);

\draw (-6,2.2) -- (-10,2.2) -- (-10,2.9) -- (-6,2.9) -- (-6,2.2);

%\draw (2.5,0.35)node{} -- (2.5,2.55)node[fill=black]{};

%\draw (0.5,2.55)node{};
\draw (0.5,0.35)node[fill=black]{};
\draw (1.5,0.35)node{};
\draw (1.5,2.55)node[fill=black]{};

\draw (-7.5,0.35) -- (-9.5,2.55);
\draw (0.5,0.35) to [out=140,in=330] (-9.5,2.55);
\draw (0.5,-0.3)node[nodelabel]{$b_1$};
\draw (1.5,-0.3)node[nodelabel]{$u_0$};
%\draw (2.5,-0.3)node[nodelabel]{$u_1$};

\draw (1.5,3.2)node[nodelabel]{$w_1$};
%\draw (2.5,3.2)node[nodelabel]{$w_1$};

\draw (-7.5,0.35)node{};
\draw (-7.5,-0.3)node[nodelabel]{$u_1$};

\draw (-9.5,0.35)node{} -- (-9.5,2.55)node[fill=black]{};
\draw (-9.5,3.2)node[nodelabel]{$y$};
\draw (-9.5,-0.3)node[nodelabel]{$z$};
\draw (-9.75,1.45)node[nodelabel]{$e$};

\end{tikzpicture}
\vspace*{-0.3in}
\caption{When $|S_1|=3$}
\label{fig:rank-plus-index-proof-barrier-size-three}
%\bigskip
\end{figure}

We let $S_3$ denote the unique maximal nontrivial barrier of~$G-f_1$ which is a subset of~$B$, and $I_3$
the set of isolated vertices of~$(G-f_1)-S_3$. We denote by $S_4$ the unique maximal nontrivial barrier of~$(G-f_1)/(S_3 \cup I_3)$,
and by $I_4$ the set of isolated vertices of~$(G-f_1)-S_4$.
See Figure~\ref{fig:rank-plus-index-proof-barrier-size-three}.
By the Three Case Lemma (\ref{lem:three-case}),
the end~$u_1$ of~$f_1$ lies in~$I_4$,
and its end~$w_1$ lies in~$I_3$. Since each of $b_1$~and~$u_0$ is a neighbour
of~$w_1$ in $G-f_1$, they both lie in the barrier~$S_3$.
By Lemma~\ref{lem:equal-rank}{\it (ii)},
with $f_1$ playing the role of~$f$,
we conclude that $S_3 = \{b_1,u_0\}$ and that $I_3 = \{w_1\}$,
as shown in the figure.

\smallskip
Observe that by the choice of~$S_1$, the barrier~$S_2$ of~$G-e$
contains either two or three vertices.
However, irrespective of the cardinality of~$S_2$,
it follows from the above and from Lemma~\ref{lem:equal-rank}{\it (iv)}
that $S_4 = S_2 \cup \{$\vertexa$\}$ and that $I_4 = I_2 \cup \{u_1\}$.
In particular, the barrier $S_4$ of~$G-f_1$ contains either three or four vertices.
Note that the end~\vertexb~of~$e$ lies in~$I_2$ which is a subset of~$I_4$,
and its end~\vertexa~lies in~$S_4$.
Furthermore, Lemma~\ref{lem:equal-rank}{\it (vi)} implies
that $e$ is a member of the candidate set \candidateset{f_1}{S_4}.

\begin{Claim}
\label{claim:assume-other-barrier-cardinality-two}
We may assume that $e$ is the only member of \candidateset{f_1}{S_4}
which is incident with~\vertexb. Furthermore, we may assume that $|S_2|=2$.
\end{Claim}
\begin{proof}
Suppose there exists an edge~$g$ incident with~\vertexb~such
that $g$ is distinct from~$e$ and that $g \in$ \candidateset{f_1}{S_4}.
By Proposition~\ref{prop:edges-of-candidate-set}, $g$ is an \Rcomp\ edge
of the brick~$G$.
We now apply Corollary~\ref{cor:adjacent-candidates-greater-rank} (with
$f_1$ playing the role of~$e$, and with edges $e$ and $g$ playing the
roles of $f$ and $g$); at least one of $e$~and~$g$ has rank strictly
greater than ${\sf rank}(f_1)$. However,
by Claim~\ref{claim:assume-special-candidate-equal-rank}, the ranks
of $e$~and~$f_1$ are equal;
consequently, ${\sf rank}(g) > {\sf rank}(f_1) = {\sf rank}(e)$.
By Propostion~\ref{prop:greater-rank-suffices}, the edge~$g$ satisifes
condition~{\it (ii)}, Theorem~\ref{thm:rank-plus-index},
and it satisfies condition~{\it (i)} because it is adjacent with the edge~$e$,
and thus we are done. So we may assume that $e$ is the only
member of \candidateset{f_1}{S_4} which is incident with~\vertexb.
Using this, we shall deduce that the barrier~$S_2$ of~$G-e$
has only two vertices.

\smallskip
Suppose to the contrary that $|S_2|=3$.
By Claim~\ref{claim:assume-candidate-set-matching},
the candidate set \candidateset{e}{S_2} is a matching.
Consequently, as we did in the case of~$S_1$, we
may now invoke the observations made in Section~\ref{sec:candidate-set-matching},
with $S_2$ playing the role of~$S$, and $I_2$ playing the role of~$I$,
and likewise, $X_2:=S_2 \cup I_2$ playing the role of~$X$.
In particular, by Remark~\ref{rem:candidate-set-matching-barrier-size-three},
all of the edges of~$G[X_2]$ are determined.
It is worth noting that $S_2$ is also a maximal barrier of~$G-e$
(by the choice of~$S_1$). That is, each of $S_1$ and $S_2$ is a maximal
barrier of~$G-e$ with exactly three vertices.
Keeping this symmetry in view, we now choose
appropriate notation for those vertices of $X_2$ which are relevant to our argument.
See Figure~\ref{fig:rank-plus-index-proof-both-barriers-size-three}.

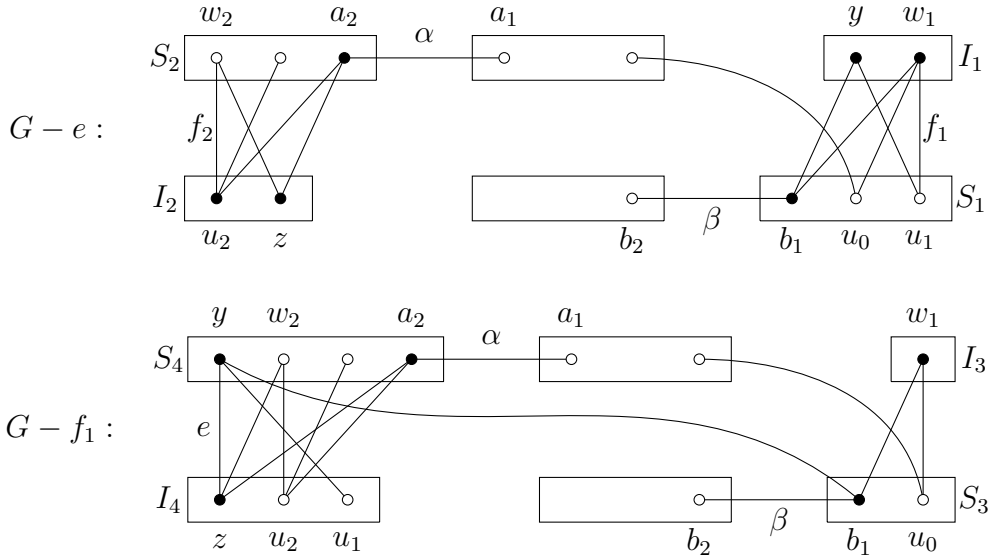
\begin{figure}[!ht]
%\bigskip
\centering
\begin{tikzpicture}[scale=0.85]
\draw (-11,1.45)node[nodelabel]{$G-e:$};

\draw (3.3,0.35)node[nodelabel]{$S_1$};
\draw (3.3,2.55)node[nodelabel]{$I_1$};

\draw (-9.3,0.35)node[nodelabel]{$I_2$};
\draw (-9.3,2.55)node[nodelabel]{$S_2$};

\draw (1.5,0.35) to [out=100,in=0] (-2,2.55);
\draw (-2,2.55)node{};

\draw (0.5,0.35) -- (-2,0.35);
\draw (-0.75,0)node[nodelabel]{$\beta$};
\draw (-2,0.35)node{};
\draw (-2,-0.3)node[nodelabel]{$b_2$};

%\draw (2.5,0.35) -- (0.5,2.55);
\draw (2.5,0.35) -- (1.5,2.55);

\draw (0.5,0.35) -- (1.5,2.55);
\draw (0.5,0.35) -- (2.5,2.55);

\draw (1.5,0.35) -- (2.5,2.55);

\draw (2.75,1.45)node[nodelabel]{$f_1$};

\draw (0,0) -- (3,0) -- (3,0.7) -- (0,0.7) -- (0,0);

\draw (1,2.2) -- (3,2.2) -- (3,2.9) -- (1,2.9) -- (1,2.2);

\draw (-1.5,0) -- (-4.5,0) -- (-4.5,0.7) -- (-1.5,0.7) -- (-1.5,0);

\draw (-7,0) -- (-9,0) -- (-9,0.7) -- (-7,0.7) -- (-7,0);

\draw (-1.5,2.2) -- (-4.5,2.2) -- (-4.5,2.9) -- (-1.5,2.9) -- (-1.5,2.2);

\draw (-6,2.2) -- (-9,2.2) -- (-9,2.9) -- (-6,2.9) -- (-6,2.2);

\draw (2.5,0.35)node{} -- (2.5,2.55)node[fill=black]{};

%\draw (0.5,2.55)node{};
\draw (0.5,0.35)node[fill=black]{};
\draw (1.5,0.35)node{};
\draw (1.5,2.55)node[fill=black]{};

\draw (0.5,-0.3)node[nodelabel]{$b_1$};
\draw (1.5,-0.3)node[nodelabel]{$u_0$};
\draw (2.5,-0.3)node[nodelabel]{$u_1$};

\draw (1.5,3.2)node[nodelabel]{$y$};
\draw (2.5,3.2)node[nodelabel]{$w_1$};

\draw (-7.5,0.35) -- (-8.5,2.55);
\draw (-6.5,2.55) -- (-7.5,0.35);
\draw (-6.5,2.55) -- (-8.5,0.35);
\draw (-8.5,0.35) -- (-7.5,2.55);
\draw (-8.5,0.35)node[fill=black]{} -- (-8.5,2.55)node{};
\draw (-8.5,-0.3)node[nodelabel]{$u_2$};
\draw (-8.5,3.2)node[nodelabel]{$w_2$};
\draw (-8.75,1.45)node[nodelabel]{$f_2$};

\draw (-6.5,2.55) -- (-4,2.55);
\draw (-5.25,2.9)node[nodelabel]{$\alpha$};
\draw (-4,2.55)node{};
\draw (-4,3.2)node[nodelabel]{$a_1$};
\draw (-6.5,2.55)node[fill=black]{};
\draw (-6.5,3.2)node[nodelabel]{$a_2$};
\draw (-7.5,-0.3)node[nodelabel]{$z$};
\draw (-7.5,0.35)node[fill=black]{};
\draw (-7.5,2.55)node{};
\end{tikzpicture}

%\vspace*{0.2in}
\begin{tikzpicture}[scale=0.85]
\draw (-12,1.45)node[nodelabel]{$G-f_1:$};

\draw (2.3,0.35)node[nodelabel]{$S_3$};
\draw (2.3,2.55)node[nodelabel]{$I_3$};

\draw (-10.3,0.35)node[nodelabel]{$I_4$};
\draw (-10.3,2.55)node[nodelabel]{$S_4$};

\draw (1.5,0.35) to [out=100,in=0] (-2,2.55);
\draw (-2,2.55)node{};

\draw (0.5,0.35) -- (-2,0.35);
\draw (-0.75,0)node[nodelabel]{$\beta$};
\draw (-2,0.35)node{};
\draw (-2,-0.3)node[nodelabel]{$b_2$};

%\draw (2.5,0.35) -- (0.5,2.55);
%\draw (2.5,0.35) -- (1.5,2.55);

%\draw (0.5,0.35) -- (1.5,2.55);
%\draw (0.5,0.35) -- (2.5,2.55);

%\draw (1.5,0.35) -- (2.5,2.55);

%\draw (2.75,1.45)node[nodelabel]{$f_1$};

\draw (1.5,2.55) -- (0.5,0.35);
\draw (1.5,2.55) -- (1.5,0.35);

\draw (0,0) -- (2,0) -- (2,0.7) -- (0,0.7) -- (0,0);

\draw (1,2.2) -- (2,2.2) -- (2,2.9) -- (1,2.9) -- (1,2.2);

\draw (-1.5,0) -- (-4.5,0) -- (-4.5,0.7) -- (-1.5,0.7) -- (-1.5,0);

\draw (-7,0) -- (-10,0) -- (-10,0.7) -- (-7,0.7) -- (-7,0);

\draw (-1.5,2.2) -- (-4.5,2.2) -- (-4.5,2.9) -- (-1.5,2.9) -- (-1.5,2.2);

\draw (-6,2.2) -- (-10,2.2) -- (-10,2.9) -- (-6,2.9) -- (-6,2.2);

%\draw (2.5,0.35)node{} -- (2.5,2.55)node[fill=black]{};

%\draw (0.5,2.55)node{};
\draw (0.5,0.35)node[fill=black]{};
\draw (1.5,0.35)node{};
\draw (1.5,2.55)node[fill=black]{};

\draw (-7.5,0.35) -- (-9.5,2.55);
\draw (0.5,0.35) to [out=140,in=330] (-9.5,2.55);
\draw (0.5,-0.3)node[nodelabel]{$b_1$};
\draw (1.5,-0.3)node[nodelabel]{$u_0$};
%\draw (2.5,-0.3)node[nodelabel]{$u_1$};

\draw (1.5,3.2)node[nodelabel]{$w_1$};
%\draw (2.5,3.2)node[nodelabel]{$w_1$};

\draw (-7.5,0.35)node{};
\draw (-7.5,-0.3)node[nodelabel]{$u_1$};

\draw (-9.5,0.35) -- (-9.5,2.55)node[fill=black]{};
\draw (-9.5,3.2)node[nodelabel]{$y$};
\draw (-9.5,-0.3)node[nodelabel]{$z$};
\draw (-9.75,1.45)node[nodelabel]{$e$};

\draw (-8.5,3.2)node[nodelabel]{$w_2$};
\draw (-8.5,-0.3)node[nodelabel]{$u_2$};
\draw (-8.5,2.55) -- (-8.5,0.35);
\draw (-8.5,2.55) -- (-9.5,0.35);

\draw (-6.5,2.55) -- (-9.5,0.35);
\draw (-6.5,2.55) -- (-8.5,0.35);

\draw (-8.5,0.35) -- (-7.5,2.55);
\draw (-6.5,2.55) -- (-4,2.55);
\draw (-5.25,2.9)node[nodelabel]{$\alpha$};
\draw (-4,2.55)node{};
\draw (-4,3.2)node[nodelabel]{$a_1$};
\draw (-6.5,2.55)node[fill=black]{};
\draw (-6.5,3.2)node[nodelabel]{$a_2$};
\draw (-8.5,0.35)node{};
\draw (-8.5,2.55)node{};
\draw (-9.5,0.35)node[fill=black]{};
\draw (-7.5,2.55)node{};

\end{tikzpicture}
\vspace*{-0.3in}
\caption{When $|S_1|=|S_2|=3$}
\label{fig:rank-plus-index-proof-both-barriers-size-three}
%\bigskip
\end{figure}

We shall let~$f_2:=u_2w_2$ denote the unique member of the
candidate set \candidateset{e}{S_2}, where $u_2 \in I_2$ and $w_2 \in S_2$.
In particular, $I_2=\{u_2,$\vertexb$\}$.
One of the ends of $\alpha=a_1a_2$ lies in the barrier~$S_2$; we adjust
notation so that $a_2 \in S_2$. Consequently, $w_2$~and~$a_2$ are distinct
vertices of~$S_2$.
%We shall denote by~$w_0$ the third vertex of~$S_2$.
%So, $S_2 = \{w_2,w_0,a_2\}$.
The vertex~$a_2$ is cubic, and its neighbours are~\vertexb$, u_2$ and~$a_1$.
%The vertex~$w_0$ is adjacent with~$u_2$, and all of its remaining neighbours
%lie in~$\overline{X_2}$.
The vertex~$w_2$ is adjacent with~\vertexb~and~$u_2$, and all of its remaining
neighbours lie in~$\overline{X_2}$.

\smallskip
Observe that $Q:=$~\vertexb$w_2u_2a_2$\vertexb~is a $4$-cycle of the
bipartite graph \bipartitebarrier{f_1}{S_4} which contains
the vertex~\vertexb~whose degree is three. Consequently,
by Corollary~\ref{cor:quadrilateral-LV},
at least one of \vertexb$w_2$~and~\vertexb$a_2$ is removable in
\bipartitebarrier{f_1}{S_4}. However, since $a_2$ has degree two
in \bipartitebarrier{f_1}{S_4}, \vertexb$a_2$ is non-removable;
whence~\vertexb$w_2$ is removable.
It follows that~\vertexb$w_2$ is a member of the
candidate set \candidateset{f_1}{S_4}; this contradicts our first assumption.
We conclude that the barrier~$S_2$ has only two vertices, and this
completes the proof of Claim~\ref{claim:assume-other-barrier-cardinality-two}.
\end{proof}

By Proposition~\ref{prop:equivalent-definition-Rthin}, an \Rcomp\ edge of
index two is thin if and only if its rank is $n-4$; where $n:=|V(G)|$.
Observe that, since $|S_1|=3$ and $|S_2|=2$, the rank of~$e$ is $n-6$,
and in this sense, it is very close to being thin;
the same holds for the edge~$f_1$.
We will establish a symmetry between the barrier structure
of~$G-e$ and that of~$G-f_1$;
see Figure~\ref{fig:rank-plus-index-proof-only-one-barrier-size-three}.
Thereafter, we will
argue that the edge~$g:=$~\vertexa$u_1$ is an \Rthin\ edge of index two;
in particular, it is \Rcomp\ and its rank is~$n-4$,
and thus it satisfies condition {\it (ii)}, Theorem~\ref{thm:rank-plus-index}.
Since $g$ is adjacent with~$e$,
it satisfies condition {\it (i)} as well.

\smallskip
Since~$|S_2|=2$, the set~$I_2$ contains only the end~\vertexb~of~$e$,
and the neighbourhood of~\vertexb~is precisely the
set~$S_2 \cup \{$\vertexa$\}=S_4$.
Also, \mbox{$I_4 = I_2 \cup \{u_1\} = \{$\vertexb$,u_1\}$},
and by Claim~\ref{claim:assume-other-barrier-cardinality-two},
$e=$~\vertexa\vertexb~is the only member of the candidate set \candidateset{f_1}{S_4} which
is incident with~\vertexb.
In other words,~\vertexb~is incident with only one removable edge of the bipartite graph
\bipartitebarrier{f_1}{S_4}, namely, the edge~$e$.
We now deduce some consequences
of this fact using standard arguments.

\begin{figure}[!ht]
\centering
%\bigskip
\begin{tikzpicture}[scale=0.85]
\draw (-10,1.45)node[nodelabel]{$G-e:$};

\draw (3.3,0.35)node[nodelabel]{$S_1$};
\draw (3.3,2.55)node[nodelabel]{$I_1$};

\draw (-8.3,0.35)node[nodelabel]{$I_2$};
\draw (-8.3,2.55)node[nodelabel]{$S_2$};

\draw (-7.5,2.55) to [out=280,in=180] (-4,0.35);
\draw (-4,0.35)node{};
\draw (-7.5,3.2)node[nodelabel]{$w_0$};
\draw (1.5,0.35) to [out=100,in=0] (-2,2.55);
\draw (-2,2.55)node{};

\draw (2.5,0.35) to [out=140,in=320] (-6.5,2.55);

\draw (0.5,0.35) -- (-2,0.35);
\draw (-0.75,0)node[nodelabel]{$\beta$};
\draw (-2,0.35)node{};
\draw (-2,-0.3)node[nodelabel]{$b_2$};

%\draw (2.5,0.35) -- (0.5,2.55);
\draw (2.5,0.35) -- (1.5,2.55);

\draw (0.5,0.35) -- (1.5,2.55);
\draw (0.5,0.35) -- (2.5,2.55);

\draw (1.5,0.35) -- (2.5,2.55);

\draw (2.75,1.45)node[nodelabel]{$f_1$};

\draw (0,0) -- (3,0) -- (3,0.7) -- (0,0.7) -- (0,0);

\draw (1,2.2) -- (3,2.2) -- (3,2.9) -- (1,2.9) -- (1,2.2);

\draw (-1.5,0) -- (-4.5,0) -- (-4.5,0.7) -- (-1.5,0.7) -- (-1.5,0);

\draw (-7,0) -- (-8,0) -- (-8,0.7) -- (-7,0.7) -- (-7,0);

\draw (-1.5,2.2) -- (-4.5,2.2) -- (-4.5,2.9) -- (-1.5,2.9) -- (-1.5,2.2);

\draw (-6,2.2) -- (-8,2.2) -- (-8,2.9) -- (-6,2.9) -- (-6,2.2);

\draw (2.5,0.35)node[fill=black]{} -- (2.5,2.55)node[fill=black]{};

%\draw (0.5,2.55)node{};
\draw (0.5,0.35)node[fill=black]{};
\draw (1.5,0.35)node{};
\draw (1.5,2.55)node[fill=black]{};

\draw (0.5,-0.3)node[nodelabel]{$b_1$};
\draw (1.5,-0.3)node[nodelabel]{$u_0$};
\draw (2.5,-0.3)node[nodelabel]{$u_1$};

\draw (1.5,3.2)node[nodelabel]{$y$};
\draw (2.5,3.2)node[nodelabel]{$w_1$};

\draw (-7.5,0.35) -- (-6.5,2.55);
\draw (-7.5,0.35) -- (-7.5,2.55)node{};
\draw (-7.5,0.35)node[fill=black]{};
\draw (-7.5,-0.3)node[nodelabel]{$z$};

\draw (-6.5,2.55) -- (-4,2.55);
\draw (-5.25,2.9)node[nodelabel]{$\alpha$};
\draw (-4,2.55)node{};
\draw (-4,3.2)node[nodelabel]{$a_1$};
\draw (-6.5,2.55)node[fill=black]{};
\draw (-6.5,3.2)node[nodelabel]{$a_2$};
\end{tikzpicture}

%\vspace*{0.2in}
\begin{tikzpicture}[scale=0.85]
\draw (-11,1.45)node[nodelabel]{$G-f_1:$};

\draw (2.3,0.35)node[nodelabel]{$S_3$};
\draw (2.3,2.55)node[nodelabel]{$I_3$};

\draw (-9.3,0.35)node[nodelabel]{$I_4$};
\draw (-9.3,2.55)node[nodelabel]{$S_4$};

\draw (-7.5,2.55) to [out=280,in=180] (-4,0.35);
\draw (-4,0.35)node{};
\draw (1.5,0.35) to [out=100,in=0] (-2,2.55);
\draw (-2,2.55)node{};

\draw (0.5,0.35) -- (-2,0.35);
\draw (-0.75,0)node[nodelabel]{$\beta$};
\draw (-2,0.35)node{};
\draw (-2,-0.3)node[nodelabel]{$b_2$};

%\draw (2.5,0.35) -- (0.5,2.55);
%\draw (2.5,0.35) -- (1.5,2.55);

%\draw (0.5,0.35) -- (1.5,2.55);
%\draw (0.5,0.35) -- (2.5,2.55);

%\draw (1.5,0.35) -- (2.5,2.55);

%\draw (2.75,1.45)node[nodelabel]{$f_1$};

\draw (1.5,2.55) -- (0.5,0.35);
\draw (1.5,2.55) -- (1.5,0.35);

\draw (0,0) -- (2,0) -- (2,0.7) -- (0,0.7) -- (0,0);

\draw (1,2.2) -- (2,2.2) -- (2,2.9) -- (1,2.9) -- (1,2.2);

\draw (-1.5,0) -- (-4.5,0) -- (-4.5,0.7) -- (-1.5,0.7) -- (-1.5,0);

\draw (-7,0) -- (-9,0) -- (-9,0.7) -- (-7,0.7) -- (-7,0);

\draw (-1.5,2.2) -- (-4.5,2.2) -- (-4.5,2.9) -- (-1.5,2.9) -- (-1.5,2.2);

\draw (-6,2.2) -- (-9,2.2) -- (-9,2.9) -- (-6,2.9) -- (-6,2.2);

%\draw (2.5,0.35)node{} -- (2.5,2.55)node[fill=black]{};

%\draw (0.5,2.55)node{};
\draw (0.5,0.35)node[fill=black]{};
\draw (1.5,0.35)node{};
\draw (1.5,2.55)node[fill=black]{};

%\draw (-7.5,0.35) -- (-9.5,2.55);
\draw (0.5,0.35) to [out=140,in=320] (-8.5,2.55);
\draw (0.5,-0.3)node[nodelabel]{$b_1$};
\draw (1.5,-0.3)node[nodelabel]{$u_0$};
%\draw (2.5,-0.3)node[nodelabel]{$u_1$};

\draw (1.5,3.2)node[nodelabel]{$w_1$};
%\draw (2.5,3.2)node[nodelabel]{$w_1$};

\draw (-7.5,0.35)node[fill=black]{};
\draw (-7.5,-0.3)node[nodelabel]{$u_1$};
\draw (-7.5,3.2)node[nodelabel]{$w_0$};
\draw (-8.5,3.2)node[nodelabel]{$y$};
\draw (-8.5,-0.3)node[nodelabel]{$z$};
\draw (-8.75,1.45)node[nodelabel]{$e$};

\draw (-7.5,0.3) -- (-8.5,2.55);
\draw (-7.5,0.3) -- (-6.5,2.55);
\draw (-8.5,0.3) -- (-8.5,2.55)node[fill=black]{};
\draw (-8.5,0.3) -- (-7.5,2.55)node{};
\draw (-8.5,0.3)node[fill=black]{} -- (-6.5,2.55);

\draw (-6.5,2.55) -- (-4,2.55);
\draw (-5.25,2.9)node[nodelabel]{$\alpha$};
\draw (-4,2.55)node{};
\draw (-4,3.2)node[nodelabel]{$a_1$};
\draw (-6.5,2.55)node[fill=black]{};
\draw (-6.5,3.2)node[nodelabel]{$a_2$};

\end{tikzpicture}
\vspace*{-0.2in}
\caption{When $|S_1|=3$ and $|S_2|=2$}
\label{fig:rank-plus-index-proof-only-one-barrier-size-three}
%\bigskip
\end{figure}
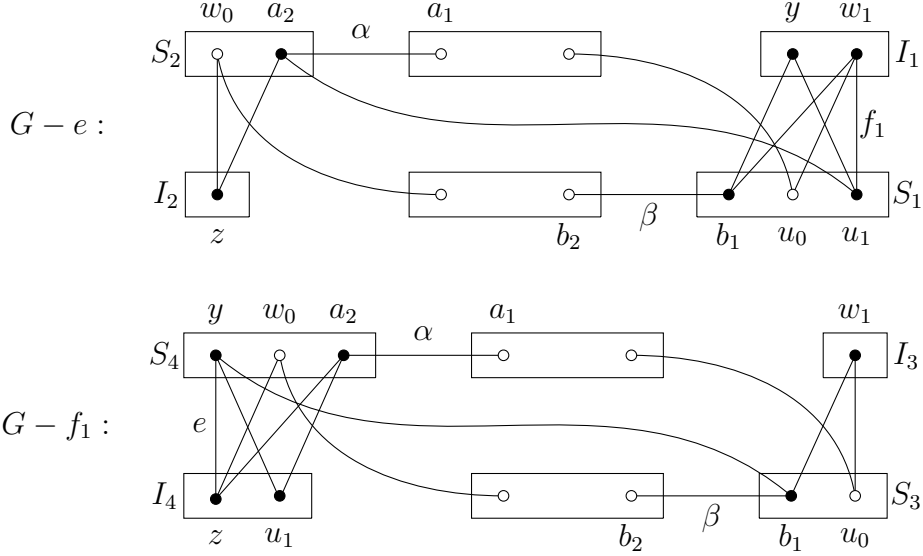

First of all, by Lemma~\ref{lem:non-removable-at-I}{\it (i)},
an end of the edge~$\alpha=a_1a_2$ lies in the barrier~$S_4$. Adjust notation
so that $a_2 \in S_4$. By statement {\it (ii)} of the same lemma,
$a_2$ has no neighbours in the set~$\overline{X_4}$
where $X_4:=S_4 \cup I_4$.
Consequently,
the neighbourhood of~$a_2$ is precisely $I_4 \cup \{a_1\} = \{$\vertexb$,u_1,a_1\}$.
Clearly, \vertexa~and $a_2$ are distinct vertices of~$S_4$, and we denote by~$w_0$
the remaining vertex of~$S_4$. Note that \mbox{$S_2 = \{w_0,a_2\}$}.

\smallskip
Next, we observe that if the vertices $u_1$ and $w_0$ are adjacent then
$Q:=$~\vertexb$w_0u_1a_2$\vertexb\ is a $4$-cycle of the bipartite
graph \bipartitebarrier{f_1}{S_4} and it contains the vertex~\vertexb~which has degree
three; by Corollary~\ref{cor:quadrilateral-LV}, one of the two
edges~\vertexb$w_0$ and~\vertexb$a_2$ is removable; however, this contradicts the fact
that $e=$~\vertexa\vertexb~is the only removable edge incident
with~\vertexb. Thus, the vertices $u_1$ and $w_0$ are nonadjacent.
It follows that $u_1$ is cubic, and its neighbourhood is precisely
$\{$\vertexa$,a_2,w_1\}$.

\smallskip
Observe that we have six cubic vertices whose neighbourhoods are fully
determined; these are:
the ends~\vertexa~and~\vertexb~of~$e$, the ends $u_1$ and $w_1$ of~$f_1$,
the end~$b_1$ of~$\beta$, and the end~$a_2$ of~$\alpha$.
There is a symmetry between the barrier structure of~$G-e$ and that of~$G-f_1$;
as is self-evident from
Figure~\ref{fig:rank-plus-index-proof-only-one-barrier-size-three}.
We have not determined the degrees of the two vertices $u_0$~and~$w_0$;
observe that if these vertices are not adjacent with each other
then $u_0$ has at least two neighbours in~$A-(S_2 \cup I_1)$ and
likewise, $w_0$ has at least two neighbours in $B-(S_1 \cup I_2)$;
whereas if $u_0w_0$ is an edge of~$G$ then
$u_0$ has at least one neighbour in~$A-(S_2 \cup I_1)$ and likewise,
$w_0$ has at least one neighbour in~$B-(S_1 \cup I_2)$.

\smallskip
As mentioned earlier,
we now proceed to prove that $g=$~\vertexa$u_1$ is an \Rthin\ edge. 
We let $J := ((G-e) / X_1 \rightarrow x_1) / X_2 \rightarrow x_2$
denote the unique brick of~$G-e$, where $X_1=S_1 \cup I_1$ and $X_2:=S_2 \cup I_2$.
Note that $J$ is near-bipartite with removable doubleton~$R$.

\begin{Claim}
\label{claim:g-Rthin}
The edge~$g=$~\vertexa$u_1$ is \Rthin.
(That is, $g$ is an \Rcomp\ edge of index two and its rank is $n-4$.)
\end{Claim}
\begin{proof}
Observe that $Q:=$~\vertexa$u_1w_1b_1$\vertexa~is a $4$-cycle in~$H = G-R$ which contains
the cubic vertex~\vertexa.
By Corollary~\ref{cor:quadrilateral-LV},
at least one of the edges $g=$~\vertexa$u_1$ and \vertexa$b_1$ is removable in~$H$.
Note that \vertexa$b_1$ is not removable,
whence $g$ is removable in~$H$.
To conclude that $g$ is \Rcomp,
it suffices to show that edges $\alpha$ and $\beta$ are admissible in~$G-g$.
We shall prove something more general, which is useful in establishing the thinness of~$g$ as well.

\smallskip
Observe that, in~$G-g$, the vertex~\vertexa~has neighbour set~$\{$\vertexb$,b_1\}$, and vertex~$u_1$ has neighbour set~$\{w_1,a_2\}$.
We will show that, if $v_1$ and $v_2$ are distinct vertices of the color class~$B$ such that $\{v_1,v_2\} \neq \{$\vertexb$,b_1\}$,
then $(G-g) - \{v_1,v_2\}$ has a perfect matching, say~$M$. This has two consequences worth noting. First of all, if $\{v_1, v_2\} = \{b_1,b_2\}$
then $M + \beta$ is a perfect matching of $G-g$ which contains $\alpha$ and $\beta$ both, whence $g$ is an \Rcomp\ edge of~$G$.
Secondly, it shows that $\{$\vertexb$,b_1\}$ is a maximal nontrivial barrier of~$G-g$.
An analogous argument establishes that $\{w_1,a_2\}$ is
also a maximal nontrivial barrier of~$G-g$, and consequently
Proposition~\ref{prop:equivalent-definition-Rthin} implies that $g$ is indeed \Rthin.

\smallskip
As mentioned above, suppose that $v_1$~and~$v_2$ are distinct vertices of~$B$ such that
$\{v_1,v_2\} \neq \{$\vertexb$,b_1\}$.
Let $N$ be a perfect matching of~$G-\{v_1,v_2\}$.
In what follows, we consider different possibilities, and in each of them, we exhibit a perfect matching~$M$ of $(G-g) -\{v_1,v_2\}$.
If $g \notin N$ then clearly $M:=N$. Now suppose that $g \in N$.
Note that, since $v_1,v_2 \in B$, the edge~$\alpha$ lies in~$N$
and $\beta$ does not lie in~$N$. If $b_1 \notin \{v_1,v_2\}$,
then the edge $b_1w_1$ lies in~$N$, and we let $M:= (N - g - b_1w_1) + f_1 + yb_1$.

\smallskip
Now consider the case in which $b_1 \in \{v_1,v_2\}$,
and adjust notation so that $b_1 = v_1$. Thus $v_2 \neq$~\vertexb,
whence~\vertexb$w_0 \in N$.
Also, $w_1u_0$ lies in~$N$.
Observe that $v_2$ lies in the set \mbox{$B-(S_1 \cup I_2)$}.
First, we consider the case when $u_0w_0$ is an edge of~$G$.
Observe that the six cycle \mbox{$C:=u_1$\vertexa\vertexb$w_0u_0w_1u_1$}
is $N$-alternating
and it contains the edge~$g$.
In this case, let $M$ denote the symmetric difference of $N$~and~$C$.

\smallskip
Finally, consider the situation in which $u_0w_0$ is not an edge of~$G$.
(In this case, to construct~$M$, we will not use the matching~$N$.)
As noted earlier, since $u_0$ and $w_0$ are nonadjacent,
$w_0$ has at least two distinct neighbours in the set~$B - (S_1 \cup I_2)$.
In particular, $w_0$ has at least one neighbour, say~$v'$,
which lies in~$B - (S_1 \cup I_2)$ and is distinct from~$v_2$.
Now, let $M_J$ be a perfect matching of~$J - \{v', v_2\}$.
Observe that $\alpha \in M_J$ and $\beta \notin M_J$.
Note that, in the matching~$M_J$, the
contraction vertex~$x_1$ is matched with some vertex in~$A- (S_2 \cup I_1)$,
which is a neighbour of~$u_0$ in the graph~$G$.
Now, we let $M:=M_J + w_0v' + f_1 + e$.

\smallskip
In every scenario, $M$ is a perfect matching of~$(G-g)-\{v_1,v_2\}$, as desired.
Thus, as discussed earlier, $g$ is \Rcomp\ as well as thin.
This proves Claim~\ref{claim:g-Rthin}.
\end{proof}

In summary, we have shown that $g=$~\vertexa$u_1$ is an \Rcomp\ edge
which satisfies both conditions {\it (i)} and {\it (ii)}, Theorem~\ref{thm:rank-plus-index}. This completes the proof.~\qed

%\newpage
\noindent
{\bf Acknowledgments}:
We are greatly indebted to both Joseph Cheriyan and U. S. R. Murty who have
helped us throughout this work by participating in all of
our discussions and making valuable suggestions.

\bibliographystyle{plain}
\bibliography{clm}

\end{document}